\documentclass[a4paper, 10pt]{article}  %
\usepackage[margin=1in]{geometry} 		%

\usepackage[utf8]{inputenc} %

\usepackage{amsmath, amsthm, amsfonts, amssymb, bbm, mathtools} %
\usepackage{bm,csquotes}
\usepackage{enumitem} %
\usepackage{titlesec} %
\usepackage{graphicx} %
\usepackage{caption} %
\usepackage{xcolor} %
\usepackage{animate,media9}	
\usepackage[T1]{fontenc} %
\usepackage{hyperref} %
\usepackage[capitalize]{cleveref} %

\theoremstyle	{plain}
\newtheorem		{theorem}		{Theorem}		[section]
\theoremstyle	{definition}
\newtheorem		{lemma}		[theorem] 		{Lemma}			
\newtheorem		{defn}   	[theorem] 		{Definition}
\newtheorem		{prop}		[theorem] 		{Proposition}	
\newtheorem		{corollary}	[theorem]		{Corollary} 	
\newtheorem 	{remark} 	[theorem]		{Remark}
\crefname{prop}{Proposition}{Propositions} %
\Crefname{prop}{Proposition}{Propositions} %

\titleformat*{\section}{\large\bfseries} 			 	%
\titleformat*{\subsection}{\normalsize\bfseries} 		%
\newcommand{ \indentForAssumptions }{ 1.5 cm }
\newcommand{\allNewAssumptions}{\ref{assumption increasing generator rates}--\ref{assumption intertwining} }
\newcommand{\allAssumptions}{\ref{assumption finite state spaces}--\ref{assumption intertwining} }
\newcommand{\AssumptionsOneFive}{\ref{assumption increasing generator rates}--\ref{assumption intertwining} }

\providecommand{\keywords}[1]
{\noindent
  {\small	
  \textbf{\textit{Keywords---}} #1}
} 

\DeclareMathOperator\Law{Law} 
\DeclareMathOperator\Leb{Leb} 
\DeclareMathOperator\linearSpan{span} 
\DeclareMathOperator\range{range} 
 
\DeclareMathOperator{\polylog}{Li}
\DeclareMathOperator{\sep}{sep}
\DeclareMathOperator{\cont}{cont}

\newcommand{ \mc }[1] 			{ \mathcal #1 }
\newcommand{ \mb }[1] 			{ \mathbf #1 }
\newcommand{ \mbb }[1] 			{ \mathbb #1 }
\newcommand{ \norm }[1] 		{ \Vert #1 \Vert }
\newcommand{ \floor }[1] 		{ \lfloor #1 \rfloor }
\newcommand{ \ceiling }[1] 		{ \lceil #1 \rceil }

\newcommand{ \ldef } 				{ \hspace{ 1 pt } \raisebox{ 0.4 pt }{:} \hspace{ -4 pt }= }

\newcommand{ \eps } 		{ \varepsilon }

\newcommand{ \indicator } 	{ \mathbbm 1 }

\newcommand{ \zeroVertex } 		{ \varnothing }
\newcommand{ \upDownDist } 		{ M }
\newcommand{ \upDownChain } 	{ X }
\newcommand{ \limitDist } 		{ \upDownDist }
\newcommand{ \limitProcess } 	{ F }
\newcommand{ \sepDist } 		{ \Delta }

\newcommand{ \downKernel } 				{ p^\downarrow }
\newcommand{ \upKernel } 				{ p^\uparrow }
\newcommand{ \upKernelell } 				{ p^{\uparrow,\ell} }
\newcommand{ \downKernelell } 				{ p^{\downarrow,\ell} }

\newcommand{ \upDownKernel } 			{ p }
\newcommand{ \intertwiningKernel } 		{ k }

\newcommand{ \commutationWeightProduct } 	{ \omega } 	%
\newcommand{ \coeffEigInDensityBasis } 		{ \eta }
\newcommand{ \coeffDensityInEigBasis } 		{ \eta^* }
\newcommand{ \generatorRates } 				{ c }
\newcommand{ \constantsInEstimate } 		{ q }

\newcommand{ \density } 				{ d } 	%
\newcommand{ \finiteFiltrations } 		{ V } 	%
\newcommand{ \eigenfunction } 			{ h }
\newcommand{ \core } 					{ \mc H }
\newcommand{ \spaceOfDensityFunctions } { H }

\newcommand{ \downOperator } 					{ D }
\newcommand{ \upOperator } 		 				{ U }
\newcommand{ \upDownOperator } 	 				{ T }
\newcommand{ \limitSemigroupOnGraph } 	 		{ Q }
\newcommand{ \projection } 						{ \pi }
\newcommand{ \intertwiningOperator } 			{ K } 	%
\newcommand{ \pregenerator } 					{ \mc A }
\newcommand{ \discreteGenerators } 				{ A }
\newcommand{ \limitSemigroupOnMetricSpace } 	{ \mc T }
\newcommand{ \discreteToContinuousMap} 			{ \Psi }

\newcommand{ \stateSpace } 		{ \mbb S }
\newcommand{ \limitSpace } 		{ E }
\newcommand{ \inclusion } 		{ \iota }

\newcommand{ \Sn } 		{ \mathfrak S }

\newcommand\si{\sigma}
\newcommand\la{\lambda}

\newcommand{\permutonSpace} 	{\mc P }

\DeclareMathOperator{\pat}{pat}
\DeclareMathOperator{\occ}{occ}

\newcommand{\murec}{\bm\mu^{\text{\tiny rec}}}   %

\newcommand{ \orange }[1]{{#1}} 		%
\begin{document}

\title{
	Up-down chains and scaling limits: \\
	application to permuton- and graphon-valued diffusions
}
\author{Valentin F\'eray$^1$$^*$ \and Kelvin Rivera-Lopez$^2$\thanks{This work was supported in part by the "Future Leader" Program of the LUE initiative (Lorraine Université d'Excellence) and by the ANR projects CORTIPOM (ANR-21-CE40-0019) and LOUCCOUM (ANR-24-CE40-7809).}}
\date{%
    $^1$Universit\'e de Lorraine, CNRS, IECL, F-54000 Nancy, France valentin.feray@univ-lorraine.fr\\%
    $^2$mathbykelvin@gmail.com %
}

\maketitle

\begin{abstract}

	An \emph{up-down chain} is a Markov chain in which each transition is a two-step process that moves up to a larger object and then back down to an object of the original size.
    The first goal of this paper is to present a general framework for analyzing these chains and computing their scaling limits.
    This approach unifies much of the existing literature while extending it in several directions.
    These include explicit conditions for constructing integrable up-down chains and convergence results for families of intertwined processes.
    The latter contribute to the \emph{method of intertwiners} of Borodin and Olshanski.
	The second goal is to highlight a notable application of this framework to the settings of permutations and graphs.
	Here, we identify some integrable up-down chains and construct their scaling limits, a family of permuton-~and graphon-valued Feller diffusions. 
   \orange{Both the up-down chains and the limiting diffusions} exhibit ergodicity, %
   diagonalizable semigroups, and explicit expressions for the maximal separation distance to stationarity.
	For the diffusions, the stationary measures are the \emph{recursive separable permutons} and \emph{recursive cographons} recently introduced by the authors, and the separation distances turn out to be related to the Dedekind eta function.
\end{abstract}


 \section{Introduction}

 In the next three sections (\cref{ssec:intro-up-down,ssec:intro-scaling,ssec:intro-permuton-graphon}), we present our main results.
 Then (in \cref{ssec:literature}), we discuss the related literature.

\subsection{Up-down chains}
\label{ssec:intro-up-down}

Up-down chains are Markov chains in which each transition can be decomposed into a growth step followed by a reduction step.
More precisely, suppose that the state spaces for these chains are given by $ \{ \stateSpace_n \}_{ n \ge 0 } $.
Then we require a collection of \emph{up-steps}, given by transition matrices $ \{ \upKernel_n \}_{ n \ge 0 } $ that move from $\stateSpace_n $ to $ \stateSpace_{n+1} $, and a collection of \emph{down-steps}, given by transition matrices $ \{ \downKernel_n \}_{ n \ge 1 } $ that move from $\stateSpace_n $ to $ \stateSpace_{n-1} $.
The associated up-down chains $ \{ \upDownChain_n \}_{ n \ge 0 } $ are then obtained by performing an up-step followed by a down-step.
That is, $ \upDownChain_n $ is the Markov chain on $ \stateSpace_n $ with transition matrix $\upDownKernel_n = \upKernel_n  \downKernel_{n+1}$.

In this paper, we will consider up-down chains with state spaces that satisfy the following condition:
\begin{enumerate}[ label = (A\arabic*), leftmargin = \indentForAssumptions ]
      \setcounter{enumi}{-1} %
	\item
	\label{assumption finite state spaces}
	$ \stateSpace_0, \stateSpace_1, \stateSpace_2, \ldots $ are finite and disjoint, and $ \stateSpace_0 $ consists of a single element, denoted by $ \zeroVertex$.
\end{enumerate}

\noindent
Moreover, we will impose a certain commutation relation on their up- and down-steps. 
Letting $i_n$ be the $|\stateSpace_n| \times |\stateSpace_n| $ identity matrix, this condition is as follows:
 \begin{enumerate}[ label = (C), leftmargin = \indentForAssumptions ]
      \setcounter{enumi}{-1} 
	\item 
	\label{assumption:commutation} 
      there exist constants $\beta_1, \beta_2, \ldots$ in $ ( 0, 1 ) $ such that
     $$
     \upKernel_n
	\downKernel_{ n + 1 }
		= 
        		\beta_n
    				\downKernel_n
    				\upKernel_{ n - 1 }
			+
    			\left( 
    				1 - \beta_n
    			\right)
    			i_n
		,
			\qquad
			n \ge 1
		.
	$$
 \end{enumerate}

In \cref{section discrete framework}, we explore the spectral and asymptotic properties of these chains.
Our approach is based on analyzing the transition operators\footnote{This operator is essentially the action of the transition matrix on a column vector; see \cref{ssec:def_kernels_operators} for details.} associated to $ \upKernel_n $, $ \downKernel_n $, and $ \upDownKernel_n $, which we denote by $ \upOperator_n $, $ \downOperator_n $, and $ \upDownOperator_n = \upOperator_n \downOperator_{n+1}$, respectively. 
Our first main result is \cref{thm:intro_discrete} below, proved in Proposition~\ref{prop eigenbasis}, Proposition~\ref{prop properties of stationary measures}, and \cref{cor ergodicity of chains}.
We use the notation $\generatorRates_{ - 1 } = 0 $, $\generatorRates_k = \beta_{1}^{-1} \cdots \beta_{k}^{-1}$ for $k \ge 0$, and
$\stateSpace=\cup_{n \ge 0} \stateSpace_n$.
Let also $I$ denote an identity operator (its domain should be deduced from context),
$ C( \stateSpace_n ) $ denote the space of real-valued functions on $\stateSpace_n$,
and 
for a function $f$ defined on $\stateSpace$, let $( f )_n$ denote its restriction to $\stateSpace_n$.
Finally, throughout the paper, \orange{we say that a Markov process} is ergodic if it \orange{converges to a unique stationary distribution for any initial state.}
 \begin{theorem}\label{thm:intro_discrete}
 Suppose that \orange{$ \{ \upDownChain_n \}_{ n \ge 0 } $ are up-down chains satisfying} Assumptions \ref{assumption finite state spaces} and \ref{assumption:commutation}.
 Then the following statements hold (using the above notation).
 \begin{enumerate}
 \item 
 {
 There exists an explicit family $(\eigenfunction_s)_{s \in \stateSpace}$ of functions on $\stateSpace$ (defined in (\ref{identity expansion of eigenfunctions}) below) such that}
\begin{align*}
 	\generatorRates_n
	( \upDownOperator_n - I )
    ( \eigenfunction_s )_n
    		& =
             	-\generatorRates_{k - 1 }
				( \eigenfunction_s )_n
			,
				\qquad
				\qquad
				\qquad
				\qquad
                s \in \stateSpace_k, 
				\,
				k \le n
			,
		\\
	C( \stateSpace_n )
		& =
        	\linearSpan\left\{
        		( \eigenfunction_s )_n : s \in \bigcup_{ k = 0 }^n \stateSpace_k 
        	\right\}
		,
			\quad
			\,\,\,\,
			n \ge 0
		,
		\\
	\dim 
		\linearSpan 	
    		\{
    			( \eigenfunction_s )_n
    		\}_{ s \in \stateSpace_k }
       			& =
        			\begin{cases}
           				| \stateSpace_k | - | \stateSpace_{ k - 1 } |,
            				&            
								\qquad
								\qquad
								\qquad
								1 \le k \le n,
            				\\
								1,
							&
								\qquad
								\qquad
								\qquad
								k = 0.
        			\end{cases} 
\end{align*}

	\item The up-down chains are {ergodic} and the unique stationary distribution of $ \upDownChain_n $ is given by 
	\[M_n(s) 
		= 
			(\upKernel_0 \cdots \upKernel_{n-1}) (\zeroVertex,s)
		,
			\qquad
			s \in \stateSpace_n
		.
	\]
			 \end{enumerate}
 \end{theorem}

Note that the first item above provides diagonal descriptions for the transition operators: it establishes that each $ \upDownOperator_n $ is diagonalizable, gives the dimensions of its eigenspaces, and gives a spanning set of eigenfunctions (there does not seem to be a canonical way to extract a basis from these).
These operators can also be described by their action on certain probabilistic functions, on which they act in a triangular manner (\cref{prop action of transition operators}). 

After our spectral analysis, we proceed to establish some results on the large time behavior of the up-down chains (\cref{prop density estimate,prop initial convergence}). 
These results are uniform in $n$ and allow us to consider some mixed large time-large size behavior. 
We then provide, under a small additional condition, an exact expression for the separation distance\footnote{The separation distance is a way to compare probability distributions, classically used to study mixing times of finite Markov chains; see \cref{subsection separation distance} for details.} between the distribution of the up-down chain at any time and its stationary measure $\upDownDist_n$ (for the worst initial distribution).
This result follows the work of Fulman \cite{fulmanCommutation} and appears as \cref{thm:separation_distance}.

As a final remark, we note that these results are actually proved in a more general setting.
This context is presented in \cref{section general setting} with the introduction of five new hypotheses \allNewAssumptions\unskip.
In short, we will go beyond up-down chains and focus on chains that can be \emph{intertwined} by running them in continuous time.
A definition of intertwining can be found in \eqref{definition intertwining} in \cref{ssec:literature}.

\subsection{Scaling limits}
\label{ssec:intro-scaling}
We continue our general framework by identifying a scaling limit for our chains.
Here we require that our discrete objects $\stateSpace=\bigcup_{n \ge 0} \stateSpace_n$ can be suitably mapped into a limiting space $ \limitSpace $ (see Assumptions \ref{assumption state space approximation}--\ref{assumption continuous density functions are limits} at the beginning of \cref{section convergence}).
This ensures that the `eigenfunction' $ \eigenfunction_s $ from \cref{thm:intro_discrete} has an extension\footnote{The restriction of $ \eigenfunction_s^o $ to $\stateSpace$ need not be $ \eigenfunction_s $, but it must approximate it; see Assumption \ref{assumption continuous density functions are limits} and Eq.~\eqref{general reformulation of convergence assumption}.} to $ \limitSpace $, which we denote by $ \eigenfunction_s^o $.
The following result is established in \cref{prop limiting semigroup and generator}, \cref{corol:ConvProcesses}, and \cref{prop:Stationary Measure of Limit Process}.

\begin{theorem}
\label{thm:intro-scaling-limit}
Let $\inclusion $ denote the map from $ \stateSpace $ to $ \limitSpace $.
Suppose that \orange{$ \{ \upDownChain_n \}_{ n \ge 0 } $ are up-down chains satisfying}
Assumptions \ref{assumption finite state spaces} and \ref{assumption:commutation}
(or more generally, \orange{Markov chains satisfying \ref{assumption finite state spaces}--\ref{assumption intertwining}}), and that \ref{assumption state space approximation}--\ref{assumption continuous density functions are limits} are satisfied.
Suppose also that $ \generatorRates_n \to \infty $ and
that the distributions of $ \inclusion( \upDownChain_n( 0 ) ) $ converge to $\mu$, a distribution on $ \limitSpace $.
Then the following statements hold.
 \begin{enumerate}
   \item
    There exists a Feller process $ \limitProcess $ in $ \limitSpace $ with initial distribution $ \mu $
    such that the convergence
	$$
		\big(
			\inclusion( \upDownChain_n ( \floor{ \generatorRates_n t } ) )
		\big)_{ t \ge 0 }
			\Longrightarrow
        		\big(
        				\limitProcess( t )
        		\big)_{ t \ge 0 }
	$$
	holds in distribution in the Skorokhod space $ D([0,\infty),\limitSpace)$.
	\item 
	The transition semigroup $ \{ \limitSemigroupOnMetricSpace(t) \}_{ t \ge 0 } $ of $ \limitProcess $ admits the diagonal description
	$$
		\limitSemigroupOnMetricSpace(t) h_s^o 
			= 
				e^{ -t \generatorRates_{ k-1 } } 
				h_s^o
			,
				\qquad
				s \in \stateSpace_k, 
				\,
				k \ge 0,
				\,
				t \ge 0
			.
	$$ 

  \item The process $ \limitProcess$ is ergodic with a unique stationary measure $M$,
	which is the weak limit of the stationary measures $M_n$ of $X_n$
	(or, more precisely, of their push-forwards on $ \limitSpace$).
 \end{enumerate}
\end{theorem}

\noindent
In the case of up-down chains, recall that the above time scaling factor $\generatorRates_n$ has a very simple expression
in terms of the constants appearing in \ref{assumption:commutation} -- namely, $c_n=\beta_1^{-1} \dots \beta_n^{-1}$.
In the general case, $ \generatorRates_n $ is encoded into \allNewAssumptions\unskip.

{After addressing the scaling limit, we go on to} establish a number of results concerning the limiting process, including several explicit descriptions of the generator (\cref{prop limiting semigroup and generator,prop algebraic identities in limit}),
a sufficient condition for path continuity (\cref{prop:path_continuity}),
an intertwining relation between $ \limitProcess $ and $ \upDownChain_n $ (\cref{prop algebraic identities in limit}),
other characterizations of the stationary distribution (\cref{prop:Stationary Measure of Limit Process}),
and large time estimates for {a rich class of observables} (\cref{prop densities of limit process}).

We conclude our framework by analyzing the maximal separation distance\footnote{In \cref{subsection separation distance}, we provide a natural extension of the separation distance to the continuous setting.} between the law of $\limitProcess(t)$ and its stationary distribution.
In particular, we show that it is a limit of discrete separation distances (\cref{theorem discrete and continuous sep dist}), whose value is given explicitly in an earlier result (\cref{prop convergence of sep dist}).
This improves upon an inequality that holds in general (\cref{thm general sep dist of a limit}) and reveals an interesting feature of intertwined processes.
As in the discrete setting, our analysis of this separation distance requires some additional conditions.
These are fairly natural and are satisfied by most examples.

\subsection{Permuton- and graphon-valued Feller processes}
\label{ssec:intro-permuton-graphon}
We demonstrate our general theory by studying two novel examples of up-down chains.
Let us briefly describe their construction.
For the state space $\stateSpace_n$, we will take the permutations or simple graphs of size $n$ (\orange{the {\em size} of a graph is its number of vertices}).
In the permutation case, it will be useful to think of the associated diagrams (the diagram of $\sigma$ in $\orange{\stateSpace_n}$ is the set of points $\{(i,\sigma(i)) : 1 \le i \le n\}$).

The down-steps in these chains delete a 
uniformly random point/vertex from a permutation/graph.
For permutations, this may involve adjusting the remaining points so that there is no empty row or column.
For graphs, this involves deleting all edges incident to the selected vertex.

The up-steps in these chains depend on a fixed parameter $p \in [0,1]$ and will duplicate a uniformly random point/vertex in a permutation/graph.
For permutations, this means replacing a point in the diagram by two points that
are consecutive in positions and values, and possibly adjusting the other points so that there is exactly one point in each row and column.
Here, the parameter $ p $ denotes the probability that the two new points are placed in an increasing position.
For graphs, duplicating a vertex involves replacing it by two new vertices with the same neighborhood (except possibly the vertices themselves).
Here, the parameter $ p $ denotes the probability that the two new vertices are connected.
Examples of these duplication operations are given in Figure~\ref{fig:ExamplesUpOperator}.
Formal definitions can be found in \cref{section permutation example,sec:graph}.

A simulation of our chains is given in Figure~\ref{fig:graph-permutation-processes}.
We remark that the permutation-valued chains are strongly related to the graph-valued chains.
Indeed, their up- and down-steps satisfy a commutation relation with the map taking a permutation to its inversion graph\footnote{If $\sigma$ is a permutation of size $ n $, its inversion graph has vertex set $\{1,\dots,n\}$ and an edge between $i$ and $j$ if and only if $\{i,j\}$ is an inversion of $\sigma$, i.e.~if and only if $(j-i)(\sigma(j)-\sigma(i))<0$.}.
See \cref{remark commutation with inversion graph} for details.
\medskip

\begin{figure}[t]
	\centering
	\includegraphics[height=35mm]{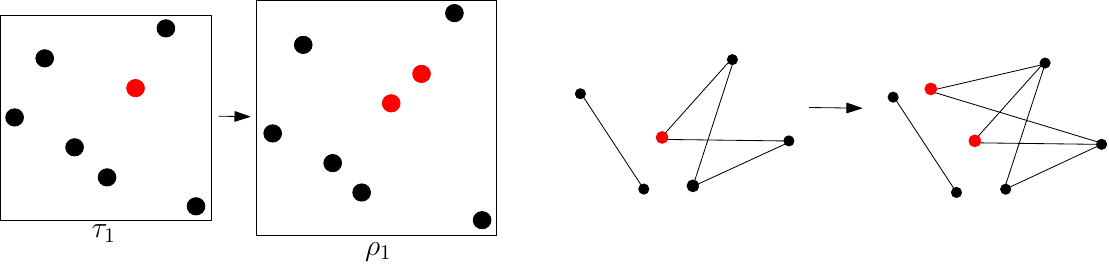}
	\caption{Examples of the duplication operations on permutations (left) and graphs (right).\label{fig:ExamplesUpOperator}}
	\end{figure}

In \cref{prop:commutation_permutations,prop:commutation_graphs},
we show that the above up-down chains satisfy condition \ref{assumption:commutation}.
Our analytic hypotheses \ref{assumption state space approximation}--\ref{assumption continuous density functions are limits} can also be seen to hold by taking the space of permutons/graphons\footnote{Permutons/graphons are limit objects for permutations/graphs that have been the subject of intense study in the last twenty years or so; see \cref{ssec:prel_permutations,ssec:graphons} for basic definitions and references.} as the limiting space for the permutations/graphs. 
As a result, our general theory applies and all of the results discussed earlier hold for these chains.
We wish to highlight a few of these results since they take particularly nice forms in these concrete settings.
We begin with the scaling limit, which now includes the path-continuity of the limiting process
and the identification of its stationary distribution.
In the following, we use the term {\em Feller diffusion}
for a Feller process with a.s.~continuous paths.
\begin{theorem}\label{thm:intro-scaling-limit-examples}  
	Let
	$\stateSpace$ be the set of permutations (resp.~graphs),
	$\stateSpace_n$ be the subset of objects of size $n$,
	and $ \limitSpace $ be the space of permutons (resp.~graphons).
	Let $ \upDownChain_n $ denote the up-down chain defined above on $ \stateSpace_n $ and $ \inclusion $ denote the map from $ \stateSpace $ to $ \limitSpace $.
	Suppose that the initial distributions of these chains converge to $ \mu $, a distribution on $ \limitSpace $.
  Then there exists a Feller diffusion $F$ in $ \limitSpace $ with initial distribution $ \mu $ such that
  	$$
		\big(
			\inclusion( \upDownChain_n ( \floor{ n^2 t } ) )
		\big)_{ t \ge 0 }
			\Longrightarrow
        		\big(
        				\limitProcess( t )
        		\big)_{ t \ge 0 }
	$$
    in distribution in the Skorokhod space $D([0,\infty),E)$.
    Moreover, this process is ergodic and its unique stationary distribution
     is the law of the recursive separable permuton (resp.~recursive cographon)
    of parameter $p$ introduced in \cite{vfkrPermuton}.
\end{theorem}
\begin{figure}[h!]
     \centering
       \makeatletter\edef\animcnt{\the\@anim@num}\makeatother
    \animategraphics[label=myAnimPerm,scale=0.4,type=pdf]{20}{perm_step}{0}{30} 
       \mediabutton[jsaction={anim.myAnimPerm.playFwd();}]{\scalebox{1.5}[1.2]{\strut $\vartriangleright$}}
       \mediabutton[jsaction={anim.myAnimPerm.pause();}]{\scalebox{1.5}[1.2]{\strut $\shortparallel$}}
    \qquad
        \makeatletter\edef\animcnt{\the\@anim@num}\makeatother
    \animategraphics[label=myAnimGraph,scale=0.4,type=pdf]{20}{graph_step}{0}{30} 
       \mediabutton[jsaction={anim.myAnimGraph.playFwd();}]{\scalebox{1.5}[1.2]{\strut $\vartriangleright$}}
       \mediabutton[jsaction={anim.myAnimGraph.pause();}]{\scalebox{1.5}[1.2]{\strut $\shortparallel$}}
        \caption{Simulation of an up-down chain on permutations (left) and graphs (right). In each case, we take $p=1/2$, $n=100$, and the initial distribution is a uniform random permutation, resp.~graph, of size $n$.
        Each movie is a succession of thirty-one pictures that show the state of the chain after $m$ steps,
         where $m \in  \{0,\dots,30\} \cdot 50$. 
    We plot permutations as diagrams (with a blue dot at coordinates $(i,\sigma(i))$ for each $i\le n$), 
    and graphs as pixel pictures, or adjacency matrices (with a black dot at coordinates $(i,j)$ and $(j,i)$ for each edge $\{i,j\}$ of the graph with an appropriate labeling).
    Animations do not work properly with all pdf viewers, but they seem to work with Acrobat Reader or Okular. \label{fig:graph-permutation-processes}}
\end{figure}
In addition to this convergence, in this setting we are also able to approximate $ \limitProcess $ by some Markov chains constructed directly in the limiting space.
In the case of permutons, this \enquote{semi-discrete approximation} is presented in \cref{ssec:semi-discrete-permutations}; the adaptation to graphons is straightforward.
It would be interesting to go beyond this and construct $ \limitProcess $ directly in the limiting space, but we do not know how to do this.

Next, we turn our attention to several results that now feature a combinatorial class of observables central %
to the theory of permutons/graphons: the pattern/subgraph density functions.
These results are \cref{prop algebraic identities in limit}\ref{claim generator on density functions},
equations \eqref{identity expansion of eigenfunctions} and \eqref{eq:ho_In_Do_Basis},
and \cref{prop density estimate,prop densities of limit process}, and they reveal that these functions triangularize the generator of $ \limitProcess $,
can be used to describe the eigenfunctions of all of the semigroups,
and admit simple large time asymptotics under the discrete chains and the diffusion $ \limitProcess $.
The exact statement of these asymptotics can be found in \cref{section asymptotics of up-down pattern densities,ssec:asymp_pattern_densities_diffusion}.

Finally, we consider \cref{thm:separation_distance,theorem discrete and continuous sep dist}, which now yield simple formulas for the separation distances and reveal a connection to the Dedekind eta function, defined by 
\[\eta(\tau)\ldef e^{\frac{\pi\, i\, \tau}{12}} \prod_{j=1}^{\infty}(1-e^{2j \pi\, i \, \tau})\]
on complex numbers $\tau$ with positive imaginary part.
These formulas are presented in the following result, which is proved in \cref{ssec:separation_permutations,section sep dist of permuton diff}.
\begin{prop}
\label{prop:separation-intro-permutations}

	Let $\stateSpace_n$ be the set of permutations or graphs of size $n$
	and $ \upDownChain_n $ the up-down chain defined above on $ \stateSpace_n $.
	Let $\sepDist_n(m)$ denote the separation distance between $ \upDownChain_n( m ) $ and the stationary distribution $ \upDownDist_n $ (for the worst initial distribution).
	Then we have the identity
  \begin{equation}
  \label{eq:separation-intro-permutations}
  \sepDist_n(m)
  	= 
		\sum_{j=1}^{n-1} 
			(-1)^{j-1} 
			(2j+1) 
			\frac{(n-1)!n!}{(n-1-j)!(n+j)!} 
			\bigg(1-\frac{j(j+1)}{n(n+1)}\bigg)^m
  	,
		\qquad
		n \ge 2,
		\,
		m \ge 0
	.
  \end{equation}

    \noindent
    Let now $ \sepDist_\limitProcess $ be the separation distance associated with the limiting diffusion of the $ \upDownChain_n $.
    Then \orange{$\sepDist_\limitProcess$ is given by the following monotonic limit and series}:
     \begin{equation}
      \label{eq:separation-intro-asymptotics}
       \sepDist_n(\lfloor t n ( n + 1 ) \rfloor)
    		\underset{ n \to \infty}{\orange{\nearrow}}
    	\sepDist_\limitProcess( t )
    		=
    			\sum_{ j = 1 }^\infty 
					( -1 )^{ j - 1 } ( 2j + 1 ) 
    				e^{ -t j ( j + 1 ) }
    		,
				\qquad
				t > 0
			.
         \end{equation}
    
    \noindent
    Moreover, $ \sepDist_\limitProcess $ exhibits the following properties:
    \begin{enumerate}[ label = (\roman*)]
    
    \item
    \label{item product form of sep dist}
    $
    	\sepDist_\limitProcess(t)
    		=
    			1
    			-
    			\prod_{ j = 1 }^\infty
    				( 1 - e^{ - 2 j t } )^3
    		=
    			1
    			-
    			e^{ t/4} 
				\eta^3( i t/\pi )
    $
    for $ t > 0 $,
  \item \label{item:symmetry_DeltaF}
	$
		1 - \sepDist_\limitProcess( t )
			=	
				\exp\big( -\frac{ \pi^2 }{ 4 t } + \frac{ t }{ 4 } \big)
				\left(
					\frac{ \pi }{ t }
				\right)^{ 3/2 }
				\big( 1 - \sepDist_\limitProcess\big( \frac{ \pi^2 }{ t } \big) \big)
	$
	for $ t > 0 $,

  \item \label{item:asymp_DeltaF_Infty}
    $ \sepDist_\limitProcess( t ) \sim 3 e^{ - 2 t } $ as $ t \to \infty $,
    
    \item \label{item:asymp_DeltaF_Zero}
	$
		1 - \sepDist_\limitProcess( t )
			\sim	
				e^{ -\frac{ \pi^2 }{ 4 t } }
				\left(
					\frac{ \pi }{ t }
				\right)^{ 3/2 }
	$
    as $ t \to 0 $,
    and

  \item \label{item:DeltaF_CInfty}
      $\sepDist_\limitProcess$ is in $C^\infty[0, \infty )$ and its successive derivatives
    satisfy
    $ 
    	\orange{
		\sepDist_\limitProcess^{(k)}( t ) = 0 
		}
	$ for $ k \ge 1 $.
    \end{enumerate}
\end{prop}

\noindent
We note that $\sepDist_\limitProcess$ is not analytic at $0$
since all of its successive derivatives vanish, but the function is not constant.

It follows from this result that the up-down chains do not exhibit a separation cutoff\footnote{A sequence of Markov chains exhibits a separation cutoff if there exists $t_n$ such that for any $\eps>0$, $\sepDist_n((1-\eps)t_n)$ tends to $1$ while $\sepDist_n((1+\eps)t_n)$ tends to $0$.
Informally, this means that the limit of the separation distance should be, after an appropriate time renormalization, of the form $\bm{1}[t <t_0]$,
which is not the case here.}.
A numerical plot of $ \sepDist_\limitProcess $ is shown in Figure~\ref{fig:limit_separation_distance}.
The flat aspect of the curve near $ t = 0 $ is consistent with the estimate $\sepDist_\limitProcess( t )=1-e^{-\Theta(1/t)}$ above.
Informally, our results indicate that the chain mixes at time scale $\Theta(n^2)$, but the mixing starts very slowly at this scale.
\begin{figure}
  \begin{center}
    \includegraphics[scale=0.7]{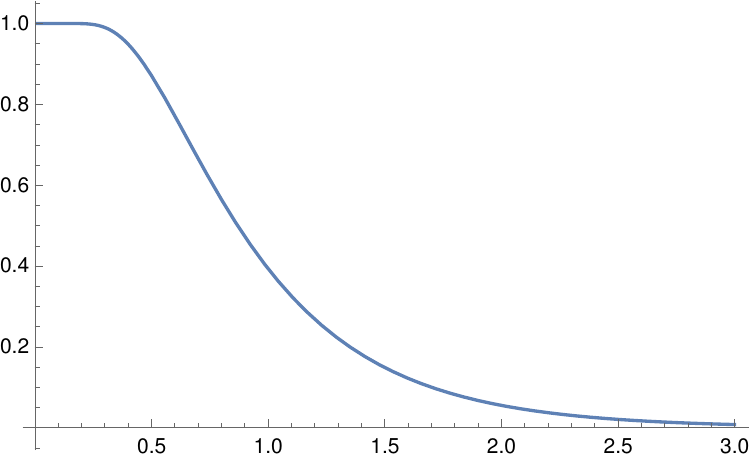}
  \end{center}
  \caption{{The separation distance of the permuton- and graphon-valued diffusions.}}
  \label{fig:limit_separation_distance}
\end{figure}

\subsection{Related literature}
\label{ssec:literature}

\textbf{Intertwined processes.}
A family of Markov processes $ \{ Y_n \}_{ n \ge 0 } $ on state spaces $ \{ E_n \}_{ n \ge 0 } $ is \emph{intertwined} if the transition semigroups of these processes $ \{ P_n(t) \}_{ n \ge 0 } $ satisfy a certain commutation relation together with some transition operators 
$\{ \Lambda_n \}_{ n \ge 1 }$ between the state spaces.
Namely, the operator $\Lambda_n$ should be associated with a transition
 from $E_n$ to $E_{ n - 1 }$ and 
\begin{equation}
	\label{definition intertwining}
	P_n( t )
	\Lambda_n
		=
			\Lambda_n
			P_{ n - 1 }( t )
		,
			\qquad
			t \ge 0,
			\,
			n \ge 1
		.
\end{equation}

\noindent
Borodin and Olshanki showed that under certain conditions, these processes give rise to a process on a natural limiting space that is also intertwined with each $ Y_n $ \cite[Proposition 2.4]{BOintertwiners}.
It seems reasonable to expect this process to be some limit of the $ \{ Y_n \} $, but this is still unknown, particularly because the construction of this process on the limiting space is abstract.

The Markov chains considered in this paper yield families of intertwined processes when run in continuous time with appropriate jump rates (see Assumption \ref{assumption intertwining}).
Therefore, Borodin and Olshanski's result applies, providing an existence version of our \cref{prop algebraic identities in limit}\ref{claim semigroup intertwining}.
We go beyond this by offering an explicit and complete description of the generator 
and by establishing the desired convergence\footnote{Since the Markov chains converge, their continuous-time variants do as well; see \cref{theorem ethier kurtz modes of convergence}.}.
\orange{It would be interesting to extend these results to the more general setting of \cite{BOintertwiners}.}
\medskip

\noindent
\textbf{Up-down chains.}
A number of up-down chains have been previously considered in the literature.
Various examples on integer partitions were studied in \cite{fulmanCommutation, BOpartitions, petrovTwoParameter, petrov2010strict, OlshAddJackParameter} and an example on integer compositions was studied in \cite{krdrDiffusions,krdrLeftmost}.
These chains have connections to well-known objects, such as $z$-measures, Schur functions, Jack polynomials, the Chinese restaurant process, and the infinitely-many-neutral-alleles diffusion model of Ethier and Kurtz.
A general class of up-down chains was considered in \cite{petrovSL2}. 
This class unifies many of the examples on partitions.
In a different body of literature, an up-down chain on trees
was introduced by Aldous~\cite{aldous2000mixing_cladograms} in connection with the Brownian Continuum Random Tree
and further studied in \cite{schweinsberg2002cladograms,lohr2020Aldous_chain,forman2023aldousdiffusion}.

One motivation for considering these chains is that given a {consistent family of distributions}, up-down dynamics often provide simple local dynamics that are mathematically tractable and have those distributions as stationary distributions.
One can then use them to construct exchangeable pairs for applying Stein's method \cite{fulman2005Stein_Plancherel,fulman2005Stein_Jack,dolega2016Jack_Plancherel},
or to construct interesting processes on some limit space as done in \cite{BOpartitions,petrovTwoParameter,petrov2010strict,OlshAddJackParameter,krdrDiffusions,lohr2020Aldous_chain} and in the present paper.
\medskip

In \cite{fulmanCommutation}, Fulman considers down-up chains\footnote{For down-up chains, down-steps are performed before up-steps. These are \orange{often} comparable to the associated up-down chains; see e.g.~\cite[Proposition 5.7]{fulmanCommutation}. \orange{In fact}, under Assumption~\ref{assumption:commutation}, the up-down chain is simply a lazy version of the down-up chain.}
that satisfy a general form of our commutation relation~\ref{assumption:commutation} (the matrices need not be transition matrices and the constants need not sum to $1$).
Fulman presents a methodology for computing the separation distance of several down-up chains and then uses ad~hoc arguments to obtain asymptotic estimates.
We follow this methodology to establish our first general formula in \cref{thm:separation_distance}, but we also extend the method to treat the continuous-time variants of the chains and the limiting processes.
Moreover, we handle the asymptotics of the separation distance in a general manner (\cref{prop convergence of sep dist}) and exploit the intertwining structure to establish a new monotonicity result (\cref{prop recursive structure in sep dist}).
Fulman also obtains a complete description of the spectrum of the transition operators (the eigenvalues and their multiplicities) but does not identify the eigenfunctions or the stationary distributions
(at the same time, the stationary distributions are known in those examples).
Our Theorem~\ref{thm:intro_discrete} can be seen as a more complete version of these results but 
in a slightly more restrictive context.
Finally, let us mention that Fulman does not consider at all the problem of scaling limits.
\medskip

On the other hand, each of the papers \cite{BOpartitions,petrovTwoParameter,petrov2010strict,OlshAddJackParameter, krdrDiffusions} is primarily interested in computing the scaling limit of a specific family of up-down chains.
These chains all satisfy our condition \ref{assumption:commutation} (see \cref{sec:previous_examples}), but this fact is not used in their analysis.
Instead, the authors rely on an auxiliary algebra that is in correspondence with functions on the combinatorial state spaces to show that the transition operators act in a triangular way on a certain family of functions
-- the Schur symmetric functions in~\cite{BOpartitions},
variants of the monomial symmetric functions in~\cite{petrovTwoParameter},
shifted variants of the complete homogeneous symmetric functions in~\cite{OlshAddJackParameter},
and variants of the monomial quasisymmetric functions in~\cite{krdrDiffusions}.
The passage to the limit is then mostly algebraic, since these triangular operators can be interpreted as `projections' of a single operator acting in the auxiliary algebra.

The results discussed in \cref{ssec:intro-up-down,ssec:intro-scaling} provide a unified theory for these papers.
Indeed, we recover essentially all of their main results with our general framework, with the notable exception of the differential form of some of the generators.
\orange{Moreover, our approach unifies the common methodology of these papers, including the identification of a triangular `basis' and an algebraic passage to the limit.}
Unlike those papers though, we do not work in an auxiliary algebra -- all of our computations are driven by \orange{algebraic properties of our transition operators, like} the commutation relation \ref{assumption:commutation}.
Our work also goes beyond the above papers: this includes the identification of eigenfunctions, the analysis of the separation distance, and the intertwining with the limiting process.
Finally, we note that the triangular `basis' we construct often has a natural interpretation in terms of substructure densities.
This reduces our Assumptions~\ref{assumption state space approximation}--\ref{assumption continuous density functions are limits} to standard facts of the corresponding limit theories.
This will be the case in our novel examples, the settings of permutons and graphons, but is not the case for the above papers.
For those examples, the verification of these hypotheses is nontrivial.
See our discussion in \cref{sec:previous_examples}.
\medskip

In \cite{petrovSL2}, Petrov introduces a class of up-down chains that unifies those in \cite{BOpartitions, petrovTwoParameter, petrov2010strict, OlshAddJackParameter}.
Petrov shows that this class admits triangular descriptions for its transition operators and identifies their spectra, as we do in \cref{prop action of transition operators} and \cref{thm:intro_discrete}.
However, the class of chains we consider is more general and more accessible probabilistically.
Indeed, our conditions \ref{assumption finite state spaces} and \ref{assumption:commutation} involve transition matrices and place little restriction on the state spaces, while the class in \cite{petrovSL2} is defined by more abstract algebraic structures.
The latter also requires the state spaces to \orange{consist} of ideals of some underlying poset and the dynamics to be reversible.
This reversibility requirement excludes, for example, the composition chains in \cite{krdrDiffusions} and the permutation and graph chains introduced in \cref{ssec:intro-permuton-graphon}.
In addition, our results go beyond \cite{petrovSL2}, which does not consider the eigenfunctions, separation distances, asymptotics, or scaling limit.
\medskip

\orange{Finally in \cite{aldous2000mixing_cladograms}, Aldous introduced a Markov chains on a family of trees (namely unrooted
non-plane trees, which he refers to as {\em cladograms}) and studied his mixing time. He also conjectured
the existence of a scaling limit, which was established with respect to different limiting spaces and topologies in 
\cite{lohr2020Aldous_chain} and \cite{forman2023aldousdiffusion}. }
\orange{
This chain is a labelled version of an up-down chain satisfying Assumptions \ref{assumption finite state spaces} 
and \ref{assumption:commutation}. 
Our results recover several results from the literature,
including diagonal and triangular formulas for the generators from \cite{gambelin:tel-05202989}
and the existence of a scaling limit in \cite{lohr2020Aldous_chain}
(we need however the construction of the limiting space provided in this paper).
More detail on this example is given in \cref{ssec:cladograms}.}
\bigskip                                                                        

\noindent
\textbf{Permutation and graph dynamics.}
To our knowledge, the permutation- and graph-valued 
up-down chains we consider have not yet appeared in the literature.
However, the \orange{deterministic} duplication operations are standard: 
the permutations/graphs that can be obtained from the permutation/graph of size $1$
by these operations are known as separable permutations/cographs 
and are well-studied objects
(see the references given in \cite{bassino2018BrownianSeparable}
and \cite{bassino2022cographs} respectively).
Thus, using these duplication operations to build 
 up-down chains on permutations and graphs seems natural.
 
 In another direction, let us point out that
there is a large existing literature on graph dynamics and graphons,
but it seems that simple dynamics on finite graphs leading to nontrivial dynamics on the space
of graphons are uncommon. 
Indeed, due to an averaging effect, the limiting dynamic is often deterministic, 
and can sometimes be analyzed through the {\em differential equation method};
see the seminal paper of Wormald~\cite{wormald1995differential} for general principles 
or the recent paper~\cite{garbe2023flip}
for a large class of random graph evolutions (defined via local edge replacements) 
that lead to deterministic dynamics on graphons in the limit.
In contrast, our model, consisting of local vertex replacements,
still has a random behavior at the level of graphons.
It would be interesting to study 
 a mix of vertex and edge replacements
  and to understand the phase transition between deterministic and random dynamics
  in the limit, but this is out of the scope of this article.

\begin{remark}
  While finalizing this article, we learned about some parallel work by Roman Gambelin \cite{gambelin2025scaling}
  on scaling limits of up-down chains.
  His results and ours have some intersection, in particular regarding the existence of scaling limits under assumption \ref{assumption:commutation}.
  In addition, he proves that, for any up-down chain satisfying  \ref{assumption finite state spaces} and \ref{assumption:commutation},
  there always exists a limiting space $E$ satisfying 
  the hypotheses of \cref{prop limiting semigroup and generator,prop algebraic identities in limit},
  and that this space is unique in some sense, see \cite{gambelin2025scaling} for details.
  On the other hand, he does not consider the computation of the separation distance, or the particular chains on permutations
  and graphs considered in this paper.
\end{remark}

\subsection{Outline}
In Section \ref{section background on feller processes}, we provide background on the theory of Feller processes, recall the notion of separation distance, and give some new results on the latter. 
In Section \ref{section discrete framework}, we analyze our discrete chains. 
In Section \ref{section convergence}, we compute the scaling limit of these chains.
In Section \ref{sec:previous_examples}, we review up-down chains from the literature and discuss the consequences of our general results.
In Sections \ref{section permutation example} and \ref{sec:graph}, we study the novel up-down chains on permutations and graphs.

\subsection{Notation}
An empty product or sum will be regarded as a one or zero, respectively.
A countable set $ E $ will always be equipped with the discrete topology.
The space of continuous functions from $ E $ to $ \mbb R $ will be denoted by $ C( E ) $ and \orange{its subset of positive functions by $ C_+( E )$.}
We note that $ C( E ) $ is a Banach space when $ E $ is compact.
For a measurable space $ E $, the Banach space of measurable bounded functions from $ E $ to $ \mbb R $ will be denoted by $ \mc M_b( E ) $.
Both of these spaces are to be equipped with the supremum norm.
The norm on a Banach space $ B $ will be denoted by $ \Vert \cdot \Vert_B $.
The indicator function of a set $ C $ will be denoted by $ \indicator_C $ and identity \orange{operators} will be denoted by $ I $ -- the domains of these objects should be deduced from context.
Finally, as already done in \cref{thm:intro-scaling-limit,thm:intro-scaling-limit-examples}, 
we use a double arrow $\Rightarrow$ for the convergence in distribution of random variables.

\section{Background on Feller Processes}
\label{section background on feller processes}

This section provides the necessary theory of Feller processes that we will use.
For a simpler presentation, we have reformulated the results we reference into our specific context.
{The discussion on separation distance contains new results, which we believe are of independent interest.}

\subsection{Kernels and transition operators}
\label{ssec:def_kernels_operators}
Let $ ( E, \mc E ) $ and $ ( F, \mc F ) $ be measurable spaces.
A \emph{probability kernel from $ E $ to $ F $} is a function $ \mu \colon E \times \mc F \to [ 0, \infty ) $ such that
\begin{enumerate}[ label = (\roman*) ]
	\item
	for every $ B \in \mc F $, the map $ \mu( \, \cdot \,, B ) $ is measurable, and
	
	\item
	for every $ x \in E $, the map $ \mu( x, \, \cdot \, ) $ is a probability measure on $ F $.
	
\end{enumerate}

\noindent
We will refer to these objects simply as kernels.
Kernels can be viewed as generalized transition matrices: given $ x \in E $, we can transition to $ F $ by sampling an object according to $ \mu( x, \, \cdot \, ) $.
In fact, when the spaces $ E $ and $ F $ are discrete, we can regard a kernel as an $ E \times F $ matrix whose entries are given by $ \mu( x, y ) \ldef \mu( x, \{ y \} ) $.

The operations of transition matrices can be generalized to operations involving kernels.
For example, a kernel $ \mu $ from $ E $ to $ F $ can act on a probability measure $ \lambda $ on $ E $ to obtain a probability measure on $ F $. This measure is given by
$$
	( \lambda \mu ) ( B )
		= 
			\int_E
				\lambda( dx )
				\mu( x, B )
		,
			\qquad
			B \in \mc F
		.
$$

\noindent
Similarly, we can take the product of $ \mu $ and a kernel $ \nu $ from $ F $ to $  G$ to obtain a kernel from $ E $ to $ G $.
This kernel is given by
$$
	( \mu \nu ) ( x, C )
		= 
			\int_Y
				\mu( x, dy )
				\nu( y, C )
		,
			\qquad
			x \in E,
			\,
			C \in \mc G
		.
$$

\noindent
The kernel $ \mu $ can also act on functions, similar to how a transition matrix can act on column vectors.
This action takes the form of a linear operator $ T_\mu \colon \mc M_b( F ) \to \mc M_b( E ) $ defined by
$$
	( T_\mu f )( x )
		=
			\int_F
				f( y ) 
				\mu( x, dy )
        = 
        	\mbb E_x[f(Y)],
$$
where, in the last expression, the random variable $Y$ has distribution $\mu(x,\cdot)$.
This operator is the \emph{transition operator} associated with $ \mu $, and it can be viewed as a dual object to $ \mu $.
The kernel $ \mu $ can be recovered from its transition operator $ T_\mu $ by the formula
\begin{equation}
	\label{identity kernel from operator}
	\mu( x, B )
		=
			( T_\mu \indicator_B )( x )
		.
\end{equation}
The properties of $ \mu $ make $ T_\mu $ a positive, contractive, and conservative operator. 
Recall that an operator $V$ is \emph{positive} if $ V f \ge 0 $ whenever $ f \ge 0 $, \emph{contractive} if its operator norm is at most 1, and \emph{conservative} if $ V \, 1 = 1 $.	
\noindent
Finally, we note that transition operators are often restricted to subspaces of their domain.
For example, whenever $F$ is a compact topological space, it is common to restrict $T_\mu$ to the space $C(F)$ 
of continuous functions on $F$.

\subsection{Markov processes, Feller semigroups, and generators}
Let $ ( E, \mc E ) $ be a measurable space and let $ Y $ be a continuous-time time-homogeneous Markov process in $ E $. 
Recall that $ Y $ is often described by a family of kernels $ \{ \mu_t \}_{ t \ge 0 } $ that specifies the conditional distributions
$$
	 \mu_t( Y( s ), B )
	 	=
			\mbb P( Y( s + t ) \in B \, | \, Y( s ) )
		,
			\qquad
			s, t \ge 0,
			\,
			B \in \mc E
		.
$$

\noindent
An alternative description of $ Y $ is given by the associated transition operators $ \{ T(t) \}_{ t \ge 0 } $, which specify the conditional expectations
\begin{align*}
	( T(t) f )( x )
		= 
			\int_E
				f( z )
				\mu_t( x, dz )
		=
			\mbb E_x[ f( Y( t ) ) ]
		,
			\qquad
			f \in \mc M_b( E ), \,
			x \in E, \,
			t \ge 0
		.
\end{align*}

\noindent
\orange{The family $ \{ T(t) \}_{ t \ge 0 } $ is called} the \emph{transition semigroup} of $ Y $.
This language indicates that these operators are associated with a Markov process and form an \emph{operator semigroup}: they are bounded, defined on a Banach space, start from $ T( 0 ) = I $, and satisfy the semigroup identity\footnote{This relation can be seen as dual to the Chapman-Kolmogorov equation.}
$$
	T( t + s )
		=
			T( t ) T( s )
		,
			\qquad
			s, t \ge 0
		.
$$

Conversely, the theory of Feller processes provides conditions for an operator semigroup to 
be the transition semigroup of a Markov process exhibiting some regularity.
\orange{For simplicity}, let $ E $ be a compact metric space.
A \emph{Feller semigroup} on $ C( E ) $ is an operator semigroup $ \{ T( t ) \}_{ t \ge 0 } $ of positive, contractive, conservative\footnote{This condition is sometimes left out.} operators that satisfy the following regularity conditions, known as the \emph{Feller properties}:
\begin{enumerate}[ label = (\roman*) ]
	
	\item
	$ T( t ) C( E ) \subseteq C( E ) $ for $ t \ge 0 $, and
	
	\item
	$ ( T( t ) f ) ( x ) \to f( x ) $ as $ t \to 0 $ for $ f \in C( E ) $ and $ x \in E $.

\end{enumerate}

\begin{theorem}[Chapter 4 Theorem 2.7 in \cite{ethierKurtzBook}]
	Let $ \{ T( t ) \}_{ t \ge 0 } $ be a Feller semigroup on $ C( E ) $.
	For every Borel probability measure $ \nu $ on $ E $, there exists a Markov process with initial distribution $ \nu $, transition semigroup $ \{ T( t ) \}_{ t \ge 0 } $, and sample paths in the Skorokhod space $ D( [ 0, \infty ), E )  $.

\end{theorem}

\orange{A} Markov process associated with a Feller semigroup is called a \emph{Feller process}.
Interestingly, this process can be described by a single operator, called the \emph{generator} of its semigroup.
The generator of a Feller semigroup $ \{ T( t ) \}_{ t \ge 0 } $ on $ C( E ) $ is the operator defined by the limit
$$
	A f
		=	
			\lim_{ t \to 0 }
				\frac{ T( t ) f - f }{ t }
		,
$$

\noindent
wherever it exists.
This domain is usually not $ C( E ) $ but is always a dense subspace of $ C( E ) $ (see \cite[Chapter 17]{kallenbergTheBible}).
Moreover, the constant function $ 1 $ lies in the domain and $ A \, 1 = 0 $ since $ T( t ) 1 = 1 $.

In general, handling a generator on its full domain can be challenging.
As a result, a generator is typically only specified on a suitable subspace of its domain, chosen so that the full generator can be obtained from this restriction.
The choice of such a subspace is informed by the fact that generators are \emph{closed}\footnote{On the other hand, generators are usually unbounded.}
operators -- that is, if $ A $ is a generator with domain $ \mc D $ then its \emph{graph} 
$
	\{ 
		( f, A f )
	:
		f \in \mc D
	\}
$
will be closed in $ C( E ) \times C( E ) $ under the product topology.
A representative subspace can therefore be found in a \emph{core}, a subspace $ D \subset \mc D $ for which the graph of $ A\vert_D $ is dense in the graph of $ A $.

\begin{prop}[Proposition 19.9 in \cite{kallenbergTheBible}] 
  \label{prop:dense_invariant_core}
	Let $ A $ be the generator of a Feller semigroup and $ \mc D $ be its domain. 
	Then every dense, invariant subspace $ D \subset \mc D $ is a core for $ A $.
\end{prop}

A notable class of Feller processes are the {\em pseudo-Poisson processes}. 
These are continuous-time variants of discrete-time chains that admit a simple construction.
Indeed, let $ Y $ be a discrete-time Markov chain on a finite set $ E $, $T$ the associated transition operator, and $ N $ an independent homogeneous Poisson process on $\mathbb R_+$ with rate $ r $.
Then the composition $ X(t)\ldef Y( N[0,t]) $ is a pseudo-Poisson process whose generator is the bounded operator $ A = r ( T - I ) $ with domain $ C( E ) $ and whose semigroup is given by $ T(t) = e^{ t A } $.

\subsection{Convergence theorems}
\label{subsection convergence theorems}

Let $ E $ be a compact metric space, to be viewed as the ambient space.
For $ n \ge 1 $, let $ E_n $ be a metric space, $ Y_n $ a Markov chain in $ E_n $, and $ \gamma_n \colon E_n \to E $ a continuous function.
Below, we state two results for analyzing the convergence of the Markov chains $ \{ Y_n \}_{ n \ge 1 } $ to a Feller process in $ E $.
In simple terms, they say that this convergence can be obtained from the convergence of semigroups or the convergence of generators.
For the precise statements, we need the following notion of convergence.
A sequence $ \{ f_n \}_{ n \ge 1 } $ with $ f_n \in C( E_n ) $ converges to $ f \in C( E ) $ (and we write $ f_n \to f $)
if it is asymptotically close to the sequence of projections $ \pi_n f = f \circ \gamma_n $ -- that is, if
$$
	\Vert
		f_n
	-
		\pi_n f
	\Vert_{ C( E_n ) }
		\xrightarrow[n \to \infty]{}
			0
		.
$$

\begin{theorem}[Chapter 4, Theorem 2.12 in \cite{ethierKurtzBook}]
    \label{theorem ethier kurtz path convergence}

    Let $ T_n $ be the transition operator of $ Y_n $ and $ \{ T( t ) \}_{ t \ge 0 } $ a Feller semigroup on $ C( E ) $.
    Suppose that the initial distributions of $ \{ \gamma_n( Y_n( 0 ) ) \} $ converge, say to $ \nu $, and that $ \{ \eps_n \}_{ n \ge 1 } $ is a positive sequence converging to zero such that
    $$
    	T_n^{ \floor{ t/\eps_n } }
    	\pi_n
    	f
    		\to
    			T( t ) f
    		,
    			\qquad
                \text{ for all }f \in C( E ),
    			\,
    			t \ge 0
    		.
    $$
    
    \noindent
    Then there exists a Markov process $ Y $ with initial distribution $ \nu $, transition semigroup $ \{ T( t ) \}_{ t \ge 0 } $, and sample paths in the Skorokhod space $ D( [ 0, \infty ), E )  $ such that the convergence
	$$
		\big(
			\gamma_n( Y_n ( \floor{ t/\eps_n }) )
		\big)_{ t \ge 0 }
			\Longrightarrow
        		\big(
        				Y( t )
        		\big)_{ t \ge 0 }
	$$
    holds in distribution in $D( [ 0, \infty ), E )$.
\end{theorem}

\begin{theorem}[Chapter 1, Theorems 6.1, 6.5 in \cite{ethierKurtzBook}]
    \label{theorem ethier kurtz modes of convergence}
    
    Let 
    $ T_n $ be the transition operator of $ Y_n $, 
    $ \{ T( t ) \}_{ t \ge 0 } $ a Feller semigroup on $ C( E ) $, 
    $ A $ the generator of $ \{ T( t ) \}_{ t \ge 0 } $,
    and 
    $ D $ a core for $ A $.
    For any positive sequence $ \{ \eps_n \}_{ n \ge 1 } $ converging to zero, the following are equivalent:
    \begin{enumerate}[ label = (\roman*) ]
    
      \item \label{item:cvTnt}
        $
        	T_n^{ \floor{ t/\eps_n } }
        	\pi_n
        	f
        		\to
        			T( t ) f
        $
        for every $ f \in C( E ) $ and $ t \ge 0 $,
  
      \item \label{item:cveTn}
        $
        	e^{ ( T_n - I ) t /\eps_n }
        	\pi_n
        	f
        		\to
        			T( t ) f
        $
        for every $ f \in C( E ) $ and $ t \ge 0 $,
     
    	\item
    	the above convergences hold uniformly in $ t $ on bounded intervals,
    
    	\item
    	for each $ f \in D $, there exists a sequence $ \{ f_n \}_{ n \ge 1 } $ with $ f_n \in C( E_n ) $ such that
        $$
        	f_n \to f
    			\qquad
    			\text{ and }
    			\qquad
    		\eps_n^{ -1 } 
    		( T_n - I )
        	f_n
        		\to
        			A f
        		.
        $$
    \end{enumerate}
\end{theorem}

\subsection{Separation Distance}
\label{subsection separation distance}

Let $ P $ and $ Q $ be probability distributions on the same finite set $ Z $.
Then their separation distance is defined as
\[ \sep(P,Q)\ldef \max_{z \in Z: Q(z) \ne 0} \bigg(1- \frac{P(z)}{Q(z)}\bigg).\]
The separation distance satisfies all of the metric axioms except symmetry and takes its values in $ [ 0, 1 ] $.
As the following result shows, it also admits a dual description.
\orange{Recall that  $ C_+(E)$ denotes the set of continuous functions from a topological space $ E $ to $ ( 0, \infty ) $.}

\begin{prop}
	\label{prop functional form of sep dist}
	Let $P$ and $Q$ be probability distributions on the same finite set $Z$.
Then
    $$
    	\sep( P, Q ) 
    		= 
        		\sup_{ f \in C_+(Z) } 
        			\bigg( 
        				1 - \frac{ \int_{Z} f dP }{ \int_{Z} f dQ}
        			\bigg)
    		.
	$$
\end{prop}

\begin{proof}
	It will suffice to show that
    \[ m \ldef \min_{z \in Z: Q(z) \ne 0}\frac{P(z)}{Q(z)} = \inf_{f \in C_+(Z)} \frac{ \int_{Z} f dP }{ \int_{Z} f dQ}.\]
	The definition of $ m $ gives us an inequality in one direction:
	$$
       \frac{ \int_{Z} f dP }{ \int_{Z} f dQ}
        	=
				\frac{ 
    		 		\sum_{ z \in Z }
    					f( z )
    					P(z)
				}{ 
    		 		\sum_{z \in Z }
    					f( z )
    					Q( z )
				}
        	\ge
				\frac{ 
    		 		\sum_{ z \in Z }
    					f( z )
    					( Q(z) m )
				}{ 
    		 		\sum_{z \in Z }
    					f( z )
    					Q( z )
				}
        	=
				m
			,
				\qquad
				f \in C_+( Z )
			.
	$$
	The reverse inequality can be obtained by using $ \indicator_z + \eps \indicator_Z \in C_+( Z) $ to approximate $ \indicator_z $:
    \begin{equation*}
		\inf_{ f \in C_+(Z)}
			\frac{  \int_Z f dP
				}{ 
					\int_Z f dQ
				}
			\le
		\frac{  \int_Z \indicator_z + \eps \indicator_Z \, dP
			}{ 
				\int_Z \indicator_z + \eps \indicator_Z \, dQ
			}
			\xrightarrow[ \eps \to 0^+ ]{}
		\frac{  \int_Z \indicator_z \, dP
			}{ 
				\int_Z \indicator_z\, dQ
			}
    		=
		\frac{  P(z)
			}{ 
				Q(z)
			}
		,
			\qquad
			z \in Z,
			\,
			Q(z) \neq 0
		.\qedhere
	\end{equation*}
\end{proof}

This dual description naturally extends to richer spaces\footnote{For an alternative extension of separation distance beyond finite spaces (using the essential supremum of the Radon-Nikodym derivative), see \cite{kolehe2024separation}.}: for probability distributions $P$ and $Q$ on a compact topological space $Z$, we define
\[	\sep(P,Q)\ldef
        		\sup_{ f \in C_+(Z) } 
        			\bigg( 
        				1 - \frac{ \int_{Z} f dP }{ \int_{Z} f dQ}
  \bigg).  \]
\orange{We note that this quantity continues to lie in $ [ 0, 1 ] $ and will satisfy the metric axioms (except symmetry) when $ Z $ is a separable metric space.}
\bigskip

In the context of Markov chains, it is typical to let $Q$ be the unique stationary distribution of the Markov chain and let $ P $ be the distribution of the Markov chain after a certain number of steps.
The separation distance then depends on the initial distribution of the chain and on time.
Following \cite{aldous1987separation}, it is standard here to maximize over all initial positions in the state space and let the separation distance be a function of time.
More precisely, for a Markov chain $Y$ with unique stationary distribution $\nu$, transition matrix $p$, and finite state space $ E $, one considers
$$
	\sepDist(m) 
		= 
			\max_{ y \in E} 
				\, 
				\sep( p^m( y, \, \cdot \, ), \nu) 
		= 
			\max_{y, y' \in E, \, \nu(y') \ne 0 } 
				\Big(
					1 - \frac{p^m (y,y')}{\nu(y')}
				\Big)
		,
			\qquad
			m \ge 0
		.
$$
We call this the separation distance of $ Y $.
Similarly, we define the separation distance of a Markov process $Y$ with unique stationary distribution $ \nu $, transition semigroup $\{ T( t ) \}_{t \ge 0}$, and compact topological state space $E$ as
$$
\sepDist(t) 
	= \sup_{x \in E} \sep( \Law_x(Y(t)), \nu) 
	=	\sup_{ x\in E, \, f \in C_+(E) }
			\bigg( 
				1 - \frac{ (  T( t ) f )( x ) }{ \int_{E} f d\nu}
			\bigg)
		,
			\qquad
			t \ge 0
		,
$$
\noindent
where $\Law_x(Y(t))$ is the distribution of the Markov process $Y$
at time $t$, initiated at $x$. %

A simple property of the separation distance is the following inequality:
\begin{equation}\label{claim compare positive and nonnegative functions}
  T(t) g \ge \big(1-\sepDist(t)\big) {\textstyle \int_E g \, d\nu},
  		\qquad
			g \in C( E ),
			\,
			g \ge 0
		.
\end{equation}
To see this, note that it holds by definition for positive $g$ and extends to nonnegative $g$ by using
${g + \eps \indicator_Z \in C_+( E )} $ to approximate $ g $ as before.
The following result establishes a few more properties.
\begin{prop}
	Let $ Y $ be a Markov process with unique stationary distribution $ \nu $, 
	transition semigroup $ \{ T( t ) \}_{ t \ge 0 } $, 
	and compact topological state space $ E $. 
	Then its separation distance $ \sepDist( t ) $ 
	\begin{enumerate}[ label = (\roman*) ]
		\item
		is a nonincreasing function of $ t $,
		and
		
		\item
		is submultiplicative -- that is, $ \sepDist( t + s ) \le \sepDist( t ) \sepDist( s ) $ for $ s, t \ge 0 $.
	\end{enumerate}
\end{prop}

\begin{proof}
	We begin with the second claim.
	To this end, fix $ s, t \ge 0 $ .
	Taking $f \in C_+( E )$, an application of the inequality \eqref{claim compare positive and nonnegative functions} implies that the following function is nonnegative:
	$$g = T(s) f - (1-\sepDist(s)) {\textstyle \int_E f d\nu } \cdot \indicator_E .$$	
	\noindent
	Applying \eqref{claim compare positive and nonnegative functions} now to $ g $, we must have that
	$ T(t) g \ge \big(1-\sepDist(t)\big) {\textstyle \int_E g \, d\nu} $.
	Let us simplify the left-hand side here using the semigroup identity $T(t+s) = T(t) T(s)$ and the conservativity of $ T( t ) $:
	$$
		T( t ) g
			= T( t ) T(s) f - (1-\sepDist(s)) {\textstyle \int_E f d\nu } \cdot T( t ) \indicator_E
			= T( t + s ) f - (1-\sepDist(s)) {\textstyle \int_E f d\nu } \cdot \indicator_E
			.
	$$
	Now we handle the right-hand side using the stationarity of $ \nu $:
	$${\textstyle\int_E g \, d\nu }
		= {\textstyle \int_E T(s) f d\nu } - (1-\sepDist(s)) {\textstyle \int_E f d\nu } \cdot {\textstyle \int_E \indicator_E  d\nu } 
		= {\textstyle \int_E f d\nu } - (1-\sepDist(s)) {\textstyle \int_E f d\nu } 
		= \sepDist(s) {\textstyle \int_E f d\nu } 
		.
	$$
	Our inequality thus becomes
	$$
		T( t + s ) f - (1-\sepDist(s)) {\textstyle \int_E f d\nu }
			\ge
				\big(1-\sepDist(t)\big) 
				\sepDist(s) {\textstyle \int_E f d\nu } 
			,
	$$	
	which rearranges to
	$$
		1 - \frac{ T( t + s ) f}{ \textstyle \int_E f d\nu }
			\le
				\sepDist(t)
				\sepDist(s)
			.
	$$	
	Recalling that $f \in C_+( E )$ was arbitrary gives us that $ \sepDist( t + s ) \le \sepDist( t ) \sepDist( s ) $, which proves the second claim.
	The first claim then follows from the fact that separation distances take values in $[0,1]$.
\end{proof}

We now provide a connection between the separation distance and the convergence of processes.
To this end, let us return to the setting of \cref{subsection convergence theorems}.
Let $ Y $ be a Feller process on $ E $, a compact metric space.
For $ n \ge 1 $, let $ E_n $ be a metric space, $ Y_n $ a Markov chain in $ E_n $, and $ \gamma_n \colon E_n \to E $ a \orange{continuous function}.
The projections $ \pi_n f = f \circ \gamma_n $ give rise to the following notion of convergence: a sequence $ \{ f_n \}_{ n \ge 1 } $ with $ f_n \in C( E_n ) $ converges to $ f \in C( E ) $ if 
$
	\Vert
		f_n
	-
		\pi_n f
	\Vert_{ C( E_n ) }
		\to
			0
		.
$
\orange{We remark that} if $f_n$ converges to $f$ in this sense and $\gamma_n(x_n)$ converges to $x$ in $E$, then $f_n(x_n)$ converges to $f(x)$. %
\begin{theorem}
	\label{thm general sep dist of a limit}
	Suppose that for every $ y \in E $ there is a sequence 
	$ 
		\{ y_n \}_{ n \ge 1 }
	$
	such that $ y_n \in E_n $ and
	$
		\gamma_n( y_n )
			\to
				y
			.
	$
	Suppose that each $ Y_n $ has a unique stationary measure $ \nu_n $, $ Y $ has a unique stationary measure $\nu $, and we have the weak convergence $ \nu_n \circ \gamma_n^{ -1 } \to \nu $.
    Let 
    $ T_n $ be the transition operator of $ Y_n $ 
    and 
    $ \{ T( t ) \}_{ t \ge 0 } $ be the semigroup of $ Y $,
    and
    suppose that $ \{ \eps_n \}_{ n \ge 1 } $ is a positive sequence converging to zero such that
    $$
    	T_n^{ \floor{ t/\eps_n } }
    	\pi_n
    	f
    		\to
    			T( t ) f
    		,
    			\qquad
    			f \in C( E ),
    			\,
    			t \ge 0
    		.
    $$
 
	\noindent
    Let 
    $ \sepDist_n $ be the separation distance of $ Y_n $,
    $ \sepDist_n^* $ be the separation distance of its continuous-time variant
    with jump rate $1/\eps_n$,
    and 
    $ \sepDist $ be the separation distance of $ Y $.
	Then the following inequalities hold:
	$$
        \sepDist(t) 
			\le
				\liminf_{ n \to \infty }
					\sepDist_n(\floor{ t/\eps_n}) 
                    \text{ and } \sepDist(t)              
            \le \liminf_{ n \to \infty }
					\sepDist_n^*( t )
			,
				\qquad
				t \ge 0
			.
	$$

\end{theorem}

\begin{proof}
  We prove only the first inequality -- the second can be established similarly (recall the equivalence of items~\ref{item:cvTnt} and~\ref{item:cveTn}
  in \cref{theorem ethier kurtz modes of convergence}).
  Fix $ t \ge 0 $.
  Let $x \in E$ and take a sequence $x_n \in E_n$ so that $\gamma_n( x_n )$ converges to $x$.
  Letting $ f \in C_+( E ) $, each $ \projection_n f$ will lie in $C_+( E_n ) $, so we can apply \eqref{claim compare positive and nonnegative functions} to obtain the inequality
	$$
		\left( 
		T_n^{ \floor{ t/\eps_n } } 
		\projection_n f \right) (x_n) 
		\ge 
        \big( 1 - \Delta_n(\floor{ t/\eps_n } )  \big) \,
		{\textstyle \int_{ E_n }
			\projection_n f
			d \nu_n}
				,
					\quad
                    n \ge 1
				.
	$$ 
    Let us consider the limit as $ n \to \infty $.
    By assumption, $T_n^{ \floor{ t/\eps_n } } \projection_n f$ converges to $T(t)f$, so the left-hand side tends to $(T(t)f)(x)$.
    For the constants $\Delta_n(\floor{ t/\eps_n })$, we pass to a subsequence so that they converge to their limit inferior.
	For the integrals, we use the weak convergence of stationary measures:
	$$
		{\textstyle \int_{ E_n }
			\projection_n f
			d \nu_n }
				=
            		{\textstyle \int_{ E }
            			f
            			d( \nu_n \circ \gamma_n^{ -1 } ) }
				\to
            		{\textstyle \int_{ E }
            			f
            			d\nu }
				.		
	$$
    Altogether, we obtain
	$$
		(T (t ) f)(x) \ge   
			\big( 1 - \liminf_{ n \to \infty } \sepDist_n(\floor{ t/\eps_n}) \big)
       {\textstyle  \int_E
			f
			\,
			d\nu}.
	$$ 
	Recalling that $f \in C_+( E )$ and $ x \in E $ were arbitrary
	concludes the proof.
\end{proof}

\section{The discrete chains}
\label{section discrete framework}

\subsection{The general setting}
\label{section general setting}

We start by proving some consequences of Assumptions~\ref{assumption finite state spaces} and \ref{assumption:commutation} for up-down chains.

\begin{prop}
	\label{prop up-down chains give intertwining}

    Suppose that there are up-down chains whose state spaces $ \{ \stateSpace_n \}_{ n \ge 0 } $ satisfy \ref{assumption finite state spaces} and whose up-steps $ \{ \upKernel_n \}_{ n \ge 0 } $ and down-steps $ \{ \downKernel_n \}_{ n \ge 1 } $ satisfy \ref{assumption:commutation}. 
	Let $ \{ \upDownOperator_n \}_{ n \ge 0 } $ be the transition operators of the chains and $\{ \downOperator_n \}_{ n \ge 1 } $ be the operators associated with the down-steps.
	Let $ \{ \generatorRates_n \}_{ n \ge 0 }$ be a positive sequence satisfying $\frac{\generatorRates_{ n - 1 }}{ \generatorRates_n}=\beta_{n}$ for $n \ge 1$, and define $ \discreteGenerators_n = \generatorRates_n ( \upDownOperator_n - I ) $ for $ n \ge 0 $.
Then the following statements hold for each $ n \ge 0 $:
    \begin{enumerate}[ label = (A\arabic*), leftmargin = \indentForAssumptions ]
    	\item
    	\label{assumption increasing generator rates}
    	$ 0 < \generatorRates_n < \generatorRates_{ n + 1 } $,

    	\item
    	\label{assumption down operators}
    	$ \downOperator_{ n + 1 } $ is a transition operator from $ C( \stateSpace_{ n } ) $ to $ C( \stateSpace_{ n + 1 } ) $,

    	\item
    	\label{assumption range containment}
    	$ 
    		\range ( \discreteGenerators_{ n + 1 } + \generatorRates_{ n } )
    			\subset
    				\range \downOperator_{ n + 1 }
    			,
    	$

    	\item
    	\label{assumption injective down operators}
    	$ \downOperator_{ n + 1 } $ is injective,
    	and

    	\item
    	\label{assumption intertwining}
    	$
    		\discreteGenerators_{ n + 1 }
    		\downOperator_{ n + 1 }
    			=
    				\downOperator_{ n + 1 }
    				\discreteGenerators_n
    			.
    	$

    \end{enumerate}

\end{prop}

\begin{proof}
    Claims \ref{assumption increasing generator rates}
    and \ref{assumption down operators} are trivially satisfied.
    To prove \ref{assumption range containment}, let $\{ \upOperator_n \}_{ n \ge 0 } $ be the operators associated with the up-steps, and recall that we have the factorizations $\upDownOperator_n=\upOperator_n \downOperator_{n+1}$.
    Given $n \ge 0$, we apply Assumption~\ref{assumption:commutation}
    and the identity $c_{n+1} \beta_{n+1}=c_n$ {to compute}
    \begin{equation*} %
      A_{n+1} =c_{n+1} (U_{n+1} D_{n+2}-I)
    =c_{n+1} (\beta_{n+1} D_{n+1} U_n -\beta_{n+1} I)
    =c_n (D_{n+1} U_n - I).
    \end{equation*}
    {Note that this can be rewritten as} $A_{n+1} +c_n= c_n D_{n+1} U_n$, which implies \ref{assumption range containment}.
    {On the other hand}, we can multiply by $D_{n+1}$ to obtain
    \[  A_{n+1} D_{n+1} 
    =c_n (D_{n+1} U_n D_{n+1} - D_{n+1})= D_{n+1} c_n ( U_n D_{n+1} - I)= D_{n+1} A_n,\]
    {establishing} \ref{assumption intertwining}.
    To \orange{prove} \ref{assumption injective down operators}, \orange{it will suffice to show that the $U_n D_{n+1}$ are injective. To do this, we will show that they have only positive eigenvalues. We proceed by induction.} 
    The case $n=0$ is trivial, since $C( \stateSpace_{ 0 } )$ has dimension $1$
    and {$U_0 D_{1}$} maps the function $1$ to itself (as any transition operator).
    Assume now that the statement holds for some $n$.
    {A classical fact from linear algebra says that} for linear operators $ A $ and $ B $, 
    the compositions $AB$ and $BA$ have the same eigenvalues 
    {with the same multiplicities}, except possibly $0$.
    Applied to our setting, this implies that $D_{n+1} U_n $ has only nonnegative eigenvalues.
    {The positivity of $ \beta_{n+1} $ and the commutation relation} \ref{assumption:commutation} then imply that the operator
     \[ 
     	U_{n+1} D_{n+2}= \beta_{n+1} D_{n+1} U_n +(1-\beta_{n+1}) I
	\]
    
    \noindent
    has only positive eigenvalues. 
    \orange{This completes the induction, establishing \ref{assumption injective down operators}.}
  \end{proof}

We now proceed to the more general setting, \orange{which will apply throughout Sections~\ref{section discrete framework} and~\ref{section convergence}}.
For every $ n \ge 0 $, let $ \upDownChain_n $ be a Markov chain with state space $ \stateSpace_n $ and transition operator $ \upDownOperator_n $
(\orange{we emphasize that} $\upDownChain_n$ need not be an up-down chain).
We fix some rates $ \{ \generatorRates_n \}_{ n \ge 0 } $ to construct continuous-time variants of these chains -- the pseudo-Poisson processes with generators $ \discreteGenerators_n = \generatorRates_n ( \upDownOperator_n - I ) $.
We assume that the state spaces satisfy \ref{assumption finite state spaces}
and that these processes, together with some operators $ \{ \downOperator_n \}_{ n \ge 1 } $, satisfy conditions \ref{assumption increasing generator rates}--\ref{assumption intertwining} for each $ n \ge 0 $.

Let us briefly comment on our assumptions.
The assumption \ref{assumption down operators} says that $ \downOperator_{ n + 1 } $ is a down-operator, a transition operator associated with a kernel from $ \stateSpace_{ n + 1 } $ to $ \stateSpace_n $. 
The assumption \ref{assumption range containment} is equivalent to the existence of a \emph{pseudo} up-operator\footnote{We use the expression \emph{pseudo} up-operator since $\upOperator_n$ need not be positive \orange{and as a result}, need not be a transition operator in the sense of \cref{ssec:def_kernels_operators}.} $ \upOperator_n \colon C( \stateSpace_{ n + 1 } ) \to C( \stateSpace_n ) $ satisfying
	\begin{equation}
	\label{definition up-operator}
		\generatorRates_n
		\downOperator_{ n + 1 }
		\upOperator_n
			=
            		\discreteGenerators_{ n + 1 }
        		+
            		\generatorRates_n
			,
	\end{equation}

\noindent
and \ref{assumption injective down operators} is equivalent to the uniqueness of this operator.
\noindent
The assumption \ref{assumption intertwining} is equivalent to the intertwining of the pseudo-Poisson processes, \orange{as in \eqref{definition intertwining}}.
Indeed, one can move from the generator relation to the semigroup relation by applying a result like Corollary 7.1 in \cite{krdrLeftmost}.

We have already seen that that the up-down setting of the introduction fits into our new context (\cref{prop up-down chains give intertwining}). 
In the following result, we show that this new context still exhibits much of the structure of the up-down setting.

\begin{prop}%
    \label{prop commutation relations}
    Recall that the pseudo up-operators $\{ \upOperator_n \}_{ n \ge 0 }$ are defined by \eqref{definition up-operator}.
    The following pseudo up-down factorizations and commutation relations hold:
    \begin{align}
    	\label{UD factorization}
    		\upDownOperator_n
    			& =
        			\upOperator_n
            		\downOperator_{ n + 1 }    				
    		,
    			\qquad\hspace{ 9.6 em }
    			n \ge 0
    		,
			\\
    \label{eq:commutation_relations}
    	\upOperator_n
    	\downOperator_{ n + 1 }
    		& = 
            		\frac{ \generatorRates_{ n - 1 } }{ \generatorRates_n }
        				\downOperator_n
        				\upOperator_{ n - 1 }
    			+
        			\Big( 
        				1 - \frac{ \generatorRates_{ n - 1 } }{ \generatorRates_n } 
        			\Big)
        			I
    		,
    			\qquad
    			n \ge 1
    		.
    \end{align}
\end{prop}

\begin{remark}	
	\orange{Despite the reappearance of these identities}, our new context is still more general than the up-down setting of the introduction. 
	Indeed, since the $ \upOperator_n $ need not be positive operators, we do not have genuine up-down factorizations here, and as a result, our chains need not be up-down chains.
\end{remark}

\begin{proof}    
	Observe that we have two forms for the generators $ \{ \discreteGenerators_n \}_{ n \ge 1 } $: the defining form $ \discreteGenerators_n = \generatorRates_n ( \upDownOperator_n - I )$ and the new form $ \discreteGenerators_n = \generatorRates_{ n - 1 } ( \downOperator_n \upOperator_{ n - 1 } - I ) $ obtained from  \eqref{definition up-operator}.
    Equating these immediately gives us that 
    	\begin{equation*}
    		\frac{ \generatorRates_{ n - 1 } }{ \generatorRates_n }
    		\downOperator_n
    		\upOperator_{ n - 1 }
        		+
    		\Big( 1 - \frac{ \generatorRates_{ n - 1 } }{ \generatorRates_n } \Big) I
    			=
            		 \upDownOperator_n
    			,
    				\qquad
    				n \ge 1
    			.
    	\end{equation*}
    
    \noindent
    It only remains to prove the first claim.
    For this, we substitute the two forms of the generator into the intertwining relation \ref{assumption intertwining} to obtain
    $$
    	\downOperator_{ n + 1 }
    	\generatorRates_n ( \upDownOperator_n - I )
			=
            	\generatorRates_n
            	(
                	\downOperator_{ n + 1 }
                	\upOperator_n
                	-
                	I
            	)
            	\downOperator_{ n + 1 }
			=
            	\generatorRates_n
				\downOperator_{ n + 1 }
            	(
                	\upOperator_n
					\downOperator_{ n + 1 }
                	-
                	I
            	)
			,
				\qquad
				n \ge 0
			.
    $$
    The factorization \eqref{UD factorization} now follows from the fact that $ \generatorRates_n \neq 0 $ and $ \downOperator_{ n + 1 } $ is injective (see \ref{assumption increasing generator rates} and \ref{assumption injective down operators}).
\end{proof}

\orange{Our next result is a natural extension of the above commutation relations} to operators going down more than one step at a time\footnote{\orange{It also appears in \cite{fulmanCommutation} as Lemma 4.4.}}. 
To state it, we will need the constants 
\begin{equation}
    \label{defn extended commutation weights}
    	\generatorRates_{ -1 }
    		=
    			0,
    		\quad
    	\commutationWeightProduct_{ k, n }
    		=
    			\frac{ \generatorRates_{ k - 1 } }{ \generatorRates_n }
    		,
    			\quad\qquad
    			0 \le n,
    			\, \,
    			0 \le k \le n + 1
    		,
\end{equation}
and the operators
\begin{equation}
\label{defn extended down operators}
	\downOperator_{ n, k }
		=
			\downOperator_n
			\downOperator_{ n - 1 }
			\ldots
			\downOperator_{ k + 1 }
		,
			\qquad
			0 \le k \le n
		.
\end{equation}
Note that $ \downOperator_{ n, k } $ is a transition operator from $ C( \stateSpace_k ) $ to $ C( \stateSpace_n ) $ and we have the special case
$
	\downOperator_{ n, n }
		=
			I
		.
$
\begin{prop}%
\label{prop extended commutation relations}
The following relations hold:
$$
	\upOperator_n
	\downOperator_{ n + 1, k }
		= 
        		\commutationWeightProduct_{ k, n }
            	\downOperator_{ n, k - 1 }
    			\upOperator_{ k - 1 }
			+
    			( 
    				1 - \commutationWeightProduct_{ k, n }
    			)
            	\downOperator_{ n, k }
		,
			\qquad
			1 \le k \le n
		.
$$

\end{prop}

\begin{proof}
	Let us fix $ n \ge 1 $ and induct on $ k $.
	The case $ k = n $ has already been established (Proposition \ref{prop commutation relations}).
	Assume now that the result holds for some $ k $ satisfying $ 2 \le k \le n $.
	Using 
	(\ref{defn extended down operators}),
	the inductive hypothesis,
	Proposition \ref{prop commutation relations},
	and
	(\ref{defn extended commutation weights}),
	we can compute
	\begin{align*}
    	\upOperator_n
    	\downOperator_{ n + 1, k - 1 }
			& =
            	\upOperator_n
            	\downOperator_{ n + 1, k }
				\downOperator_k
			\\
			& =
            		\commutationWeightProduct_{ k, n }
                	\downOperator_{ n, k - 1 }
            			\upOperator_{ k - 1 }
            			\downOperator_k
    			+
        			( 
        				1 - \commutationWeightProduct_{ k, n }
        			)
                	\downOperator_{ n, k  }D_k
			\\
			& =
            		\commutationWeightProduct_{ k, n }
                	\downOperator_{ n, k - 1 }
        			\left(
                		\commutationWeightProduct_{ k - 1, k - 1 }
            				\downOperator_{ k - 1 }
            				\upOperator_{ k - 2 }
            			+
            			\left( 
            				1 - \commutationWeightProduct_{ k - 1, k - 1 }
            			\right)
            			I 
					\right)
    			+
        			( 
        				1 - \commutationWeightProduct_{ k, n }
        			)
                	\downOperator_{ n, k - 1 }
			\\
			& =
            		\commutationWeightProduct_{ k, n }
            		\commutationWeightProduct_{ k - 1, k - 1 }
                	\downOperator_{ n, k - 2 }
    				\upOperator_{ k - 2 }
				+
            		\commutationWeightProduct_{ k, n }
        			\left( 
        				1 - \commutationWeightProduct_{ k - 1, k - 1 }
        			\right)
                	\downOperator_{ n, k - 1 }
    			+
        			( 
        				1 - \commutationWeightProduct_{ k, n }
        			)
                	\downOperator_{ n, k - 1 }
			\\
			& =
            		\commutationWeightProduct_{ k, n }
            		\commutationWeightProduct_{ k - 1, k - 1 }
                	\downOperator_{ n, k - 2 }
    				\upOperator_{ k - 2 }
				+
            		\left(
						1 
					- 
						\commutationWeightProduct_{ k, n } 
						\commutationWeightProduct_{ k - 1, k - 1 }
					\right)
					\downOperator_{ n, k - 1 }
			.
   	\end{align*}
	
	\noindent
	Verifying that 
	$ 
		\commutationWeightProduct_{ k - 1, n }
			=
        		\commutationWeightProduct_{ k, n }
        		\commutationWeightProduct_{ k - 1, k - 1 }
	$
	concludes the proof.
\end{proof}

\subsection{The density functions and the triangular descriptions}
\label{ssec:discrete_triangular}

In what follows, it will be useful to consider elements from arbitrary state spaces.
For this, we define
	$ 
		\stateSpace 
			= 
				\sqcup_{ n \ge 0 }
					\stateSpace_n
	$ 
and reserve the symbols $ r $, $ s $, and $ u $ for elements of this set. 
We will also need to consider elements from certain state spaces.
To do this, we let $ | s | $ denote the index of the state space that $ s $ belongs to and impose restrictions on this quantity.
For example, $ | s | \ge 1 $ specifies that $ s \in \sqcup_{ n \ge 1 } \stateSpace_n $ and $ | s | = k $ specifies that $ s \in \stateSpace_k $. 
We will use this notation freely, generalizing it in a natural way.

Let $ \upDownKernel_n $, $ \downKernel_n $, and $ \downKernel_{ n, k } $ be the kernels associated with the operators $ \upDownOperator_n $, $ \downOperator_n $, and $ \downOperator_{ n, k } $.
These can be computed using (\ref{identity kernel from operator}).
Let us also define matrix analogues for the pseudo up-operators,
\begin{align}
	\begin{split}
	\upKernel_n( s, u )
		& = 
			(
				\upOperator_n
    			\indicator_u
			)( s )
		,
			\hspace{41 mm}
			| u | - 1 = | s | = n \ge 0
		,
		\\
	\upKernel_{ k, n }( s, u )
		& = 
			(
				\upOperator_k
				\upOperator_{ k + 1 }
				\ldots
				\upOperator_{ n - 1 }
    			\indicator_u
			)( s )
		,
			\hspace{20 mm}
			n = | u | \ge | s | = k \ge 0
		,
	\end{split}
	\label{def of up-matrix}
\end{align}
keeping in mind that these need not be transition matrices (although they are \orange{in the up-down setting}).

To describe the transition operators, we will rely on the following family of functions
    \begin{equation}
    \label{definition of density functions}
    	\density_s ( u )
    		=
    			\begin{cases}
                    \downKernel_{ | u |, | s | }( u, s ),
						&
							| u | \ge | s |,
						\\
					0,
						&
							\text{else}.
				\end{cases}
    \end{equation}

\noindent
We will also make use of their restrictions
	$$ 
		( \density_s )_n
			=
				\density_s
					\vert_{ \stateSpace_n }
			,
				\qquad
				n \ge 0
			.
	$$

\noindent 
The Assumptions \ref{assumption finite state spaces} and \ref{assumption down operators} lead to the special case
$
	\density_\zeroVertex 
		\equiv 
			1
		.
$
We note also that
	\begin{equation}
		( \density_s )_{ | s | }
			=
				\indicator_s
			,
				\qquad
				s \in \stateSpace
			.
		\label{identity density as delta function}
	\end{equation}

\noindent
We like to think of $ \density_s( u ) $ as the \emph{density} of $ s $ in $ u $.
This is inspired by the probabilistic definition and by the fact that in certain contexts, these quantities yield familiar notions of density.
The basic properties of the density functions are summarized below.

\begin{prop}
	\label{prop properties of density functions}

    Let $ s \in \stateSpace $ and $ n \ge k \ge | s | $.
    The following identities hold:
    \begin{enumerate}[ label = (\roman*) ]
    	
		\item
		\label{Pieri rule for density functions}
		$
    		( \density_s )_n
                =
            			\sum_{ | u | = k }
            				\density_s( u )
            				( \density_u )_n
        		,
		$

		\item
        \label{density functions form partition of unity}
        $
			1
    			\equiv
                	\sum_{ | u | = k }
                		( \density_u )_n
					,
        $
		and

		\item
		\label{action of down operators on density functions}
		$
        	\downOperator_{ n + 1 }
    		( \density_s )_n
    			=	
                		( \density_s )_{ n + 1 }
				.
		$
    	
    \end{enumerate}

\end{prop}

\begin{remark}
	\label{remark discrete filtrations and lower order terms}
	The first identity above implies that the state spaces of our Markov chains induce a filtration on each $ C( \stateSpace_n ) $:
	namely,
	the subspaces
    $
    	\finiteFiltrations_{ k, n }
    		=
    			\linearSpan 	
    					\{
    						( \density_s )_n
    					\}_{ | s | = k }
    $
	form the filtrations
	$$
    	\finiteFiltrations_{ 0, n }
    		\subset
    			\finiteFiltrations_{ 1, n }
					\subset
					\ldots
            		\subset
                    	\finiteFiltrations_{ n, n }
                    		=
                    			C( \stateSpace_n )
		,
			\qquad
			n \ge 0
		.
	$$

	\noindent
	For this reason, we view $ ( \density_r )_n $ as being of `lower order' than $ ( \density_s )_n $ whenever $ | r | < | s | \le n $, and $ \density_r $ as lower order than $ \density_s $ whenever $ | r | < | s | $.
\end{remark}

\begin{proof}

	For the first identity, we evaluate the right-hand side at $ r \in \stateSpace_n $ and apply (\ref{definition of density functions}) and (\ref{defn extended down operators}):
	\begin{align*}
		\sum_{ | u | = k }
    		\density_s( u )
			( \density_u )_n( r )
				 =
    	           		\sum_{ | u | = k }
                            \downKernel_{ k, | s | }( u, s )
                            \downKernel_{ n, k }( r, u )
				 =
                        ( \downKernel_{ n, k } \downKernel_{ k, | s | } )( r, s )
				 =
                        \downKernel_{ n, | s | }( r, s )
				 =
                		\density_s( r )
				.
	\end{align*}

	\noindent
	Taking $ s = \zeroVertex $ in \ref{Pieri rule for density functions} then establishes the second identity (recall that $ \density_\zeroVertex \equiv 1 $).
	The third identity follows immediately from \ref{Pieri rule for density functions} and (\ref{definition of density functions}):
    \[
    	\downOperator_{ n + 1 }
		( \density_s )_n
				 =	
        				\sum_{ | u | = n }
        					\downKernel_{ n + 1 }( \, \cdot \,, u ) 
        					( \density_s )_n( u )
				 =	
        				\sum_{ | u | = n }
        					( \density_u )_{ n + 1 }
        					\density_s( u )
				 = 
    				( \density_s )_{ n + 1 }
				.
				\qedhere
                \]
\end{proof}

The next result illustrates the primary reason we consider the density functions: 
	they lead to triangular descriptions of the transition operators.
\begin{prop}
	\label{prop action of transition operators}

	Let $ n \ge 0 $. 
	The operator
	$
		\upDownOperator_n
	$
	is completely described by the identities
	\begin{align*}
		\upDownOperator_n
		( \density_s )_n
			& =
    				( 1 - \commutationWeightProduct_{ | s |, n } )
            		( \density_s )_n
    			+
    				\commutationWeightProduct_{ | s |, n }
    				\sum_{ | r | = | s | - 1 }
						( \density_r )_n
						\,
        				\upKernel_{ | r | }( r, s )
			,
				\qquad
				| s | \le n
			.
	\end{align*}
\end{prop}
\begin{remark}
	\label{remark triangularity}
	We use the term {\em triangular description} instead of triangularization because the functions 
	$
		\{ ( \density_s )_n \}_{ | s | \le n }
	$
	are not linearly independent (recall Proposition \ref{prop properties of density functions}\ref{Pieri rule for density functions}).
	\end{remark}

\begin{proof}

	\noindent
	When $ | s | = 0 $, the identity follows from the fact that $ s = \zeroVertex $, $ ( d_\zeroVertex )_n \equiv 1 $, and $ \commutationWeightProduct_{ 0, n } = 0 $ (see (\ref{defn extended commutation weights})).
	Let then $ n \ge | s | \ge 1 $.
    Combining \eqref{UD factorization}, \eqref{definition of density functions}, and \eqref{defn extended down operators} yields
	\begin{align*}
	\upDownOperator_n
	( \density_s )_n
			& = 
					\sum_{ | r | = n + 1 }
						\upKernel_n( \, \cdot \,, r )
    					\sum_{ | u | = n }
    						\downKernel_{ n + 1 }( r, u )
    						( \density_s )_n( u )
			= 
					\sum_{ | r | = n + 1 }
					\sum_{ | u | = n }
						\upKernel_n( \, \cdot \,, r )
    					\downKernel_{ n + 1, n }( r, u )
						\downKernel_{ n, | s | }( u, s )
			\\
			& = 
    				(
        				\upKernel_n
						\downKernel_{ n + 1, n }
        				\downKernel_{ n, | s | }
    				)( \, \cdot \,, s )
			= 
    				(
        				\upKernel_n
						\downKernel_{ n + 1, | s | }
    				)( \, \cdot \,, s )
			.
	\end{align*}

	Now we use Proposition \ref{prop extended commutation relations} and (\ref{definition of density functions}) to evaluate at some $ u \in \stateSpace_n $:
    \begin{align*}
		(
			\upKernel_n
			\downKernel_{ n + 1, | s | }
		)( u, s )
			& = 
    				\commutationWeightProduct_{ | s |, n }
    				(
						\downKernel_{ n, | s | - 1 }
        				\upKernel_{ | s | - 1 }
    				)( u, s )
    			+
    				( 1 - \commutationWeightProduct_{ | s |, n } )
            		\downKernel_{ n, | s | }( u, s )
			\\
			& = 
    				\commutationWeightProduct_{ | s |, n }
    				\sum_{ | r | = | s | - 1 }
						\downKernel_{ n, | s | - 1 }( u, r )
        				\upKernel_{ | r | }( r, s )
    			+
    				( 1 - \commutationWeightProduct_{ | s |, n } )
            		( \density_s )( u )
			\\
			& = 
    				\commutationWeightProduct_{ | s |, n }
    				\sum_{ | r | = | s | - 1 }
						\density_r ( u )
        				\upKernel_{ | r | }( r, s )
    			+
    				( 1 - \commutationWeightProduct_{ | s |, n } )
            		( \density_s )( u )
			.
    \end{align*}

	\noindent
	This establishes the identity.
	To see that these formulas completely describe the operator, recall that
	$
		\upDownOperator_n
	$
	acts on $ C( \stateSpace_n ) $, which is spanned by the functions 
	$ 
		\{
			( \density_s )_n
		\}_{ | s | = n }
	$
	(see (\ref{identity density as delta function})).
\end{proof}

\subsection{The diagonal descriptions}

The next step in our analysis is to obtain diagonal descriptions of the transition operators.
For this, we will need the constants
$$
	\coeffEigInDensityBasis_{ i, j }
		=
			\prod_{ m = i }^{ j - 1 }
				\frac{
						\generatorRates_m
					}{
						\generatorRates_{ m - 1 } - \generatorRates_{ j - 1 }
					}
		,\qquad
	\coeffDensityInEigBasis_{ i, j }
		=
			\prod_{ m = i }^{ j - 1 }
				\frac{
						\generatorRates_m
					}{
						\generatorRates_m - \generatorRates_{ i - 1 }
					}
		,
			\quad\qquad
			0 \le i \le j
		,
$$
and the functions
\begin{equation}
	\label{identity expansion of eigenfunctions}
	\eigenfunction_s
		=
    		\sum_{ | r | \le | s | }
    			\upKernel_{ | r |, | s | }( r, s )
				\,
				\coeffEigInDensityBasis_{ | r |, | s | }
				\,
    			\density_r
		,
			\qquad
			s \in \stateSpace
		.
\end{equation}

\noindent
Note that \ref{assumption increasing generator rates} ensures that the denominators above are nonzero and we have the special cases  $ \coeffDensityInEigBasis_{ i, i } \equiv 1 $ and $ \coeffDensityInEigBasis_{ 0, j } \equiv 1 $ (see (\ref{defn extended commutation weights})).
Moreover, the special case $ \coeffEigInDensityBasis_{ i, i } \equiv 1 $ reveals that
$
	\eigenfunction_s
		=
			\density_s
		+	
			\text{`lower order terms'}
$
and in particular,
$
	\eigenfunction_\zeroVertex
		=
			\density_\zeroVertex
		\equiv
			1
		.
$
We will need to express the density functions in terms of these new functions.
This is done in the next few results.
\begin{lemma}
    \label{lem identity for coefficient of density in eigenbasis}
    The following identity holds:
    \begin{equation*}
		\sum_{ k = i }^{ j }
			\coeffEigInDensityBasis_{ i, k }
			\coeffDensityInEigBasis_{ k, j }
    		=
				\indicator( i = j )
    		,
    			\qquad
				0 \le i \le j
			.
    \end{equation*}
\end{lemma}

\begin{remark}
	Interpreting this as a matrix identity reveals that $ \coeffEigInDensityBasis $ and $ \coeffDensityInEigBasis $ can be interchanged.
\end{remark}

\begin{proof}

	The $ i = j $ case is trivial since $ \coeffEigInDensityBasis_{ i, i } = \coeffDensityInEigBasis_{ i, i } = 1 $.
	Let then $ 0 \le i < j $.
	Consider the Lagrange interpolating polynomial for the points
	$
		\{
			( \generatorRates_{ i }, \coeffDensityInEigBasis_{ i, j } ),
			( \generatorRates_{ i + 1 }, \coeffDensityInEigBasis_{ i, j } ),
			\ldots,
			( \generatorRates_{ j - 1 }, \coeffDensityInEigBasis_{ i, j } )
		\}.
	$
	This polynomial must be $ L( x ) \equiv \coeffDensityInEigBasis_{ i, j } $, since it interpolates the points and its degree is less than the number of points.
	Using the Lagrange basis expansion, we can write this as
	$$
		\coeffDensityInEigBasis_{ i, j }
			\equiv
				\sum_{ k = i }^{ j - 1 }
					\coeffDensityInEigBasis_{ i, j }
					\prod_{ \substack{ m = i \\ m \neq k } }^{ j - 1 }
						\frac{ \generatorRates_m - x }{ \generatorRates_m - \generatorRates_k }
			.
	$$
	
	\noindent
	Evaluating now at $ x = \generatorRates_{ i - 1 } $, we obtain
	\begin{align*}
		\coeffDensityInEigBasis_{ i, j }
			& =
				\sum_{ k = i }^{ j - 1 }
    			\prod_{ m = i }^{ j - 1 }
    				\frac{
    						\generatorRates_m
    					}{
    						\generatorRates_m - \generatorRates_{ i - 1 }
    					}
					\prod_{ \substack{ m = i \\ m \neq k } }^{ j - 1 }
						\frac{ 
							\generatorRates_m - \generatorRates_{ i - 1 } 
						}{ 
							\generatorRates_m - \generatorRates_k 
						}
			=
				\sum_{ k = i }^{ j - 1 }
    				\frac{
							1
    					}{
    						\generatorRates_k - \generatorRates_{ i - 1 }
    					}
    			\prod_{ m = i }^{ j - 1 }
					\generatorRates_m
					\prod_{ m = i }^{ k - 1 }
						\frac{ 
							1
						}{ 
							\generatorRates_m - \generatorRates_k 
						}
					\prod_{ m = k + 1 }^{ j - 1 }
						\frac{ 
							1
						}{ 
							\generatorRates_m - \generatorRates_k 
						}
			\\
			& =
				-\sum_{ k = i }^{ j - 1 }
					\prod_{ m = i }^{ k }
						\frac{ 
							\generatorRates_m
						}{ 
							\generatorRates_{ m - 1 } - \generatorRates_k 
						}
					\prod_{ m = k + 1 }^{ j - 1 }
						\frac{ 
							\generatorRates_m
						}{ 
							\generatorRates_m - \generatorRates_k 
						}
			=
                -\sum_{ k = i }^{ j - 1 }
                    \coeffEigInDensityBasis_{ i, k + 1 }
                    \coeffDensityInEigBasis_{ k + 1, j }
			=
                -\sum_{ k = i + 1 }^{ j }
                    \coeffEigInDensityBasis_{ i, k }
                    \coeffDensityInEigBasis_{ k, j }
			.
	\end{align*}
	
	\noindent
	Recalling that $ \coeffEigInDensityBasis_{ i, i } = 1 $ concludes the proof.
\end{proof}

\begin{prop}
	\label{prop expansion of density in eigenbasis}
	The following expansions hold:
    $$
    	\density_s
    		=
        		\sum_{ | r | \le | s | }
        			\upKernel_{ | r |, | s | }( r, s )
    				\,
    				\coeffDensityInEigBasis_{ | r |, | s | }
    				\,
        			\eigenfunction_r
    		,
				\qquad
				s \in \stateSpace
			.
    $$
\end{prop}

\begin{proof}%
    Let $ s \in \stateSpace $.
    Using (\ref{identity expansion of eigenfunctions}), the definition of the up-matrices (\ref{def of up-matrix}), and Lemma \ref{lem identity for coefficient of density in eigenbasis}, we can compute
    \begin{align*}
		\sum_{ | r | \le | s | }
			\upKernel_{ | r |, | s | }( r, s )
			\,
			\coeffDensityInEigBasis_{ | r |, | s | }
			\,
			\eigenfunction_r
				& = 
            		\sum_{ | r | \le | s | }
            			\upKernel_{ | r |, | s | }( r, s )
            			\,
            			\coeffDensityInEigBasis_{ | r |, | s | }
            			\,
                		\left(\sum_{ | u | \le | r | }
                    			\upKernel_{ | u |, | r | }( u, r )
                				\,
                				\coeffEigInDensityBasis_{ | u |, | r | }
                				\,
                    			\density_u\right)
				\\
				& = 
            		\sum_{ | u | \le | s | }
            			\density_u
    					\left(\sum_{ k = | u | }^{ | s | }
            				\coeffEigInDensityBasis_{ | u |, k }
            				\,
                			\coeffDensityInEigBasis_{ k, | s | }
                    		\sum_{ | r | = k }
                    			\upKernel_{ | u |, k }( u, r )
                    			\upKernel_{ k, | s | }( r, s ) \right)
				\\
				& = 
            		\sum_{ | u | \le | s | }
            			\density_u
    					\left( \sum_{ k = | u | }^{ | s | }
            				\coeffEigInDensityBasis_{ | u |, k }
            				\,
                			\coeffDensityInEigBasis_{ k, | s | }
            				\,
                			\upKernel_{ | u |, | s | }( u, s )  \right)
				\\
				& = 
            		\sum_{ | u | \le | s | }
            			\density_u
        				\,
            			\upKernel_{ | u |, | s | }( u, s ) 
    					\indicator( | u | = | s | )
				\\
				& = 
            		\sum_{ | u | = | s | }
            			\density_u
        				\,
            			\upKernel_{ | s |, | s | }( u, s ) 
				\\			
				& = 
            		\density_s
				. \qedhere
    \end{align*}
\end{proof}

We now obtain diagonal descriptions for the transition operators.
\begin{prop}
	\label{prop eigenbasis}
	Let $ n \ge 0 $. 
	The operator
	$
		\upDownOperator_n
	$
	is completely described by the identities 
	\begin{equation}\label{eq:Tnh}
    	\upDownOperator_n
    	( \eigenfunction_s )_n
    		=
        		( 1 - \commutationWeightProduct_{ | s |, n } )
				( \eigenfunction_s )_n
			,
				\qquad
				| s | \le n
			.
		\end{equation}
    Moreover, the eigenspaces of $T_n$ provide the decomposition
	\begin{equation}\label{eq:eigenspaces}
		C( \stateSpace_n )
			=
        		\bigoplus_{ k = 0 }^n
            		\,
        			\linearSpan 	
            			\{
            				( \eigenfunction_s )_n
            			\}_{ | s | = k }
			,
	\end{equation}
	and the multiplicities of the eigenvalues are given by
	\begin{equation*}%
		\dim 
		\linearSpan 	
    		\{
    			( \eigenfunction_s )_n
    		\}_{ | s | = k }
       			=
        			\begin{cases}
           				| \stateSpace_k | - | \stateSpace_{ k - 1 } |,
            				&
            					\quad
								1 \le k \le n,
            				\\
        				1,
            				&
            					\quad
								k = 0
            				.
        			\end{cases}
	\end{equation*}
\end{prop}
\begin{remark}
	\label{remark diagonal}
	
	We call this a \emph{diagonal description} instead of a diagonalization because the functions $ \{ ( \eigenfunction_s )_n \}_{ | s | \le n } $ need not be non-zero or independent.
\end{remark}

\begin{proof}
	We first prove \eqref{eq:Tnh}. The case $s=\zeroVertex$ is trivial, so let us assume $1 \le |s| \le n$.
	Applying (\ref{identity expansion of eigenfunctions}) and Proposition \ref{prop action of transition operators} gives us that
	\begin{align}
		( \upDownOperator_n \orange{- I} )
		( \eigenfunction_s )_n
			& =
        			\sum_{ | r | \le | s | }
        				\upKernel_{ | r |, | s | }( r, s )
        				\,
        				\coeffEigInDensityBasis_{ | r |, | s | }
        				\,
                		( \upDownOperator_n - I )
        				( \density_r )_n
		\nonumber	\\
			& = 
					\sum_{ | r | \le | s | }
        				\upKernel_{ | r |, | s | }( r, s )
        				\,
        				\coeffEigInDensityBasis_{ | r |, | s | }
        				\,
						\bigg(
                				( - \commutationWeightProduct_{ | r |, n } )
                        		( \density_r )_n
 				+
                				\commutationWeightProduct_{ | r |, n }
                				\sum_{ | u | = | r | - 1 }
            						( \density_u )_n
            						\,
                    				\upKernel_{ | u | }( u, r )        				
        				\bigg)
			\nonumber	\\
			& =
				-\!\!\sum_{ | r | \le | s | }
        				\upKernel_{ | r |, | s | }( r, s )
        				\,
        				\coeffEigInDensityBasis_{ | r |, | s | }
        				\,
				\commutationWeightProduct_{ | r |, n }  ( \density_r )_n
			+
				\hspace{-3mm}
				\sum_{ | r | \le | s | \atop  | u | = | r | - 1  }
					\hspace{-2mm}
					\upKernel_{ | r |, | s | }( r, s )
    				\,
    				\coeffEigInDensityBasis_{ | r |, | s | }
					\,
					\commutationWeightProduct_{ | r |, n }
                	( \density_u )_n
					\,
    				\upKernel_{ | u | }( u, r )     .\label{eq:Tsh}
	\end{align}
We then use (\ref{defn extended commutation weights}) and the identity
	\begin{equation}
		\label{eta recursion}
			(
						\generatorRates_{ i - 1 }
			-
						\generatorRates_{ j - 1 }
			)\,
			\coeffEigInDensityBasis_{ i, j }
			=
				\coeffEigInDensityBasis_{ i + 1, j }
				\,
						\generatorRates_i,
				\qquad
				0 \le i \le j - 1
			,
	\end{equation}
	to rewrite the contribution from the double sum as
	\begin{align*}
		\sum_{ | r | \le | s | \atop  | u | = | r | - 1  }
			\upKernel_{ | r |, | s | }( r, s )
			\,
			\coeffEigInDensityBasis_{ | r |, | s | }
			\,
			\commutationWeightProduct_{ | r |, n }
			( \density_u )_n
			\,
			\upKernel_{ | u | }( u, r )
        			& = 
        				 	\sum_{ | u | < | s | }
        						( \density_u )_n
                				\coeffEigInDensityBasis_{ | u | + 1, | s | }
                				\,
                				\commutationWeightProduct_{ | u | + 1, n }
                			\hspace{-2.9mm}
							\sum_{ | r | = | u | + 1 }
                				\hspace{-3mm}
								\upKernel_{ | u |, | r | }( u, r )
                				\,
        						\upKernel_{ | r |, | s | }( r, s )
        			\\
        			& = 
        					\sum_{ | u | < | s | }
        						( \density_u )_n
                				\frac{
									\coeffEigInDensityBasis_{ | u | + 1, | s | }
                    				\,
									\generatorRates_{ | u | }
								}{ 
									\generatorRates_{ n }
								}
        						\upKernel_{ | u |, | s | }( u, s )
        			\\
        			& = 
        					\sum_{ | u | < | s | }
        						\upKernel_{ | u |, | s | }( u, s )
                				\,
        						( \density_u )_n
                				\frac{
									\coeffEigInDensityBasis_{ | u |, | s | }
									( 
										\generatorRates_{ | u | - 1 }
									-
										\generatorRates_{ | s | - 1 }
									)	
								}{ 
									\generatorRates_{ n }
								}
        			\\
        			& = 
        					\sum_{ | u | \le | s | }
        						\upKernel_{ | u |, | s | }( u, s )
                				\,
        						( \density_u )_n
                				\,
                				\coeffEigInDensityBasis_{ | u |, | s | }
                				\,
        						( \commutationWeightProduct_{ | u |, n } - \commutationWeightProduct_{ | s |, n } )
					\\
        			& = 
				- \commutationWeightProduct_{ | s |, n } (h_s)_n + 
				\sum_{ | u | \le | s | }
        						\upKernel_{ | u |, | s | }( u, s )
                				\,
        						( \density_u )_n
                				\,
                				\coeffEigInDensityBasis_{ | u |, | s | }
                				\,
 \commutationWeightProduct_{ | u |, n }
        			.
	\end{align*}
	\orange{Substituting this form} into \eqref{eq:Tsh} \orange{then results in a cancellation of sums, establishing \eqref{eq:Tnh}}.
	\orange{Moving on to \eqref{eq:eigenspaces}}, Proposition \ref{prop expansion of density in eigenbasis} reveals that the functions
	$
		\{
			( \eigenfunction_s )_n
		\}_{ | s | \le n }
	$
	span
	$
		C( \stateSpace_n ).
	$
	The direct sum decomposition then follows from observing that the summands correspond to eigenspaces \orange{(see \eqref{eq:Tnh})} with distinct eigenvalues (see (\ref{defn extended commutation weights}) and \ref{assumption increasing generator rates}.
	\orange{Finally}, recall from \ref{assumption injective down operators} and (\ref{identity density as delta function}) that the down-operators are injective and the set
	$
		\{
			( \density_s )_k
		\}_{ | s | = k }
	$
	is independent for any $ k $.
	Applying Proposition \ref{prop properties of density functions}\ref{action of down operators on density functions}, we find that the following sets are independent:
	$$
		\{
			( \density_s )_n	
		\}_{ | s | = k }
			=
        		\downOperator_n
				\downOperator_{ n - 1 }
				\cdots
				\downOperator_{ k + 1 }
				\{
        			( \density_s )_k
        		\}_{ | s | = k }
			,
				\qquad
				0 \le k \le n
			.
	$$

	\noindent
	Combining this with Propositions~\ref{prop properties of density functions}\ref{Pieri rule for density functions} and \ref{prop expansion of density in eigenbasis} and the independence of the eigenspaces, we obtain the identity
	\begin{align*}
		| \stateSpace_k |
	-
		| \stateSpace_{ k - 1 } |
			& = 
            		\dim
            		\linearSpan 	
                		\{
                			( \density_s )_n
                		\}_{ | s | = k }
            	-
            		\dim
            		\linearSpan 	
                		\{
                			( \density_s )_n
                		\}_{ | s | = k - 1 }
			\\
			& = 
               		\dim
               		\linearSpan 	
                   		\{
                   			( \density_s )_n
                   		\}_{ | s | \le k }
               	-
               		\dim
               		\linearSpan 	
                   		\{
                   			( \density_s )_n
                   		\}_{ | s | \le k - 1 }
			\\
			& = 
               		\dim
               		\linearSpan 	
                   		\{
                   			( \eigenfunction_s )_n
                   		\}_{ | s | \le k }
               	-
               		\dim
               		\linearSpan 	
                   		\{
                   			( \eigenfunction_s )_n
                   		\}_{ | s | \le k - 1 }
			\\
			& = 
            		\dim
            		\linearSpan 	
                		\{
                			( \eigenfunction_s )_n
                		\}_{ | s | = k }
	\end{align*}

	\noindent
	for $ 1 \le k \le n $.
	This establishes the final claim (the $ k = 0 $ case is trivial).
\end{proof}

\subsection{Large time behavior}
\label{ssec:large_time}
We now use the diagonal descriptions to carry out an asymptotic analysis.
We begin by considering the behavior of the chains for large time.
\begin{prop}
	\label{prop density estimate}
	Let
	$
		f 
			= 
				a_\zeroVertex
				\eigenfunction_\zeroVertex
			+
				\sum_{ | s | = j }^k
    				a_s
    				\eigenfunction_s
	$
	for some $ k \ge j > 0 $.
	Then there exists some $B_f>0$ such that
	$$
		\Big|
    		\mbb E
    		\left[
    				f
    				\big( 
    					\upDownChain_n ( m )
    				\big)
    		\right]
		-
    		a_\zeroVertex
		\Big|
			\le
        		B_f
                ( 1 - \commutationWeightProduct_{ j, n } )^m
			,
				\qquad
				m \ge 0,
				\,
				n \ge k,
	$$
	\noindent
	for any initial distribution.
	Consequently, we have the convergence (for any initial distribution)
	$$
		\mbb E
		\left[
				\density_s
				\big( 
					\upDownChain_n ( m )
				\big)
		\right]
				\xrightarrow[ m \to \infty ]{}
					\upKernel_{ 0, | s | }( \zeroVertex, s )
				,
					\qquad
					n \ge | s |
				.
	$$
\end{prop}

\begin{proof}
    First observe that each $ \eigenfunction_s $ is bounded on $ \stateSpace $: 
    indeed, each $\density_r$ is bounded, and the functions $\eigenfunction_s$
    are linear combinations of $\density_r$ (see \eqref{identity expansion of eigenfunctions}).
    Therefore, for $ s \in \stateSpace $, $ m \ge 0 $, and $ n \ge | s | $, 
    we can apply Proposition \ref{prop eigenbasis} to obtain the estimate
    \begin{multline*}
    	\left|
        	\mbb E
        	\left[
        			\eigenfunction_s
        			\big( 
        				\upDownChain_n ( m )
        			\big)
        	\right]
    	\right|
    	    		 \le 
            	\sup_{ u \in \stateSpace_n }
    			\left|
        				\mbb E
                    	\left[
                    			\eigenfunction_s
                    			\big( 
                    				\upDownChain_n ( m )
                    			\big)
        				\big|
        						\upDownChain_n( 0 ) = u
                    	\right]
            	\right| \\
    	= 
            	\sup_{ u \in \stateSpace_n }
    			\left|
    				( 
					\upDownOperator_n^m
                	( \eigenfunction_s )_n
                	)( u )
            	\right|
    	 = 
                ( 1 - \commutationWeightProduct_{ | s |, n } )^m
                \sup_{ u \in \stateSpace_n }
                	| \eigenfunction_s ( u ) |
        \le 
        		( 1 - \commutationWeightProduct_{ | s |, n } )^m
                \sup_{ u \in \stateSpace }
                	| \eigenfunction_s ( u ) |
		.
    \end{multline*}

    \noindent
    Given now
	$
		f 
			= 
				a_\zeroVertex
				\eigenfunction_\zeroVertex
			+
				\sum_{ | s | = j }^k
    				a_s
    				\eigenfunction_s
			,
	$
	$ m \ge 0 $, and $ n \ge k $, we can write (recall that $ \eigenfunction_\zeroVertex \equiv 1 $)
    \begin{align*}
    	\Big|
        	\mbb E
        	\left[
        			f
        			\big( 
        				\upDownChain_n ( m )
        			\big)
        	\right]
    	-
    		a_\zeroVertex
    	\Big|
			& =
            	\bigg|
    				a_\zeroVertex
                	\mbb E
                	\left[
    				\eigenfunction_\zeroVertex
        			\big( 
        				\upDownChain_n ( m )
        			\big)
                	\right]
				+
    				\sum_{ | s | = j }^k
        				a_s
                    	\mbb E
                    	\left[
        				\eigenfunction_s
            			\big( 
            				\upDownChain_n ( m )
            			\big)
                    	\right]
            	-
            		a_\zeroVertex
            	\bigg| 
		\\
		 & \le
				\sum_{ | s | = j }^k
    				| a_s |
                	\left|
						\mbb E
                        	\left[
                				\eigenfunction_s
                    			\big( 
                    				\upDownChain_n ( m )
                    			\big)
                        	\right]
					\right|
		 \le
        		( 1 - \commutationWeightProduct_{ j, n } )^m
				\sum_{ | s | = j }^k
    				| a_s |
                \sup_{ u \in \stateSpace }
                	| \eigenfunction_s ( u ) |
			.
    \end{align*}
    Taking $B_f$ to be the above sum establishes the inequality.
    Recalling that $ \commutationWeightProduct_{ j, n } $ lies in $ ( 0, 1 ) $ shows that the expectations should converge to $ a_\zeroVertex $ as $ m \to \infty $.
	In the case when $ f = \density_s $, this coefficient can be identified from the expansion in Proposition \ref{prop expansion of density in eigenbasis} (recall that $ \coeffDensityInEigBasis_{ 0, j } \equiv 1 $).
\end{proof}

We wish to obtain convergence in distribution from this result.
For this, we will show that the limiting densities $ \upKernel_{ 0, | s | }( \zeroVertex, s ) $ are the expected densities with respect to some distribution.
A crucial observation here is that the formula
\begin{equation}
	\upDownDist_n
		=
        	\upKernel_{ 0, n }( \zeroVertex, \, \cdot \, )
		,
			\qquad
			n \ge 0
		,
	\label{defn stationary measures}
\end{equation}

\noindent
defines a probability measure on each state space.
This can be verified using Propositions~\ref{prop density estimate}~and~\ref{prop properties of density functions}\ref{density functions form partition of unity}.
We investigate the properties of these measures in the following result.
Here we will need to consider the spaces
\begin{equation*}
	\orange{\spaceOfDensityFunctions
		=
			\linearSpan
        		\{
        			\eigenfunction_s
        		\}_{ s \in \stateSpace }
		,
			\quad\qquad
    \spaceOfDensityFunctions_n
    	=
    		\linearSpan\{ \eigenfunction_s \}_{ | s | \le n }
		,
			\quad
			n \ge 0}
		,
\end{equation*}
and the coefficient functional $[ \eigenfunction_\zeroVertex ] \colon \spaceOfDensityFunctions \to \mbb R $ mapping
$
	f = {\textstyle \sum_s} a_s \eigenfunction_s
$
to $ a_\zeroVertex $.
Note that this functional is well-defined since the functions
$
	\{
		\eigenfunction_s
	\}_{ s \in \stateSpace }
$
are independent.
\begin{prop}
    \label{prop properties of stationary measures}
    Let $ n \ge 0 $.
    The following statements hold:
    \begin{enumerate}[ label = (\roman*) ]
    	
    	\item
    	\label{up consistency of stationary measures}
    	$
    		\upDownDist_n \upKernel_n
    			=
    				\upDownDist_{ n + 1 }
    			,
    	$
    
    	\item
    	\label{down consistency of stationary measures}
    	$
    		\upDownDist_{ n + 1 } 
    		\downKernel_{ n + 1 }
    			=
    				\upDownDist_n
    			,
    	$
    		
    	\item
    	\label{up down stationary measures}
    	$ \upDownDist_n $ is the unique stationary distribution of $ \upDownChain_n $, 
    
    	\item
    	\label{action of stationary measures on density functions}
    	$
    		\upDownDist_{ | s | }( s )
				=
					\orange{\int_{ \stateSpace_n }
						( \density_s )_n
						\,
						d\upDownDist_n}
    	$
    	for $ | s | \le n $,
    	and
    	\item
    	\label{integrating against stationary measure}
    	$
    		{
    		[ \eigenfunction_\zeroVertex ]
    		f
    			=
    				\int_{ \stateSpace_n }
    					( f )_n
    					d \upDownDist_n
    		}
    	$
    	for $ f \in \spaceOfDensityFunctions_n $.
    \end{enumerate}
\end{prop}

\begin{proof}

	The first identity follows from definitions:
	given $ n \ge 0 $ and $ s \in \stateSpace_{ n + 1 } $, we simply compute
	\begin{equation*}
		(
    		\upDownDist_n
    		\upKernel_n
		)( s )
			= 
					\sum_{ | r | = n }
						\upDownDist_n( r )
						\upKernel_n( r, s )
			= 
					\sum_{ | r | = n }
						\upKernel_{ 0, n }( \zeroVertex, r )
						\upKernel_n( r, s )
			 = 
            		\upKernel_{ 0, n + 1 }( \zeroVertex, s )
			 = 
            		\upDownDist_{ n + 1 }( s )
			.
	\end{equation*}

	Let us show \ref{up down stationary measures}.
	We show that $ \upDownDist_n $ is a stationary distribution for $ \upDownChain_n $ via induction.
	The $ n = 0 $ case is trivial.
	Suppose now that $ \upDownDist_n $ is a stationary distribution for $ \upDownChain_n $ for some $ n \ge 0 $.
	Applying \ref{up consistency of stationary measures}, Proposition \ref{prop commutation relations}, and the induction hypothesis, we find that
	$ 
		\upDownDist_{ n + 1 }
	$ 
	is stationary for $ \upDownChain_{ n + 1 } $:
	\begin{align*}
		\upDownDist_{ n + 1 }
		\upDownKernel_{ n + 1 }
			& = 
            		(
                		\upDownDist_n
                		\upKernel_n
            		)
            		\upKernel_{ n + 1 }
					\downKernel_{ n + 2 }
			\\
			& = 
                	\upDownDist_n
                	\upKernel_n
            		(
    					\tfrac{c_n}{c_{n+1}}
						\downKernel_{ n + 1 }
    					\upKernel_n
					+
						( 1 - \tfrac{c_n}{c_{n+1}} )
						I
            		)
			\\
			& = 	
    					\tfrac{c_n}{c_{n+1}}
                        \upDownDist_n
						\upDownKernel_n
    					\upKernel_n
					+
						( 1 - \tfrac{c_n}{c_{n+1}} )
                        \upDownDist_n
						\upKernel_n
			\\
			& = 	
    					\tfrac{c_n}{c_{n+1}}
                        \upDownDist_n
    					\upKernel_n
					+
						( 1 - \tfrac{c_n}{c_{n+1}} )
                        \upDownDist_n
						\upKernel_n
			= 	
                        \upDownDist_n
    					\upKernel_n
			= 
						\upDownDist_{ n + 1 }
			.
	\end{align*}
	
	\noindent
	The uniqueness follows from Proposition \ref{prop eigenbasis}, which showed that the eigenvalue 1 has multiplicity 1 in each transition operator.

	We now show, using Proposition \ref{prop commutation relations}, that
	$ 
		\upDownDist_{ n + 1 } 
		\downKernel_{ n + 1 }
	$ 
	is a stationary measure for $ \upDownChain_n $:
	\begin{align*}
		(
    		\upDownDist_{ n + 1 }
    		\downKernel_{ n + 1 }
		)
		\upDownKernel_n
			& = 
            		\upDownDist_{ n + 1 }
            		\downKernel_{ n + 1 }
            		\,
					\upKernel_n
					\downKernel_{ n + 1 }
			\\
			& = 
                	\upDownDist_{ n + 1 }\,
            				\tfrac{c_{n+1}}{c_{n}}
            		\big(
						\upKernel_{ n + 1 }
    					\downKernel_{ n + 2 }
					-
						( 1 - \tfrac{c_n}{c_{n+1}} )
						I
            		\big)
					\downKernel_{ n + 1 }
			\\
			& = 
            		\big(
						\tfrac{c_{n+1}}{c_{n}}\upDownDist_{ n + 1 }
						\upDownKernel_{ n + 1 }
					-
						( \tfrac{c_{n+1}}{c_{n}} - 1 )
						\upDownDist_{ n + 1 }
            		\big)
            		\,
					\downKernel_{ n + 1 }
			\\
			& = 
            		\big(
						\tfrac{c_{n+1}}{c_{n}} \upDownDist_{ n + 1 }
					-
						( \tfrac{c_{n+1}}{c_{n}} - 1 )
						\upDownDist_{ n + 1 }
            		\big)            		\,
					\downKernel_{ n + 1 }
			\\
			& = 	
                    \upDownDist_{ n + 1 }
					\downKernel_{ n + 1 }
			.
	\end{align*}

	\noindent
	Therefore, \ref{down consistency of stationary measures} follows from \ref{up down stationary measures}.
	Claim \ref{action of stationary measures on density functions} follows from (\ref{definition of density functions}) and \ref{down consistency of stationary measures}: for $ n \ge | s | $,
	\begin{align*}
		\orange{\int_{ \stateSpace_n }
			( \density_s )_n
			\,
			d\upDownDist_n}
				& = 
                		\orange{\int_{ \stateSpace_n }
                			\downKernel_{ n, | s | }( \, \cdot \,, s )
                			\,
                			d\upDownDist_n}
				= 
                		\sum_{ | r | = n }
                			\upDownDist_n( r )
                			\downKernel_{ n, | s | }( r, s )
				= 
                		(
							\upDownDist_n
                    		\downKernel_{ n, | s | }
						)( s )
				= 
                		\upDownDist_{ | s | }( s )
				.
	\end{align*}

	The $ f = \density_s $ case of \ref{integrating against stationary measure} follows from \cref{prop expansion of density in eigenbasis}, (\ref{defn stationary measures}), and \ref{action of stationary measures on density functions}:
	$$
		[ \eigenfunction_\zeroVertex ]
		\density_s
			=
        		\upDownDist_{ | s | }( s )
			=
				\int_{ \stateSpace_n }
					( \density_s )_n
					\,
					d \upDownDist_n
			,
				\qquad
				| s | \le n
			.
	$$
	
	\noindent
	This identity then extends to $ \spaceOfDensityFunctions_n = \linearSpan\{ \density_s \}_{ | s | \le n } $ by linearity.
\end{proof}

We immediately obtain the desired convergence in distribution.
\begin{corollary}
	\label{cor ergodicity of chains}
	Let $ n \ge 0 $. For any initial condition, $\upDownChain_n( m )$ converges to $ \upDownDist_n $ in distribution~as~$ m \to \infty $.
\end{corollary}

\begin{proof}
	Let $ n \ge 0 $.
	Since the density functions $ \{ ( \density_s )_n \}_{ | s | \le n } $ span $ C( \stateSpace_n ) $, it suffices to show that the expected densities of
	$
		\{ \upDownChain_n( m ) \}_{ m \ge 0 }
	$
	converge to the expected densities with respect to $ \upDownDist_n $.
	This follows from Proposition \ref{prop density estimate} and Proposition \ref{prop properties of stationary measures}\ref{action of stationary measures on density functions}.
\end{proof}

\subsection{Large time and large size asymptotics}
\label{ssec:large_time_large_size}
We now consider the behavior of the chains when both the size of the objects and the number of steps are large.
Here, we cannot yet identify a limit object, but we can obtain an asymptotic estimate.
We will see that, under an appropriate scaling, our semigroups are comparable to the artificial\footnote{The semigroup identity holds but $ \spaceOfDensityFunctions $ is not a Banach space since it has a countable Hamel basis.} semigroup $ \{ \limitSemigroupOnGraph_t \}_{ t \ge 0 } $ defined on 
$
	\spaceOfDensityFunctions
		\subset
			C( \stateSpace )
$
\noindent
by
$$
	\limitSemigroupOnGraph_t
	\eigenfunction_s
		=
			e^{- t \generatorRates_{ | s | - 1 } }
			\,
			\eigenfunction_s
		,
			\qquad
			s \in \stateSpace
		.
$$

\noindent
Note that these operators are well-defined since the functions
$
	\{
		\eigenfunction_s
	\}_{ s \in \stateSpace }
$
are independent.

For convenience, we begin with a lemma.
The integer part of a real number $ z $ is denoted by $ \floor{ z } $.
\begin{lemma}
	\label{lemma exponential estimate}
    The following inequality holds:
    $$
    	| ( 1 - y )^{ \floor{x}} - e^{ -x y } |
    		\le
    			y e^{1 - x y } 
				\frac{ \max( 1, xy ) }{ 1 - y } 
    		,
    			\qquad
    			0 \le x,
    			\,
    			0 \le y < 1
    		.
    $$
\end{lemma}

\begin{proof}
	\noindent
	Let $ 0 \le x $ and $ 0 \le y < 1$.
	The Mean Value Theorem gives us that
    \begin{equation*}
		\big|
				\left(1-y\right)^{ \floor{ x }  }
    		-
            	e^{ -x y }
		\big| =
            		\big|
            				e^{
                				\floor{ x }
                				\ln\left( 1-y\right)
							}
					-
                		e^{ -x y }
            		\big|
			 \le
            		\big|
            				\floor{ x }
            				\ln\left(1-y\right)
                		+
                        	x y
            		\big|
            		\,e^{\max\big( 
                				\floor{ x }
                				\ln\left(1-y\right)
            				,\,           				
                    		-x y
                		\big)}
			.
      \end{equation*}
	The first factor here can be controlled using the classical inequalities $ \frac{ z - 1 }{ z } \le \ln z \le z - 1$ and $ 0 \le z - \floor{ z } \le 1 $ and the hypothesis.
	We obtain the upper and lower bounds
	\[ 
		-\frac{x y^2}{1-y} = - \frac{x y}{1-y} + xy
			\le
				\floor{ x } \ln(1-y) + x y 
			\le 
				-\floor{ x } y + xy
			\le 
				y
			\le 
				\frac{y}{ 1 - y }
			,
	\]
   which combine into
   \[\big|
            \floor{ x }
            \ln
            \left(
                    1
                -
                    y
            \right)
        +
            x y                                                                
	\big| 
		\le 
			\frac{y}{1-y} \max(1, \, x y) 
		.
	\]
    A similar approach to the maximum concludes the proof:
    \begin{equation*}%
		\max\big( 
				\floor{ x }
				\ln\left(1-y\right)
            	,\,           				
				-x y
			\big)
				\le
            		\max( 
            				-\floor{ x } y
                        	,\,           				
            				-x y
            		)
				\le
    				( 1 - x ) y
				\le
    				1 - x y
				.\qedhere
	\end{equation*}
\end{proof}

\begin{prop}
	\label{prop initial convergence}	
	There exists a sequence of constants
	$
		\{ \constantsInEstimate_m \}_{ m \ge 1 }
	$
	such that
	$$
		\Big\Vert
                \upDownOperator_n
    				^{ 
    					\floor{ t \generatorRates_n }
    				}
				( f )_n
    		-
        		(
    				\limitSemigroupOnGraph_t
        			f
    			)_n
		\Big\Vert_{ C( \stateSpace_n ) }
			\le
        		\frac{
        				\max( 1, t )
				}{
        				\generatorRates_n				
				}
        		\,
				\sum_{ | s | = 1 }^k
					| a_s |
					\constantsInEstimate_{ | s | }
					e^{ - t \generatorRates_{ | s | - 1 } }
	$$

	\noindent
	whenever
	$
		f 
			= 
				\sum_{ | s | \le k }
    				a_s
    				\eigenfunction_s
			,
	$
	$ t \ge 0 $,
	and $ n \ge k $.
\end{prop}

\begin{remark}
    \label{remark limit semigroup vs limiting process}	
	In the short-term, this result provides some useful scaling limit-type estimates.
	For example, if $ | s | = 1 $, the expansion $ \density_s = \eigenfunction_s + \upDownDist_1( s ) \eigenfunction_\zeroVertex $ gives us that
	$$
		\mbb E_{\mu_n}\big[ 
			\density_s( \upDownChain_n( \floor{ \generatorRates_n t } ) ) 
		\big]
			=
    			\upDownDist_1( s ) ( 1 - e^{- t \generatorRates_0 } )
                +
        		e^{- t \generatorRates_0 }
				{\textstyle \int_{ \stateSpace_n }}
        			\,
        			( \density_s )_n
					\,
					d \mu_n
				+
    				\max( 1, t )
					e^{ - t \generatorRates_0 }
					O(\generatorRates_n^{-1})
			,
				\qquad
				t \ge 0
			,				
	$$
	
	\noindent
	as $ n \to \infty $, where $ \mu_n $ is any distribution on $ \stateSpace_n $.
	In the long-term, this result will provide a crucial ingredient towards identifying a scaling limit.
    Indeed, it essentially establishes the semigroup convergence needed to apply \cref{theorem ethier kurtz path convergence} (if $ \generatorRates_n \to \infty $, the above norms converge to zero for all $ f \in \spaceOfDensityFunctions $). 
	What remains now is to address the fact that $ \{ \limitSemigroupOnGraph_t \}_{ t \ge 0 } $ is not a Feller semigroup on $ C( E ) $ for some compact metric space $ E $.
	This is done in \cref{section convergence} by introducing some additional structure.
	\end{remark}
\begin{proof}
	Let $ n \ge | s | $ and $ t \ge 0 $.
	The identity
    $
    	\upDownOperator_n
    		^{ 
    			\floor{ t \generatorRates_n }  
    		}
		( \eigenfunction_s )_n
			= 
                	( 1 - \commutationWeightProduct_{ | s |, n } )
						^{ 
								\floor{ t \generatorRates_n } 
						}
					( \eigenfunction_s )_n
    $ (see \cref{prop eigenbasis})
    and the estimate in \cref{lemma exponential estimate} (taking $ x = t \generatorRates_n $ and $ y = \generatorRates_{ | s | - 1 }/\generatorRates_n $) give us that
	\begin{align*}
		\Big\Vert
            	\upDownOperator_n
					^{ 
						\floor{ t \generatorRates_n }  
					}
				( \eigenfunction_s )_n
    		-
        		(
    				\limitSemigroupOnGraph_t
        			\eigenfunction_s
    			)_n
		\Big\Vert_{ C( \stateSpace_n ) }
				& =	
                		\left\Vert
                        \left(
                            ( 1 - \commutationWeightProduct_{ | s |, n } )
        						^{ 
        							\floor{ t \generatorRates_n }  
        						}
                		-
                                e^{- t \generatorRates_{ | s | - 1 } }
        				\right)
						( \eigenfunction_s )_n
                		\right\Vert_{ C( \stateSpace_n ) }
				\\
				& =	
                		\Big|
								\left(
            						1
            					-
            						\tfrac{
                            				\generatorRates_{ | s | - 1 }
            							}{
                            				\generatorRates_n				
            							}
            					\right)
            						^{ 
            							\floor{ t \generatorRates_n }  
            						}                				
                    		-
                                e^{- t \generatorRates_{ | s | - 1 } }
						\Big|
						\Vert
							( \eigenfunction_s )_n
						\Vert_{ C( \stateSpace_n ) }
				\\
				& \le
        			\frac{
						\generatorRates_{ | s | - 1 } 
						\, e^{1 - t \generatorRates_{ | s | - 1 } } 
					}{ 
						\generatorRates_n
					}  
        			\cdot
					\frac{
            			\max( 1, t \generatorRates_{ | s | - 1 } )
					}{ 
						1 - \frac{\generatorRates_{ | s | - 1 }}{\generatorRates_n}
					}  
					\cdot
					\max_{ |r|= | s | } 
    					\Vert
    							( \eigenfunction_r )_n
    					\Vert_{ C( \stateSpace_n ) }
				.
	\end{align*}
   Using \ref{assumption increasing generator rates}, the second fraction here can be bounded above by $ q'_{|s|} \max(1,t) $, where $ q'_{|s|} $ depends only on $| s |$.
	Setting
	$
		\constantsInEstimate_{|s|} 
			= 
				\generatorRates_{ | s | - 1 } 
				e q'_{|s|} \max_{ | r | = | s |} \Vert ( \eigenfunction_s )_n \Vert_{ C( \stateSpace_n ) }
	$
    and \orange{observing that $ \constantsInEstimate_0 = 0 $}
	then establishes the result whenever $ f $ is equal to some $ \eigenfunction_s $.
	This extends to general $ f $ by linearity. 
\end{proof}

\subsection{Separation distance}
\label{ssec:separation_distance_discrete}
We conclude our analysis in the discrete setting by considering the separation distance of the chains and their continuous-time variants.
In this section, we do actually need the $ \upOperator_n $ to be positive operators, or equivalently, the $ \upKernel_n $ to have nonnegative entries.
Therefore, we will now assume that the $ \{ \upDownChain_n \}_{ n \ge 0 } $ are up-down chains satisfying \ref{assumption finite state spaces} and \ref{assumption:commutation}.
In addition, \orange{throughout the section}, we assume that there exist sequences
$ \{ r_n \}_{ n \ge 1 } $
and
$ \{ s_n \}_{ n \ge 1 } $
with $ r_n, s_n \in \stateSpace_n $ satisfying the following condition for $ n \ge 1 $:
\begin{enumerate}[ label = (S\arabic*), leftmargin = \indentForAssumptions ]
	\item
	\label{assumption distant elements}
    $\upDownDist_n(s_n)>0$ and $ n = \min\{ m \ge 0 : p_n^{m}(r_n, s_n) > 0 \} $ (i.e.~it takes $ \upDownChain_n $ exactly $ n $ transitions to reach $ s_n $ from $ r_n $).
    
\end{enumerate}

\noindent
For certain results, we will also require that
\begin{enumerate}[ label = (S\arabic*), leftmargin = \indentForAssumptions, resume ]
    \item
	\label{assumption consistency rn}
	$ \downKernel_{ n + 1 }( r_{ n + 1 }, r_n ) = 1 $ for $ n \ge 1 $.
\end{enumerate}

\noindent
We note that these new conditions are typically straightforward to verify for concrete examples.

To begin, let us denote the separation distance of $ \upDownChain_n $ by
\begin{equation}\label{eq:def_separation}
	\sepDist_n(m) 
		= 
			\max_{ r, s \in \stateSpace_n, \, M_n(s) \ne 0 } 
				\bigg(
					1 - \frac{\upDownKernel_{ n }^m (r,s)}{M_n(s)}
				\bigg)
		,
			\qquad
			m \ge 0
		,
\end{equation}

\noindent
and the separation distance of its continuous-time variant by
\begin{equation*}%
    \sepDist_n^*(t) 
    	= 
    		\sup_{ r \in\stateSpace_n , f \in C_+(\stateSpace_n )} 
    			\bigg( 
    				1 - \frac{ ( e^{ t \discreteGenerators_n } f )( r ) }{ \int_{\stateSpace_n} f d\upDownDist_n}
    			\bigg)
		,
			\qquad
			t \ge 0
		.
\end{equation*}

\noindent
The following result provides explicit formulas for these quantities;
note that \ref{assumption consistency rn} is not yet required.

\begin{theorem}
	\label{thm:separation_distance}
    The following identities hold: %
  \begin{equation}
  \sepDist_n(m) 
  	= 
		\sum_{i=0}^{n-1} 
			\bigg(1-\frac{c_{i}}{c_n}\bigg)^{m} 
            \prod_{\substack{{0 \le j \le n-1} \\ j \ne i}}
					\frac{c_{j}}{c_{j}-c_{i}}
	,
		\qquad {m \ge 0}, \, n \ge 1
	,
	\label{eq:Deltanm}
  \end{equation}

  \noindent
  and
    \begin{equation*}
\sepDist_n^*(t) 
  	= 
		\sum_{i=0}^{n-1} 
			e^{ - \generatorRates_i t }
            \prod_{\substack{ 0 \le j \le n-1 \\ j \ne i}}
				\frac{c_{j}}{c_{j}-c_{i}}
	,
		\qquad t \ge 0, \, n \ge 1
	.
  \end{equation*}
\end{theorem}
\begin{proof}
  Fix $ n \ge 1 $.
  Let us show that the maximum in \eqref{eq:def_separation} is attained by taking $r=r_n$ and $s=s_n$ (for every $ m $).
  By iterating the commutation relation \eqref{eq:commutation_relations}, it can be seen that the powers 
  of the transition matrix $\upDownKernel_n=\upKernel_n \downKernel_{n-1}$ can be written as
  \orange{follows}\footnote{See \cite[Proposition 4.5]{fulmanCommutation} for a recursive description of the coefficients.}:
  for $m \ge0$, %
  \[p_n^m= \sum_{k=0}^{\min(m,n)} \alpha_{n,k,m} ( \downKernel_{n, n-k} \upKernel_{n-k, n} ),
  	\qquad
    \text{\orange{for some positive real numbers} } \alpha_{n,k,m}>0.
	\]

  \noindent
  Recalling from Assumption \ref{assumption distant elements} that $ \upDownKernel_n^{\orange{n-1}}( r_n, s_n ) = 0 $,
  it then follows from this formula that
  \[ ( \downKernel_{n, n-k} \upKernel_{n-k, n} )(r_n, s_n) 
  	=
		0
	,
		\qquad
		0 \le k \le n - 1
	.
  \]
  This vanishing of terms then leads to the simplification
  \[p_n^m(r_n, s_n)= \alpha_{n,n,m} ( \downKernel_{n, 0} \upKernel_{0, n} )(r_n, s_n),
  	\qquad
	m \ge n.
  \] 
  On the other hand, since $\stateSpace_0 = \{ \zeroVertex \}$, 
  for every $r,s$ in $\stateSpace_n$ we have
  \[ ( \downKernel_{n, 0} \upKernel_{0, n} )(r, s)
  = \downKernel_{n, 0}(r,\zeroVertex)
  \cdot \upKernel_{0, n} (\zeroVertex,s) = 1 \cdot M_n(s)=M_n(s).
  \]
  \orange{Using this, together with the nonnegativity of the up- and down-kernels,}
  we find that, for $m \ge n$, $r,s \in \stateSpace_n$, whenever $M_n(s) \ne 0$,
  we have
  \[\frac{p_n^m(r, s)}{M_n(s)} \ge \frac{\alpha_{n,n,m} \cdot ( \downKernel_{n, 0} \upKernel_{0, n} )(r, s)}{M_n(s)} = \alpha_{n,n,m} = \frac{p_n^m(r_n, s_n)}{M_n(s_n)}.
 \]
  
  \noindent
  This shows that the maximum in \eqref{eq:def_separation}
  is indeed attained for $r = r_n$ and $s = s_n$ when $ m \ge n $.
  For $ m < n $, this fact is an immediate consequence of Assumption \ref{assumption distant elements}, which implies that $ \upDownKernel_n^m( r_n, s_n ) = 0 $.

  Having established that $ \sepDist_n(m) = 1- ( \upDownKernel_n^m(r_n, s_n) / \upDownDist_n(s_n) )$, we proceed by applying Proposition 5.1 in \cite{Fulman2010separation} (we verify its hypothesis later).
  This result provides the explicit form
  \[1-\frac{p_n^m(r_n, s_n)}{M_n(s_n)} = \sum_{i=1}^n \la_i^m \prod_{j \ne i}\frac{1-\la_j}{\la_i-\la_j},\]
  where $\la_1,\dots\la_n$ are the eigenvalues of $\upDownOperator_n$ different from $1$.
  Since \cref{prop eigenbasis} identified these eigenvalues as $\la_i = 1-\omega_{i,n}=1-(\generatorRates_{i-1}/\generatorRates_n)$, the desired formula is obtained by reindexing.

  Let us now verify the hypothesis of \cite[Proposition 5.1]{Fulman2010separation}.
  The ergodicity of $ \upDownChain_n $ was established in \cref{cor ergodicity of chains}.
  Due to \cite[Remark 5.2, item (3)]{Fulman2010separation}, we can replace the reversibility requirement with the diagonalizability of $ \upDownOperator_n $, which was established in \cref{prop eigenbasis}.
  Finally, \cref{prop eigenbasis} shows that $ \upDownOperator_n $ has exactly $n+1$ distinct eigenvalues, one more than the distance between $r_n$ and $s_n$.

	Turning our attention to the second identity, fix $ t \ge 0 $ and $ n \ge 1 $. 
	Let $ r \in \stateSpace_n $ and $ f \in C_+( \stateSpace_n ) $.
	Combining \cref{prop functional form of sep dist} with the fact that $ r_n $ and $ s_n $ maximize (\ref{eq:def_separation}), we have that (recall (\ref{identity kernel from operator}) and (\ref{identity density as delta function}))
	$$
		1 - \frac{ ( \upDownOperator_n^m f )( r ) }{ \int_{\stateSpace_n} f d\upDownDist_n}
		\le
			\sepDist_n( m )
		=
    		1 - \frac{ ( \upDownOperator_n^m ( \density_{ s_n } )_n )( r_n ) }{ \int_{\stateSpace_n} ( \density_{ s _n } )_n d\upDownDist_n}
		,
			\qquad
			m \ge 0
		.
	$$
    \noindent
	Multiplying by $ e^{ - t \generatorRates_n } ( t \generatorRates_n )^m/m! $ (with the convention $ 0^0 = 1 $), summing over $ m $, and recalling that $A_n=   \generatorRates_n ( \upDownOperator_n - I) $, this becomes
	$$
		1 - \frac{ ( e^{ t A_n } f )( r ) }{ \int_{\stateSpace_n} f d\upDownDist_n}
		\le
			e^{ - t \generatorRates_n }
			\sum_{ m = 0 }^\infty
				\frac{( t \generatorRates_n )^m}{ m! }
				\sepDist_n( m )
		=
    		1 - \frac{ ( e^{ t A_n} ( \density_{ s_n } )_n )( r_n ) }{ \int_{\stateSpace_n} ( \density_{ s _n } )_n d\upDownDist_n}
		.
	$$
	Since this holds for arbitrary $r \in \stateSpace_n$ and $f \in  C_+( \stateSpace_n )$, we have that
\[  \sepDist_n^*( t )  \le 1 - \frac{ ( e^{ t A_n} ( \density_{ s_n } )_n )( r_n ) }{ \int_{\stateSpace_n} ( \density_{ s _n } )_n d\upDownDist_n}. \]
On the other hand, the reverse inequality can be established using \eqref{claim compare positive and nonnegative functions} and the positivity of $\int_{\stateSpace_n} ( \density_{ s_n } )_n d\upDownDist_n $ (see \cref{prop properties of stationary measures}\ref{action of stationary measures on density functions} \orange{and \ref{assumption distant elements}}).
We must therefore have equality:
\begin{equation}\sepDist_n^*( t ) =1 - \frac{ ( e^{ t A_n} ( \density_{ s_n } )_n )( r_n ) }{ \int_{\stateSpace_n} ( \density_{ s _n } )_n d\upDownDist_n}=e^{ - t \generatorRates_n }
			\sum_{ m = 0 }^\infty
				\frac{( t \generatorRates_n )^m}{ m! }
				\sepDist_n( m ).
				\label{eq:1mDeltakStar}
				\end{equation}
	Using now the expression for the separation distance in \eqref{eq:Deltanm}, we obtain the desired formula: 
	\begin{align*}
		\sepDist_n^*( t ) 
			& =
    			e^{ - t \generatorRates_n }
    			\sum_{ m = 0 }^\infty
    				\frac{( t \generatorRates_n )^m}{ m! }
        		\sum_{i=0}^{n-1} 
        			\bigg(1-\frac{\generatorRates_{i}}{\generatorRates_n}\bigg)^{m} 
        			\prod_{0 \le j \le n-1 \atop j \ne i} 
        				\frac{c_{j}}{c_{j}-c_{i}}
			\\
			& =
    			e^{ - t \generatorRates_n }
        		\sum_{i=0}^{n-1} 
					e^{ t \generatorRates_n ( 1-\frac{\generatorRates_{i}}{\generatorRates_n})}
        			\prod_{0 \le j \le n-1 \atop j \ne i} 
        				\frac{c_{j}}{c_{j}-c_{i}}
			\\
			& =
        		\sum_{i=0}^{n-1} 
					e^{ - t \generatorRates_{i} }
        			\prod_{0 \le j \le n-1 \atop j \ne i} 
        				\frac{c_{j}}{c_{j}-c_{i}}
			.\qedhere
	\end{align*}
\end{proof}

\orange{We note that the above formulas lead to asymptotic estimates whenever the rates are sufficiently well-behaved.}
An example is given in the following result.

\begin{prop}
	\label{prop convergence of sep dist}
	Suppose that $ \sum_{ n \ge 0 } \tfrac{1}{\generatorRates_n} < \infty $ and that $ \{ \generatorRates_{n+1} - \generatorRates_{n} \}_{ n \ge 0 } $ is an unbounded, nondecreasing sequence.
	Then we have the limits
	$$ 
    	\lim_{ n \to \infty }
    		\sepDist_n(\floor{ \generatorRates_n t})
      		=
    	\lim_{ n \to \infty }
    		\sepDist_n^*( t )
      		=
            	\sum_{i=0}^{\infty} 
            		e^{- t \generatorRates_i } 
            		\prod_{j = 0 \atop j \ne i}^\infty 
    					\frac{\generatorRates_{j}}{\generatorRates_{j}-\generatorRates_{i}}
    		,
    			\qquad
    			t > 0
    		.
	$$  

	\noindent
	Moreover, 
	for $ n \ge 0 $ we have the asymptotic description
    \begin{equation}
    	\sum_{ i = n + 1 }^{\infty} 
    		e^{- t \generatorRates_i } 
    		\prod_{j = 0 \atop j \ne i}^\infty 
				\frac{\generatorRates_{j}}{\generatorRates_{j}-\generatorRates_{i}}
			=
				o( e^{ - t \generatorRates_n } )
		,
			\qquad
			t \to \infty
      \label{eq:bound_remainder_Sep}
    \end{equation}
\end{prop}
\begin{proof}
	Fix $t>0$. Let us address the first limit (the second limit can be established similarly).
	Using \cref{thm:separation_distance}, the first limit above can be written as
      \begin{equation}\label{convergence to prove}
		\sum_{i=0}^{\infty} 
			\indicator( i < n )
			\bigg(1-\frac{\generatorRates_{i}}{\generatorRates_n}\bigg)^{\floor{\generatorRates_n t}} 
			\prod_{ j = 0 \atop j \ne i}^{n-1} 
				\frac{\generatorRates_{j}}{\generatorRates_{j}-\generatorRates_{i}}
  		\longrightarrow
        	\sum_{i=0}^{\infty} 
        		e^{- t \generatorRates_i } 
        		\prod_{j = 0 \atop j \ne i}^\infty 
					\frac{\generatorRates_{j}}{\generatorRates_{j}-\generatorRates_{i}}
		.
    	\end{equation}

	\noindent
	We will establish this using the dominated convergence theorem.
	It should be clear that the summands on the right are the limits of those on the left.
	{Therefore, it only remains to show} that the summands on the left are uniformly bounded in $n$ 
	by a summable sequence of $i$. 
	
	The first \orange{two factors} can be handled by the inequality $c_i<c_n$ and some standard estimates:
	\begin{equation}\label{bound power}
		0 
			\le 
		{\indicator( i < n )}
		\bigg( 
			1 - \frac{\generatorRates_{i}}{\generatorRates_n}
		\bigg)^{\floor{\generatorRates_n t}} 
			\le
        		\left( 
        			e^{ - \generatorRates_{i}/\generatorRates_n }
        		\right)^{\floor{\generatorRates_n t}} 
			\le
        		\left( 
        			e^{ - \generatorRates_{i}/\generatorRates_n }
        		\right)^{\generatorRates_n t - 1} 
			\le
        			e^{ 1 - \generatorRates_{i} t } 
			.
		\end{equation}
	Next, we consider the product of the $c_j/(c_j-c_i)$ for $j<i$.
	Here, we use the inequality $c_j<c_i$, the monotonicity of $ \{ c_{k+1} - c_k \}_{ k \ge 0 } $, and the summability of $ \{ \frac{1}{\generatorRates_k} \}_{ k \ge 0 } $:
	\begin{equation}\label{bound finite product}
		0
		\le
    		\prod_{j = 1 }^{ i - 1 }
        		\frac{\generatorRates_{j}}{\generatorRates_{i}-\generatorRates_{j}}	
		\le
    		\prod_{j = 1 }^{ i - 1 }
        		\frac{\generatorRates_{j}}{\generatorRates_{ i - j }-\generatorRates_{0}}	
		=
    		\prod_{j = 1 }^{ i - 1 }
        		\frac{ \generatorRates_{j} }{ \generatorRates_{j} - \generatorRates_{0} }
		=
    		\prod_{j = 1 }^{ i - 1 }
        		\left(1+\frac{ \generatorRates_{0} }{ \generatorRates_{j} - \generatorRates_{0} }\right)
				\le
			\exp\bigg( 
        			\sum_{j = 1 }^\infty 
        				\frac{ \generatorRates_{0} }{ \generatorRates_{j}-\generatorRates_{0} } 
				\bigg)
		<
			\infty
		.
	\end{equation}
	The remaining terms are the product of the $c_j/(c_j-c_i)$ for $j>i$.
	We claim that
	\begin{align} \label{bound infinite product}
		0
		\le
		\prod_{j = i + 1}^{{n-1}}
			\frac{\generatorRates_{j}}{\generatorRates_{j}-\generatorRates_{i}}
		\le
		\prod_{j = i + 1}^\infty 
			\left(1+\frac{\generatorRates_{i}}{\generatorRates_{j}-\generatorRates_{i}}\right)
    			\le 
    				\prod_{j = i + 1}^\infty 
                		e^{ \frac{\generatorRates_{i}}{\generatorRates_{j}-\generatorRates_{i}} }
    			=
    				\exp\bigg( 
    						\sum_{j = 1 }^\infty 
    							\frac{\generatorRates_{i}}{\generatorRates_{j + i}-\generatorRates_{i}} 
    					\bigg)=\exp(o(c_i))
				.
	\end{align}
    \orange{Since $c_j>c_i$ whenever $ j > i $, we only need to justify the limit}
    \begin{equation}\label{eq:Tech}
      \lim_{i \to +\infty} \sum_{j = 1 }^\infty                                      
            \frac{1}{\generatorRates_{j + i}-\generatorRates_{i}}  =0.
          \end{equation}
    For this, we shall apply the dominated convergence theorem.
    The termwise limits \orange{are zero} since, by hypothesis, $ \{ c_{k+1} - c_k \}_{ k \ge 0 } $ is unbounded.
	\orange{To identify a summable bound}, we use the monotonicity of $ \{ c_{k+1} - c_k \}_{ k \ge 0 } $ and the summability of $ \{ \frac{1}{\generatorRates_k} \}_{ k \ge 0 } $:
	\[
	\sum_{ j = 1 }^\infty
		\frac{1}{\generatorRates_{j + i}-\generatorRates_{i}} 
			\le 
	\sum_{ j = 1 }^\infty
			\frac{1}{\generatorRates_{j}-c_0}
			<
				\infty
			. 
	\]
	
	\noindent
    This establishes \eqref{eq:Tech} and thus \eqref{bound infinite product}.
	Combining \cref{bound power,bound finite product,bound infinite product} \orange{then yields our final bound}:
	\[\indicator( i < n )\,
			\bigg(1-\frac{\generatorRates_{i}}{\generatorRates_n}\bigg)^{\floor{\generatorRates_n t}} \,
			\prod_{ j = 0 \atop j \ne i}^{n-1} 
				\frac{\generatorRates_{j}}{\generatorRates_{j}-\generatorRates_{i}}
				\le 
					\exp\big(1-\generatorRates_i t +O(1) + o(\generatorRates_i)\big)
				= 
					e^{-\generatorRates_i (t-o(1))}
				.
	\]
	We remark that this bound is indeed independent of $n$.
	Moreover, since the $\generatorRates_i$ grow at least linearly ($ \generatorRates_{i+1} - \generatorRates_i$ is nondecreasing), this bound is summable in $i$ for any fixed $t>0$.
	This verifies that the dominated convergence theorem applies, establishing the limit in
	\eqref{convergence to prove}.
	
	For the final claim in the proposition, we must show that
	$$ 
    	\sum_{ i = n + 1 }^{\infty} 
    		e^{- t ( \generatorRates_i - \generatorRates_n ) } 
    		\prod_{j = 0 \atop j \ne i}^\infty 
				\frac{\generatorRates_{j}}{\generatorRates_{j}-\generatorRates_{i}}
			\xrightarrow[t \to \infty]{}
				0
		,
			\qquad
			n \ge 0
		.
	$$  

	\noindent
	Once again, we apply the dominated convergence theorem.
	Since $ c_i > c_n $ for $ i > n $, the termwise limits are each zero. 
	It thus only remains to identify a suitable bound.
	Here, we can reuse the estimates in \eqref{bound finite product} and \eqref{bound infinite product}.
	We see that
	$$
		e^{- t ( \generatorRates_i - \generatorRates_n ) } 
		\prod_{j = 0 \atop j \ne i}^\infty 
			\frac{\generatorRates_{j}}{\generatorRates_{j}-\generatorRates_{i}}
				\le
            		\exp( - ( \generatorRates_i - \generatorRates_n ) + O( 1 ) + o(c_i)),
					\qquad
					t \ge 1
				.
	$$
	
	\noindent
    For fixed $n$, this upper bound is independent of $t$ and summable in $i$.
    Thus, the dominated convergence theorem applies and \eqref{eq:bound_remainder_Sep} is proved.
\end{proof}

The following lemma will prepare us for the final result of the section.
Recall that $ \spaceOfDensityFunctions_n = \linearSpan \{\density_s\}_{|s| \le n }$ and define
$$
	\spaceOfDensityFunctions_n^+ 
		=
			\{ f \in \spaceOfDensityFunctions_n : ( f )_n > 0 \} 
		,
			\qquad
			n \ge 0
		.
$$

\begin{lemma}
	\label{lemma H+ properties}
	For $ n \ge 0 $, 
    $C_+( \stateSpace_n ) 
			=
            	\{(f)_n : f\in \spaceOfDensityFunctions_n^+\}
    $
    and
    $ \spaceOfDensityFunctions_n^+ \subseteq \spaceOfDensityFunctions_{ n + 1}^+ $.
\end{lemma}
\begin{proof}
	For the first claim, the reverse containment follows from definitions. 
	To establish the forward containment, we take $ g$ in $C_+( \stateSpace_n ) $ and observe that the function
    $
    	f 
			= 
			\sum_{ | s | = n }
					g( s ) 
					\density_s 
	$
    lies in \orange{$ \spaceOfDensityFunctions_n $} and satisfies $ ( f )_n = g > 0 $ 
    (recall \eqref{identity density as delta function}).

	Moving on to the second claim, observe that \cref{prop properties of density functions}\ref{action of down operators on density functions} can be extended linearly as follows:
		$$
        	\downOperator_{ n + 1 }
    		( f )_n
    			=	
                		( f )_{ n + 1 }
				,
					\qquad
					f \in \spaceOfDensityFunctions_n,
					\,
					n \ge 0
				.
		$$
	
	\noindent
	Let now $ n \ge 0 $ and $f$ be in $ \spaceOfDensityFunctions^+_n$. 
	Since $ ( f )_n $ is positive and $ \downOperator_{ n + 1 } $ is a transition operator, the above identity implies that $ ( f )_{ n + 1 } $ is also positive.
	Moreover, $f \in \spaceOfDensityFunctions_n \subseteq \spaceOfDensityFunctions_{ n + 1 } $.
	This concludes the proof.	
\end{proof}

We conclude this section by demonstrating an interplay between separation distance and intertwining -- namely, that the separation distance of a process can be computed using the dynamics of an intertwined process.
This leads to a certain monotonicity in the separation distances, which will be crucial later when we compute their limit.

\begin{prop}
	Suppose that \ref{assumption consistency rn} holds.
	Then we have the identities\footnote{\cref{prop properties of stationary measures}\ref{integrating against stationary measure} ensures that the denominators are nonzero.}
    \begin{equation}
    	\label{eqn recursive structure in sep dist}
        \sepDist_k^*(t) 
        	= 
        		\sup_{ r \in  \stateSpace_n, \, f \in \spaceOfDensityFunctions_k^+ } 
        			\bigg( 
        				1 - \frac{ ( e^{ t \discreteGenerators_n } ( f )_n )( r ) }{ [ \eigenfunction_\zeroVertex ] f}
        			\bigg)
    		,
				\qquad
				n \ge k,
				\,
				t \ge 0
			.
    \end{equation}
   Consequently, $ \sepDist^*_k( t ) $ is nondecreasing in $ k $.
	\label{prop recursive structure in sep dist}
\end{prop}

\begin{remark}
	One implication of this result is that each of these separation distances can be described by a variety of processes (by changing $ n $).
	Another is that one can fix a single process (i.e.~an $n$) and compute various separation distances by simply moving along the filtration
	$
    	\spaceOfDensityFunctions_0^+
    		\subset
            	\spaceOfDensityFunctions_1^+
    				\subset 
    				\ldots
    				\subset
            	\spaceOfDensityFunctions_n^+
			.
	$
\end{remark}

\begin{remark}
	We do not have an equivalent statement for the discrete-time chains since they do not exhibit intertwining.
\end{remark}

\begin{proof}
    {Let $ n \ge k $ and $t \ge 0 $}.
    Iterating \cref{prop properties of density functions}\ref{action of down operators on density functions} gives us that
    $
    	\downOperator_{ n, k }
		( f )_k
			=
				( f )_n
	$
	for {$ f $ in $ \spaceOfDensityFunctions_k $}.
	Similarly, we can iterate Assumption \ref{assumption intertwining} to obtain
    $ 
    	\discreteGenerators_n \downOperator_{ n, k }
    	=
    		\downOperator_{ n, k } \discreteGenerators_k
		,
    $
    which can be converted into the semigroup relation
    $
    	e^{ t \discreteGenerators_n } \downOperator_{ n, k }
    	=
    		\downOperator_{ n, k } e^{ t \discreteGenerators_k }
    $
    (see~\cite[Corollary 7.1]{krdrLeftmost} if needed).
	In summary,
$$
    	e^{ t \discreteGenerators_n } 
		( f )_n
			=
            	e^{ t \discreteGenerators_n } 
            	\downOperator_{ n, k }
        		( f )_k
        	=
        		\downOperator_{ n, k } 
				e^{ t \discreteGenerators_k }
        		( f )_k,
				\qquad 
				f \in \spaceOfDensityFunctions_k
			.		
	$$
	Since $ \downOperator_{ n, k } $ is a transition operator, this yields the bound
	\[
   		\inf_{  r\in \stateSpace_n } 
        				( e^{ t \discreteGenerators_n } ( f )_n )( r ) 
      \ge
        		\inf_{  u \in \stateSpace_k} 
        				 ( e^{ t \discreteGenerators_k } ( f )_k )( u ),
				\qquad 
				f \in \spaceOfDensityFunctions_k
			.		
	\]
	\noindent
	Combining this with the identities in
    \cref{prop properties of stationary measures}\ref{integrating against stationary measure}
    and {\cref{lemma H+ properties}}, we find that
		\begin{equation*}%
    		1 - \sepDist_k^*( t )
    			= 
            		\inf_{ \substack{  g \in C_+( \stateSpace_k ) \\  u \in \stateSpace_k } } 
            				\frac{ ( e^{ t \discreteGenerators_k } g )( u ) }{ \int_{\stateSpace_k} g \, d\upDownDist_k }
    			= 
            		\inf_{ \substack{ f \in \spaceOfDensityFunctions^+_k  \\  u \in \stateSpace_k  } } 
            				\frac{ ( e^{ t \discreteGenerators_k } ( f )_k )( u ) }{ \int_{\stateSpace_k} ( f )_k \, d\upDownDist_k }
    			\le 
            		\inf_{ \substack{ f \in \spaceOfDensityFunctions^+_k \\ r \in \stateSpace_n } } 
            				\frac{ ( e^{ t \discreteGenerators_n } ( f )_n )( r ) 
            				}{ 
                        		[ \eigenfunction_\zeroVertex ]
                        		f
            				}.
		\end{equation*}
	{To establish the first claim, it only remains to prove the reverse inequality.}
	To this end, observe from Assumption \ref{assumption consistency rn} that 
	$\downKernel_{n,k}(r_n,r_k)=1$, and as a result, $(\downOperator_{n,k} \, g) (r_n)=g(r_k)$ for any $g$ in $C( \stateSpace_k)$.
	In particular,
	\[( 
    		e^{ t \discreteGenerators_n } 
    		( f )_n
		)( r_n )
			=
            	( 
            		\downOperator_{ n, k } 
    				e^{ t \discreteGenerators_k }
            		( f )_k
        		)( r_n )
			=
            	( 
    				e^{ t \discreteGenerators_k }
            		( f )_k
        		)( r_k )
			,
				\qquad 
				f \in \spaceOfDensityFunctions_k
			.		\]
	Together with \cref{prop properties of stationary measures}\ref{integrating against stationary measure} and \cref{lemma H+ properties}, this implies that
	\begin{equation*}%
	\inf_{ \substack{ f \in \spaceOfDensityFunctions^+_k \\  r \in \stateSpace_n } } 
        				\frac{ ( e^{ t \discreteGenerators_n } ( f )_n )( r ) 
        				}{ 
                    		[ \eigenfunction_\zeroVertex ]
                    		f}
		\le
		\inf_{ \substack{ f \in \spaceOfDensityFunctions^+_k  } } 
        				\frac{ ( e^{ t \discreteGenerators_n } ( f )_n )( r_n ) 
        				}{ 
                    		[ \eigenfunction_\zeroVertex ]
                    		f}
		=
			\inf_{ \substack{ f \in \spaceOfDensityFunctions^+_k  } } 
        				\frac{ ( e^{ t \discreteGenerators_{k} } ( f )_k )( r_k ) 
        				}{ 
            				\int_{ \stateSpace_k }
            					( f )_k \,
            					d \upDownDist_k
            			}
		=
			\inf_{ \substack{ g \in C_+(\stateSpace_k)  } } 
        				\frac{ ( e^{ t \discreteGenerators_{k} } g )( r_k ) 
        				}{ 
            				\int_{ \stateSpace_k }
            					g \,
            					d \upDownDist_k
            			}
		.
	\end{equation*}
	\noindent
	We can therefore obtain the reverse inequality by using the nonnegativity of $\density_{s_k}$ and \eqref{eq:1mDeltakStar} to write

	\[\inf_{ \substack{ g \in C_+(\stateSpace_k)  } } 
    		\frac{ ( e^{ t \discreteGenerators_k } g )( r_k ) 
    		}{ 
        			\int_{ \stateSpace_k }
        				g \,
        				d \upDownDist_k
				}
				\le \lim_{\eps \to 0^+} \frac{ ( e^{ t \discreteGenerators_k } ( \density_{s_k}+\eps\density_\zeroVertex )_k )( r_k ) 
    		}{ 
        			\int_{ \stateSpace_k }
        				( \density_{s_k} +\eps\density_\zeroVertex )_k
        				d \upDownDist_k
        	}      
				=
				\frac{ ( e^{ t \discreteGenerators_k } ( \density_{s_k} )_k )( r_k ) 
    		}{ 
        			\int_{ \stateSpace_k }
        				( \density_{s_k} )_k
        				d \upDownDist_k}
				=
					1 - \sepDist_k^*( t )
				.
        	\]
    This establishes the first claim, along with the following identity, which will be useful later:
    \begin{equation}
      1 - \sepDist_k^*( t )
        = 	\inf_{ \substack{ g \in C_+(\stateSpace_k)  } } 
                \frac{ ( e^{ t \discreteGenerators_{\orange{k}} } g )( r_k ) 
                }{     
                    \int_{ \stateSpace_k }
                        g \,
                        d \upDownDist_k
                }
       .
      \label{inequality final bound for reversal}
    \end{equation}
	The second claim of the proposition follows from the containment assertion in \cref{lemma H+ properties}.
\end{proof}

\section{The scaling limit}
\label{section convergence}

In this section, we establish the convergence of our chains to a Feller process.
Our approach continues to be based on the analysis of transition operators.

\subsection{The limiting space and the extension of functions}
\label{section limiting space}

We now assume, in addition to \allAssumptions\!, that we have
	a compact metric space $ \limitSpace $,
	a map  
	$ 
		\inclusion
			\colon
				\stateSpace \to \limitSpace
			,
	$
	and a family of functions
	$
		\{
			\density_s^o
		\}_{ s \in \stateSpace }
			\subset 
				C( \limitSpace ) 
	$
	that satisfy the following conditions:
\begin{enumerate}[ label = (L\arabic*), leftmargin = \indentForAssumptions ]
	\item
	\label{assumption state space approximation}
	for every 
	$ 
		x \in \limitSpace 
	$
	there is a sequence 
	$ 
		\{ s_n \}_{ n \ge 0 }
	$
	such that 
	$ 
		s_n \in \stateSpace_n
	$
	and
	$
		\inclusion( s_n )
			\to
				x
			,
	$

	\item
	\label{assumption continuous density functions are dense} 
	the span of
	$
		\{
			\density_s^o
		\}_{ s \in \stateSpace }
	$
	is dense in
	$
		C( \limitSpace )
		,
	$ 
	\item
	\label{assumption continuous density functions are limits}
	for every $ s \in \stateSpace $, we have the convergence
	$$
		\big\Vert
    			( \density_s )_n
			-
				\density_s^o 
				\circ 
				\inclusion 
					\vert_{ \stateSpace_n }
		\big\Vert_{ C( \stateSpace_n ) }
			\xrightarrow[ n \to \infty ]{}
				0
			.
	$$

\end{enumerate}

These conditions imply that the density functions on $ \stateSpace $ can be `extended' to $ \limitSpace $.
Indeed, \ref{assumption state space approximation} and \ref{assumption continuous density functions are limits} lead to the limit
\begin{align}
\label{defn continuous density functions}
	\density_s^o( x ) 
		& =
			\lim_{ \substack{ 
					\inclusion( r ) \to x \\
					| r | \to \infty
				}}
				\density_s( r )
		,
			\qquad
			x \in \limitSpace,
			\,
			s \in \stateSpace
		.
\end{align}
Consequently, the functions
$
	\{
		\density_s^o
	\}_{ s \in \stateSpace }
$
can be seen as continuous analogues of the density functions.
We will call these the \emph{density functions (on $ \limitSpace $)} and
note that they are uniquely determined by (\ref{defn continuous density functions}).
In addition, they inherit several properties from the density functions on $ \stateSpace $:
\begin{enumerate}[ label = (\roman*) ]
	\item
	\label{continuous density functions lie in unit interval}
	$
    	0
			\le
				\density_s^o
			\le
				1
	$
	for all $ s \in \stateSpace $,

	\item
	\label{continuous zero density is constant}
    $
    	\density_\zeroVertex^o
    		\equiv
				1
			,
	$

	\item
	\label{Pieri rule for continuous density functions}
    we have the expansions
    $$
    	\density_s^o
    		=
    			\sum_{ | u | = n }
    				\density_s( u )
    				\,
    				\density_u^o
			,
				\qquad
				n \ge | s |
			,
	$$
  \item \label{item:filtration_core}
    the subspaces
    $
    	\core_n
    		=
    			\linearSpan	
    				\{
            			\density_s^o
            		\}_{ 
        					| s | = n
    					}
    $
	form a finite-dimensional filtration of
    $
    	\core
    		=
    			\linearSpan	
    				\{
            			\density_s^o
            		\}_{ s \in \stateSpace } 
			,
    $
    $$
    	\core_0
    		\subset
    			\core_1
    				\subset 
    				\ldots
    				\subset
    					\core
    						=
    							\bigcup_n
    								\core_n
    				.
    $$
\end{enumerate}
Due to the latter property, we say that $ \density_r^o $ is of `lower order' than $ \density_s^o $ whenever $ | r | < | s | $.

For the upcoming analysis, it will be useful to extend this analogy to other functions.
We do this by defining a linear operator
$ 
	\discreteToContinuousMap 
		\colon
        	\spaceOfDensityFunctions
		\to
        	\core
$ 
by
$ 
	\density_s 
		\mapsto 
			\density_s^o 
$
so that now each 
$
	f
		\in
        	\spaceOfDensityFunctions 
$
has a continuous analogue $ f^o = \discreteToContinuousMap f $.
A number of identities carry over from the discrete setting, including
\begin{equation}
    \label{eq:ho_In_Do_Basis}
	\eigenfunction_\zeroVertex^o \equiv 1,
	\qquad\quad
	\eigenfunction_s^o
		=
    		\sum_{ | r | \le | s | }
    			\upKernel_{ | r |, | s | }( r, s )
				\,
				\coeffEigInDensityBasis_{ | r |, | s | }
				\,
    			\density_r^o
		,
			\qquad
			s \in \stateSpace
		,
\end{equation}
\begin{equation}
	\label{eq:Do_In_ho_basis}
	\density_s^o
		=
    		\sum_{ | r | \le | s | }
    			\upKernel_{ | r |, | s | }( r, s )
				\,
				\coeffDensityInEigBasis_{ | r |, | s | }
				\,
    			\eigenfunction_r^o
		,
			\qquad
			s \in \stateSpace
		,
      \end{equation}

\noindent
and
\begin{equation*}
	\linearSpan	
		\{
			\eigenfunction_s^o
		\}_{
				| s | \le n
			}
        		=
        			\linearSpan	
        				\{
                			\density_s^o
                		\}_{
        						| s | \le n
        					}
        		,
        			\qquad
        			n \ge 0 
        		.
\end{equation*}

We will also need the projections induced by $ \inclusion $. 
As in Section \ref{subsection convergence theorems}, each $ n \ge 0 $ has an associated projection
$
	\projection_n
		\colon
				C( \limitSpace )
			\to
				C( \stateSpace_n )
$
given by
$$
	\projection_n
	f
		=
			f
			\circ
				\inclusion\vert_{ \stateSpace_n }			
		.
$$
We recall that these projections give rise to the following notion of convergence.
\begin{defn} 
\label{convergenceInDifferentSpaces}
    A sequence $ \{ f_n \} $ with 
    $ 
    	f_n 
    		\in 
				C( \stateSpace_n )
	$ 
	converges to 
    $ 
    	f
    		\in 
				C( \limitSpace )
	$ 
	(and we write $ f_n \to f $) if
    $$
    	\Vert
    		f_n
    	-	\pi_n
    		f
    	\Vert_{ C( \stateSpace_n ) }
    		\xrightarrow[ n \to \infty ]{}
    			0
    		.
    $$

\end{defn}

This convergence leads to a simple reformulation of Assumption \ref{assumption continuous density functions are limits}:
$
	( \density_s )_n
		\longrightarrow
			\density_s^o
$
for $ s \in \stateSpace $,
or more generally,
\begin{equation}
	\label{general reformulation of convergence assumption}
	( f )_n
		\longrightarrow
			f^o
		,
			\qquad
			f \in \spaceOfDensityFunctions
		.
\end{equation}

\noindent
We also have the following useful property.
It is a well-known consequence of \ref{assumption state space approximation} (see e.g.~\cite{BOpartitions, OlshAddJackParameter}), but we choose to include a proof for the sake of completeness.

\begin{lemma}
	\label{lem:norm approximation}
	Let $ f \in C( \limitSpace ) $.
	Then we have the convergence
	$
		\Vert
    			\projection_n
				f
		\Vert_{ C( \stateSpace_n ) }
				\to
                		\Vert
                				f
                		\Vert_{ C( \limitSpace ) }
	$
	as $ n \to \infty $.
	Consequently, if \orange{$f_n \to f $} %
	in the sense of Definition~\ref{convergenceInDifferentSpaces}, then
	$
		\Vert
				f_n
		\Vert_{ C( \stateSpace_n ) }
				\to
                		\Vert
                				f
                		\Vert_{ C( \limitSpace ) }
	$
	as $ n \to \infty $.
\end{lemma}

\begin{proof}
	Let $ f \in C( \limitSpace ) $.	Since $ \limitSpace $ is compact, there exists some $ x \in \limitSpace $ such that
	$
		\Vert
				f
		\Vert_{ C( \limitSpace ) }
			=
				| f( x ) |
			.
	$	
	Let 
	$ 
		\{ s_n \}_{ n \ge 0 }
	$
	be an $ \stateSpace $-approximation of $ x $, as in \ref{assumption state space approximation}.
	Writing
	$$
		| f( \inclusion( s_n ) ) |
			\le
        		\sup_{ r \in \stateSpace_n }
					| f( \inclusion( r ) ) |
			=
        		\sup_{ r \in \stateSpace_n }
					| 
            			( \projection_n f )( r )
					|
			\le
        		\sup_{ z \in \limitSpace }
					| 
            			f( z )
					|
	$$
	
	\noindent
	and taking the limit as $ n \to \infty $ yields the first claim.
	The second follows immediately.
\end{proof}

\subsection{The limiting process}
\label{ssec:limiting_process}
We now construct the semigroup of our limiting process and establish the desired semigroup convergence.
The candidate for the limiting semigroup is supplied by Proposition \ref{prop initial convergence}: we seek operators that satisfy
$
	\limitSemigroupOnMetricSpace( t )
	f^o
		=
			( \limitSemigroupOnGraph_t f )^o
		,
$
where $Q_t$ is the operator on $\spaceOfDensityFunctions$ defined by
$
	\limitSemigroupOnGraph_t
	\eigenfunction_s
		=
			e^{- t \generatorRates_{ | s | - 1 } }
			\,
			\eigenfunction_s
$
for $ s \in \stateSpace $.
The following result shows that such a semigroup exists and that it is indeed the limiting semigroup.

\begin{prop}
	\label{prop limiting semigroup and generator}
	Suppose that $ \generatorRates_n \to \infty $ and $ \{ \eps_n \}_{ n \ge 0 } $ is a positive sequence converging to zero such that
	$
		\eps_n
		\generatorRates_n 
		\to
			1
		.
	$
	The following statements hold:
	\begin{enumerate}[ label = (\roman*) ]
		\item
		\label{claim existence of limit}
		for every $ f \in C( \limitSpace ) $ and $ t \ge 0 $, the limit
		\begin{equation*} 
			\limitSemigroupOnMetricSpace( t )
			f
				=
					\lim_{ n \to \infty }
                		\upDownOperator_n
    						^{ \floor{ t/\eps_n } }
                        \projection_n
                        f
		\end{equation*}
		
		\noindent
		exists in the sense of Definition \ref{convergenceInDifferentSpaces} \orange{and defines a family of linear operators},
		
		\item
		\label{claim description of limit}
		these operators satisfy
		$
			\limitSemigroupOnMetricSpace( t )
			\discreteToContinuousMap
				=
					\discreteToContinuousMap
					\limitSemigroupOnGraph_t
				,
		$
		or equivalently, 
		\begin{equation}\label{eq:Tt}
			\limitSemigroupOnMetricSpace( t )
			\eigenfunction_s^o
				=
    				e^{ - t \generatorRates_{ | s | - 1 } }
					\eigenfunction_s^o
				,
					\qquad
						s \in \stateSpace,
						\,
						t \ge 0,
		\end{equation}

		\item
		\label{claim Feller semigroup}
		these operators form a Feller semigroup on
		$
			C( \limitSpace )
			,
		$

		\item
		\label{claim generator and core}
		$ \core $ is a core for $ \pregenerator $, the generator of this semigroup, and the action of $ \pregenerator $ on $ \core $ is given by
        $$
        	\pregenerator
        	\eigenfunction_s^o
        		=
            		- \generatorRates_{ | s | - 1 }
            		\eigenfunction_s^o
        		,
        			\qquad
					s \in \stateSpace
        		,
        $$
		and
		
		\item
		\label{claim convergence of generators}
		we have the convergence of generators
		$$
			\discreteGenerators_n
			( f )_n
				\xrightarrow[ n \to \infty ]{}
					\pregenerator
					f^o
				,
					\qquad
					f \in \spaceOfDensityFunctions
				,
		$$
		
		\noindent
        in the sense of Definition \ref{convergenceInDifferentSpaces}.
		
	\end{enumerate}
\end{prop}

\begin{remark}

Although we have several characterizations of our semigroup and generator, the only one that can be easily used as a definition is the one in \ref{claim existence of limit}.
Indeed, the diagonal descriptions in \ref{claim description of limit} and \ref{claim generator and core} are inadequate since the functions $ \{ \eigenfunction_s^o \}_{ s \in \stateSpace } $ need not be independent. 
Similarly, we cannot use the limit in \ref{claim convergence of generators} without establishing the nontrivial fact that $ \discreteGenerators_n ( f )_n $ converges to zero whenever $ f^o = 0 $.

\end{remark}

\begin{proof}

	Suppose that
	$
		\generatorRates_n 
			\to
				\infty
			.
	$
	We will first prove the result for the sequence
	$ 
		\eps_n 
			= 
				\generatorRates_n^{ -1 }
	$
	and then later generalize to arbitrary sequences.

	Fix $ t \ge 0 $ and $ s \in \stateSpace $.
	Using the contractivity of
	$
    	\upDownOperator_n,
	$
	Proposition \ref{prop initial convergence}, and \ref{assumption continuous density functions are limits}, we have that
	\begin{align*}
		\left\Vert
            	\upDownOperator_n
					^{ 
						\floor{ t \generatorRates_n }  
					}
				\projection_n
				\eigenfunction_s^o
    		-
				\projection_n
				e^{ - t \generatorRates_{ | s | - 1 } }
				\eigenfunction_s^o
		\right\Vert_{ C( \stateSpace_n ) }
			& \le
                    \left\Vert
            				\upDownOperator_n
            					^{ 
            						\floor{ t \generatorRates_n }  
            					}
							(
                				\projection_n
                				\eigenfunction_s^o
    						-
                				( \eigenfunction_s )_n
							)
            		\right\Vert_{ C( \stateSpace_n ) }
			+
                    \left\Vert
            				\upDownOperator_n
            					^{ 
            						\floor{ t \generatorRates_n }  
            					}
            				( \eigenfunction_s )_n
                		-
                    		(
                				\limitSemigroupOnGraph_t
                    			\eigenfunction_s
                			)_n
            		\right\Vert_{ C( \stateSpace_n ) }
			\\
			& \qquad + 
                    \left\Vert
                    		(
                				\limitSemigroupOnGraph_t
                    			\eigenfunction_s
                			)_n
						-
            				e^{ - t \generatorRates_{ | s | - 1 } }
            				\projection_n
            				\eigenfunction_s^o
            		\right\Vert_{ C( \stateSpace_n ) }
			\\
			& \le
            		\left\Vert
                				\projection_n
                				\eigenfunction_s^o
    						-
                				( \eigenfunction_s )_n
            		\right\Vert_{ C( \stateSpace_n ) }
			+
                    O( \generatorRates_n^{ -1 } )
			+
    				e^{ - t \generatorRates_{ | s | - 1 } }
                    \left\Vert
                    		(
                    			\eigenfunction_s
                			)_n
						-
            				\projection_n
            				\eigenfunction_s^o
            		\right\Vert_{ C( \stateSpace_n ) }
			\\
			& \quad
				\longrightarrow
					0
			.
	\end{align*}

	\noindent
	This shows that the limit in \ref{claim existence of limit} exists and is given by \eqref{eq:Tt} for each $ \eigenfunction_s^o $.
	Extending linearly, this limit yields a linear operator $ \limitSemigroupOnMetricSpace( t ) \colon \core \to \core $ satisfying 	$
		\limitSemigroupOnMetricSpace( t )
		\discreteToContinuousMap
			=
				\discreteToContinuousMap
				\limitSemigroupOnGraph_t
			,
	$
	establishing \ref{claim description of limit}.
	We will extend this operator to all of $ C( \limitSpace ) $ by continuity.

	Let $ f \in \core $.
	Using \ref{claim existence of limit} and the contractivity of
	$
    	\upDownOperator_n,
	$
	we can write
	\begin{align*}
		\Vert
			\projection_n
			\limitSemigroupOnMetricSpace( t )
			f
		\Vert_{ C( \stateSpace_n ) }
			\le
        		\Vert
        			\projection_n
        			\limitSemigroupOnMetricSpace( t )
        			f
				-
    				\upDownOperator_n
    					^{ 
    						\floor{ t \generatorRates_n }
    					}
        				\projection_n
        				f
        		\Vert_{ C( \stateSpace_n ) }
			+
        		\Vert
    				\upDownOperator_n
    					^{ 
    						\floor{ t \generatorRates_n }
    					}
        				\projection_n
        				f
        		\Vert_{ C( \stateSpace_n ) }
			\le
				o( 1 )
			+
        		\Vert
        				\projection_n
        				f
        		\Vert_{ C( \stateSpace_n ) }	
			.			
	\end{align*}
	
	\noindent
	Taking the limit as $ n \to \infty $ using Lemma~\ref{lem:norm approximation}, we find that $ \limitSemigroupOnMetricSpace( t ) $ is a contraction on $ \core $.
	It follows that $ \limitSemigroupOnMetricSpace( t ) $ extends uniquely to a contraction on the topological closure of $ \core $, which according to \ref{assumption continuous density functions are dense}, is $ C( \limitSpace ) $.
	We will denote this extension with the same symbol.

	To show that this extension satisfies \ref{claim existence of limit}, we use a density argument.
	Given $ f \in C( \limitSpace ) $, we take an approximating sequence $ \{ f_k \}_{ k \ge 1 } \subset \core $ satisfying $ f_k \to f $.
	We then use the contractivity of $ \upDownOperator_n $, $ \projection_n $, and $ \limitSemigroupOnMetricSpace( t ) $ to obtain (for all $ n, k $) 
	\begin{align*}
		\left\Vert
            	\upDownOperator_n
					^{ 
						\floor{ t \generatorRates_n }  
					}
				\projection_n
				f
    		-
				\projection_n
				\limitSemigroupOnMetricSpace( t )
				f
		\right\Vert_{ C( \stateSpace_n ) }
			& \le
                    \left\Vert
                        	\upDownOperator_n
            					^{ 
            						\floor{ t \generatorRates_n }  
            					}
            				\projection_n
							(
                				f
    						-
                				f_k
							)
            		\right\Vert_{ C( \stateSpace_n ) }
			+
                    \left\Vert
            				\upDownOperator_n
            					^{ 
            						\floor{ t \generatorRates_n }  
            					}
                				\projection_n
                				f_k
    						-
                				\projection_n
                				\limitSemigroupOnMetricSpace( t )
								f_k
            		\right\Vert_{ C( \stateSpace_n ) }
			\\
			& \qquad +
                    \left\Vert
            				\projection_n
            				\limitSemigroupOnMetricSpace( t )
							(
                				f_k
    						-
                				f
							)
            		\right\Vert_{ C( \stateSpace_n ) }
			\\
			& \le
                    2
                    \left\Vert
                				f
    						-
                				f_k
            		\right\Vert_{ C( \limitSpace ) }
			+
                    \left\Vert
            				\upDownOperator_n
            					^{ 
            						\floor{ t \generatorRates_n }  
            					}
                				\projection_n
                				f_k
    						-
                				\projection_n
                				\limitSemigroupOnMetricSpace( t )
								f_k
            		\right\Vert_{ C( \stateSpace_n ) }
			.
	\end{align*}
	
	\noindent
	Observe now that the right-hand side becomes zero if we let $ n \to \infty $ and then $ k \to \infty $ since \ref{claim existence of limit} was already established on $ \core $.
	Therefore, the left-hand side converges to zero as $ n \to \infty $, and \ref{claim existence of limit} holds.

	To see that $ \limitSemigroupOnMetricSpace( t ) $ is conservative, we make use of (\ref{eq:Tt}) and (\ref{eq:ho_In_Do_Basis}):
	$ 
		\limitSemigroupOnMetricSpace( t ) 1
			=
				\limitSemigroupOnMetricSpace( t )
				\eigenfunction_\zeroVertex^o
			=
				\eigenfunction_\zeroVertex^o
			=
				1
			.
	$
	To establish positivity, we take $ f \in \core $ with $ f \ge 0 $ and fix $ x \in \limitSpace $.
	Let
	$ 
		\{ s_n \}_{ n \ge 0 }
	$
	be an $ \stateSpace $-approximation of $ x $, as in \ref{assumption state space approximation}.
	Using \ref{claim existence of limit} and the positivity of each
	$
		\upDownOperator_n
	$
	and
	$
		\projection_n,
	$
	we have that
	\begin{align*}
		(
    		\limitSemigroupOnMetricSpace( t )
    		f
		)
		( \inclusion( s_n ) )
			= 
            \left(
					(
            			\projection_n
                		\limitSemigroupOnMetricSpace( t )
                		f
            		)( s_n )
			-
    				(
        				\upDownOperator_n
        					^{ 
        						\floor{ t \generatorRates_n }
        					}
            			\projection_n
            			f
            		)( s_n )
			\right)
			+
    				(
        				\upDownOperator_n
        					^{ 
        						\floor{ t \generatorRates_n }
        					}
            			\projection_n
            			f
            		)( s_n )
			\ge
				o( 1 )
			.
	\end{align*}	

	\noindent
	Taking the limit as $ n \to \infty $ reveals that $ ( \limitSemigroupOnMetricSpace( t ) f ) ( x ) \ge 0 $, from which it follows that $ \limitSemigroupOnMetricSpace( t ) $ is a positive operator on $ \core $.
	Since the constant functions lie in $ \core $, this positivity will extend to all of $ C( \limitSpace ) $.
	Indeed, suppose $ f \in C( \limitSpace ) $ with $ f \ge 0 $.
	Using \ref{assumption continuous density functions are dense}, we can take an approximating sequence $ \{ f_n \}_{ n \ge 1 } \subset \core $ satisfying $ f_n \to f $.
	Setting
	$ 
		g_n
			=
        		f_n 
        		+ 
        		\Vert
        			 f - f_n
        		\Vert
        		\cdot 
        		1
			,
	$
	we obtain a nonnegative approximating sequence:
	$ 
		g_n \in \core,
	$
	$
		g_n \to f,
	$
	and
	$
		g_n \ge f.
	$
	The positivity on $ \core $ and the boundedness on $ C( \limitSpace ) $ then imply that
	$$
		0
			\le
            		\limitSemigroupOnMetricSpace( t ) 
					g_n
			\longrightarrow
            		\limitSemigroupOnMetricSpace( t ) 
					f
			.
	$$
	
	\noindent
	As a result, $ \limitSemigroupOnMetricSpace( t ) $ is positive on $ C( \limitSpace ) $.

	The relation 
	$
		\limitSemigroupOnMetricSpace( t )
		\discreteToContinuousMap
			=
				\discreteToContinuousMap
				\limitSemigroupOnGraph_t
	$
	shows that 
	$
		\{
			\limitSemigroupOnMetricSpace( t )\vert_\core
		\}_{ t \ge 0 }
	$
	inherits the semigroup identity from
	$
		\{
			\limitSemigroupOnGraph_t
		\}_{ t \ge 0 }.
	$
	The semigroup identity can then be extended from $ \core $ to all of $ C( \limitSpace ) $ by continuity.
	Writing
	$$
		\lim_{ t \to 0 }
			\limitSemigroupOnMetricSpace( t )
			\eigenfunction_s^o
				=
            		\lim_{ t \to 0 }
                        e^{ - t \generatorRates_{ | s | - 1 } }
            			\eigenfunction_s^o
				=
            			\eigenfunction_s^o
				,
					\qquad
					s \in \stateSpace
				,
	$$
	
	\noindent
	shows that the semigroup is strongly continuous on $ \core $.
	This property can be extended to $ C( \limitSpace ) $ by a density argument.
	It follows that the semigroup is Feller, establishing \ref{claim Feller semigroup}.

	Writing
	$$
		\lim_{ t \to 0 }
			\frac{ \limitSemigroupOnMetricSpace( t ) - I }{ t }
			\eigenfunction_s^o
				=
            		\lim_{ t \to 0 }
                        \frac{ e^{ - t \generatorRates_{ | s | - 1 } } - 1 }{ t }
            			\eigenfunction_s^o
				=
                        - \generatorRates_{ | s | - 1 }
            			\eigenfunction_s^o
				,
					\qquad
					s \in \stateSpace
				,
	$$
	
	\noindent
	shows that $ \core $ lies in the domain of $ \pregenerator $ and the formula in \ref{claim generator and core} holds.
	In particular, $ \core $ is invariant under $ \pregenerator $. 
    Together with the fact that $ \core $ is dense in $ C( \limitSpace ) $, it follows from \cref{prop:dense_invariant_core} that $ \core $ is a core for $ \pregenerator $, establishing \ref{claim generator and core}.

	Moving on to the generator convergence, we fix $ s \in \stateSpace $ and $ n \ge | s | $.
	Applying Proposition \ref{prop eigenbasis}, the identity in \ref{claim generator and core}, and (\ref{defn extended commutation weights}), we can compute (recall that $\discreteGenerators_n=\generatorRates_n(\upDownOperator_n-I)$)
	\begin{align*}
		\left\Vert
        	\discreteGenerators_n
			( \eigenfunction_s )_n
		-
			\projection_n
			\pregenerator
			\eigenfunction_s^o
		\right\Vert_{ C( \stateSpace_n ) }
			=
                    \left\Vert
                    	\generatorRates_n
						( - \commutationWeightProduct_{ | s |, n } )
        				( \eigenfunction_s )_n
            		+
        				\projection_n
        				\generatorRates_{ | s | - 1 }
        				\eigenfunction_s^o
            		\right\Vert_{ C( \stateSpace_n ) }
			=
                    \generatorRates_{ | s | - 1 }
                    \left\Vert
        				( \eigenfunction_s )_n
            		-
        				\projection_n
        				\eigenfunction_s^o
            		\right\Vert_{ C( \stateSpace_n ) }
			.
	\end{align*}

	\noindent
	Taking the limit as $ n \to \infty $ using (\ref{general reformulation of convergence assumption}) establishes \ref{claim convergence of generators} for each $ \eigenfunction_s $.
	The general case follows by linearity.

	We have proven the result for $ \eps_n = \generatorRates_n^{ -1 } $.
	Let us now handle the general case. 
	Suppose that
	$ 
		\eps_n
		\,
		\generatorRates_n
			\to 
				1
			.
	$
	Taking $ f \in \spaceOfDensityFunctions $, we can write
	\begin{align*}
    	\Vert
    		\eps_n^{ -1 }
    		(
    			\upDownOperator_n
    		-
    			I
    		)
    		( f )_n
    	-
    		\discreteGenerators_n
    		( f )_n
    	\Vert_{ C( \stateSpace_n ) }
			& = 
            	\Vert
            		( \eps_n \generatorRates_n )^{ -1 }
    				\generatorRates_n
            		(
            			\upDownOperator_n
            		-
            			I
            		)
            		( f )_n
            	-
            		\discreteGenerators_n
            		( f )_n
            	\Vert_{ C( \stateSpace_n ) }
			\\
			& = 
            	\Vert
            		( \eps_n \generatorRates_n )^{ -1 }
    				\discreteGenerators_n
            		( f )_n
            	-
            		\discreteGenerators_n
            		( f )_n
            	\Vert_{ C( \stateSpace_n ) }
			\\
			& = 
            	|
					( \eps_n \generatorRates_n )^{ -1 }
				-
					1
				|
				\Vert	
    				\discreteGenerators_n
            		( f )_n
            	\Vert_{ C( \stateSpace_n ) }
			.
	\end{align*}
	
	\noindent
	Using the previously established case of \ref{claim convergence of generators} together with Lemma~\ref{lem:norm approximation}, we find that the norms on the right-hand side are bounded.
	Consequently, the norm on the left hand side converges to zero as $ n \to \infty $.
	Combining this with the previously established case of \ref{claim convergence of generators}, we obtain the generator convergence
	$$
		\eps_n^{ -1 }
		(
			\upDownOperator_n
		-
			I
		)
		( f )_n
			\longrightarrow
				\pregenerator
				f^o
			,
				\qquad
				f \in \spaceOfDensityFunctions
			,
	$$

	\noindent
	where $ \pregenerator $ is the generator identified in the previous case.
	Applying now Theorem \ref{theorem ethier kurtz modes of convergence}, we obtain the associated semigroup convergence
	$$
		\upDownOperator_n^{ \floor{ t/\eps_n } }
        \projection_n
        f
        	\longrightarrow
    			\limitSemigroupOnMetricSpace( t )
    			f
			,
				\qquad
				f \in C( \limitSpace ), 
				\,
				t \ge 0
			,
	$$	
	
	\noindent
	where $ \limitSemigroupOnMetricSpace( t ) $ is the semigroup constructed in the previous case.
	This concludes the proof.
\end{proof}

Our main convergence result now follows from Proposition \ref{prop limiting semigroup and generator} and Theorem \ref{theorem ethier kurtz path convergence}.
\begin{corollary}
	\label{corol:ConvProcesses}
	
	In addition to the hypothesis of Proposition \ref{prop limiting semigroup and generator}, suppose that
	the distributions of
	$
		\inclusion( \upDownChain_n( 0 ) )
	$
	converge to $ \mu $.
	Then there exists a Feller process $ \limitProcess $ in $ \limitSpace $ with 
	initial distribution $ \mu $, 
	generator $ \pregenerator $ (from Proposition \ref{prop limiting semigroup and generator}), 
	and sample paths in $ D( [ 0, \infty ), \limitSpace )  $ satisfying the path convergence
	$$
		\big(
			\inclusion( \upDownChain_n ( \floor{ t/\eps_n } ) )
		\big)_{ t \ge 0 }
			\Longrightarrow
        		\big(
        				\limitProcess( t )
        		\big)_{ t \ge 0 }
			.
	$$
\end{corollary}

For the remainder of \cref{section convergence}, we will assume that $ \generatorRates_n \to \infty $, denote by $ \{\limitSemigroupOnMetricSpace( t ) \}_{ t \ge 0 } $ and $ \pregenerator $ the objects in Proposition \ref{prop limiting semigroup and generator}, and denote by $ \limitProcess $ the limiting process in \cref{corol:ConvProcesses}.
We now investigate the properties of these objects.
To begin, we provide a condition for ensuring that the paths of $ \limitProcess $ are continuous,
and \orange{consequently}, that $ \limitProcess $ is a diffusion.
\begin{prop}
  \label{prop:path_continuity}
	Let $ \rho $ denote the metric on $ \limitSpace $.
	Suppose that
	$$
		\sup_{ r, s \in \stateSpace_n, \, \upDownKernel_n( r, s ) > 0 } 
			\rho( \inclusion( r ), \inclusion( s ) )
				\xrightarrow[ n \to \infty ]{}
					0
				.
	$$
	Then $ \limitProcess $ is almost surely continuous.
\end{prop}

\begin{proof}
	For $ n \ge 0 $, let $ j_n $ be the largest jump of $ ( \inclusion( \upDownChain_n( \floor{ t/\eps_n } ) ) )_{ t \ge 0 } $.
	Then 
	\begin{align*}
		j_n
			=
				\sup_{ t \ge 0 }
					\rho(
							\inclusion( \upDownChain_n (\floor{ t/\eps_n }) )
						,
							\inclusion( \upDownChain_n (\floor{ t^-/\eps_n }) )
						)
			& \le
        		\sup_{ \substack{ r, s \in \stateSpace_n \\ \upDownKernel_n( r, s ) > 0 } } 
					\rho(
							\inclusion( r )
						,
							\inclusion( s )
						)
			.						
	\end{align*}

	\noindent
	Applying the hypothesis, we find that $ j_n \to 0 $ almost surely.
	The result now follows from Corollary \ref{corol:ConvProcesses} and \cite[Chapter 3, Theorem 10.2]{ethierKurtzBook}.
\end{proof}

\subsection{Intertwining and a formula for the generator}
In this section, we continue the analogy of \cref{section limiting space} by `extending' the down-kernels $ \{ \downKernel_n \} $ into our new setting.
This will lead to some natural kernels from $ \limitSpace $ to the $ \stateSpace_n $, 
intertwining relations, 
a triangular description of the generator, 
and
a constructive version of Assumption \ref{assumption continuous density functions are dense}.

Our starting point is equation (\ref{defn continuous density functions}), written in terms of the down-kernels:
\begin{equation}
	\label{equation density as limit down kernel}
	\density_s^o( x ) 
		=
			\lim_{ \substack{ 
					\inclusion( r ) \to x 
					\\
					| r | \to \infty
				}
				}
				\downKernel_{ | r |, n }( r, s )
		,
			\qquad
			x \in \limitSpace,
			\,
			| s | = n
		.
\end{equation}

\noindent
This equation suggests that $ \density_s^o( x ) $ can be viewed as $ \downKernel_{ \infty, n }( x, s ) $, motivating us to consider the function
\begin{equation}
\label{continuous density functions produce kernels}
	\intertwiningKernel_n( x, s )
		\ldef
			\density_s^o( x )
		,
			\qquad
			x \in \limitSpace,
			\,
			| s | = n
		.
\end{equation}

\noindent
The properties of the density functions on $ \limitSpace $ (items \ref{continuous density functions lie in unit interval}-\ref{item:filtration_core} below \eqref{defn continuous density functions})
can then be used to show that for each $ n \ge 0 $ the above formula indeed defines a kernel
$
	\intertwiningKernel_n
$
from
$
	\limitSpace
$
to
$
	\stateSpace_n
$
and that these kernels satisfy
\begin{equation}
	\label{consistency of samples}
	\intertwiningKernel_n
		=
			\intertwiningKernel_{ n + 1 }
			\downKernel_{ n + 1, n }
		,
			\qquad
			n \ge 0
		.
\end{equation}
 
\noindent
We conclude that these kernels are analogues of the down-kernels.
As such, we can interpret the quantity $ \intertwiningKernel_n( x, s ) = \density_s^o( x ) $ as the \emph{density of $ s $ in $ x $}, justifying our earlier language for the functions $ \{ \density_s^o \}_{ s \in \stateSpace } $.

This analogy carries over to the associated transition operators: the operator associated to $ \intertwiningKernel_n $, which we denote by
$
	\intertwiningOperator_n
		\colon
				C( \stateSpace_n )
			\to
				C( \limitSpace )
		,
$
can be viewed as $ \downOperator_{ \infty, n } $, a natural limit of down-operators.
This operator is given by
\begin{equation}
	\label{defn intertwining operator}
	\intertwiningOperator_n
	f
		=
			\sum_{ | u | = n }
				f( u )
				\,
				\density_u^o
		,
			\qquad
			f \in C( \stateSpace_n )
		,
\end{equation}

\noindent
or by
\begin{equation}
	\label{describing intertwining kernel with density functions}
	\intertwiningOperator_n
	( g )_n
		=
			g^o
		,
			\qquad
			g \in \spaceOfDensityFunctions_n
		.
\end{equation}
The latter description is an analogue of Proposition \ref{prop properties of density functions}\ref{action of down operators on density functions} and can be established on the basis $ \{ \density _s \}_{ | s | \le n } $ using property \ref{Pieri rule for continuous density functions} of the density functions on $ \limitSpace $.
\bigskip

In the following result, we show that the intertwining structure in \ref{assumption intertwining} carries over to the limit.
As a consequence, we obtain an analogue of the triangular descriptions in \cref{prop action of transition operators}.
\begin{prop}
	\label{prop algebraic identities in limit}
	Let $ n \ge 0 $.
	The following statements hold: 
	\begin{enumerate}[ label = (\roman*) ]
		\item
		\label{claim generator commutation relation}
		$
			\pregenerator
			\intertwiningOperator_n
				=
        			\intertwiningOperator_n
        			\discreteGenerators_n
				,
		$
		
		\item
		\label{claim semigroup intertwining}
		$
			\limitSemigroupOnMetricSpace( t )
			\intertwiningOperator_n
				=
        			\intertwiningOperator_n
        			e^{ t \discreteGenerators_n }
		$
		for $ t \ge 0 $, and

    	\item
    	\label{claim generator on density functions}
    	the generator $ \pregenerator $ can also be described by the following formula
    	$$
    		\pregenerator
    		\density_s^o
    			=
        			-\generatorRates_{ | s | - 1 }
    				\Big(
                		\density_s^o
        			-
        				\sum_{ | r | = | s | - 1 }
    						\density_r^o
    						\,
            				\upKernel_{ | r | }( r, s )
    				\Big)
    			,
    				\qquad
    				s \in \stateSpace
    			.
    	$$
	\end{enumerate}
\end{prop}

\begin{proof}

	Let $ n \ge 0 $ and $ | s | \le n $.
    Using
    \eqref{describing intertwining kernel with density functions},
	Proposition \ref{prop limiting semigroup and generator}\ref{claim generator and core},
    (\ref{defn extended commutation weights}),
	and
	Proposition \ref{prop eigenbasis}, we have that
    \begin{multline}
	\label{evaluating generator at discrete eigenfunctions}
		\pregenerator
		\intertwiningOperator_n
		( \eigenfunction_s )_n
			 = 
        		\pregenerator
        		\eigenfunction_s^o
			 = 
        		- \generatorRates_{ | s | - 1 }
        		\eigenfunction_s^o
			 = 
				-\generatorRates_n
				\commutationWeightProduct_{ | s |, n }
        		\intertwiningOperator_n
				( \eigenfunction_s )_n
			 = 
        		\intertwiningOperator_n
				\generatorRates_n
				( \upDownOperator_n - I )
				( \eigenfunction_s )_n
			 = 
        		\intertwiningOperator_n
				\discreteGenerators_n
				( \eigenfunction_s )_n
			.
	\end{multline}
	Recalling that the functions
	$
		\{
			( \eigenfunction_s )_n
		\}_{ | s | \le n }
	$
	span
	$
		C( \stateSpace_n )
	$
	(see Proposition \ref{prop expansion of density in eigenbasis}) establishes the first claim.
	The intertwining relation in \ref{claim semigroup intertwining} then follows from the generator commutation relation by applying Corollary 7.1 in \cite{krdrLeftmost}.

 	The action of the generator on the density functions can be computed directly using 
  	\eqref{describing intertwining kernel with density functions}, 
	\ref{claim generator commutation relation},
  	Proposition \ref{prop action of transition operators},
	and 
	\eqref{defn extended commutation weights}:
	given $ s \in \stateSpace $, we take $ n \ge | s | $ and compute
	\begin{align*}
		\pregenerator
		\density_s^o
			= 
                	\pregenerator
					\intertwiningOperator_n
                	( \density_s )_n
			=
                	\intertwiningOperator_n
                    \discreteGenerators_n
                	( \density_s )_n
			& =
        			\intertwiningOperator_n
        			\generatorRates_n
    				(
                		\upDownOperator_n
    				-
    					I
    				)
                	( \density_s )_n
			\\
			& =
        			\intertwiningOperator_n
        			\generatorRates_n
        			\commutationWeightProduct_{ | s |, n }
    				\Big(
                		- ( \density_s )_n
        			+
        				\sum_{ | r | = | s | - 1 }
    						( \density_r )_n
    						\,
            				\upKernel_{ | r | }( r, s )
    				\Big)
			\\
			& =
        			\generatorRates_{ | s | - 1 }
    				\Big(
                		-  \density_s^o
        			+
        				\sum_{ | r | = | s | - 1 }
    						\density_r^o
    						\,
            				\upKernel_{ | r | }( r, s )
    				\Big)
			.
	\end{align*}

	\noindent
	This formula completely describes $ \pregenerator $ since 
	$
		\core
		=
			\linearSpan
			\{
            	\density_s^o
            \}_{ s \in \stateSpace }
	$
	is a core for $ \pregenerator $
	(see Proposition \ref{prop limiting semigroup and generator}\ref{claim generator and core}).%
\end{proof}

Our final result of this section further demonstrates that the intertwining kernels are natural objects.
\begin{prop}
	The following convergence holds for any $f$ in $ C( \limitSpace )$:
	$$
		\intertwiningOperator_n
		\projection_n f
			=
		\sum_{ | s | = n }
			f( \inclusion( s ) )
			\,
			\density_s^o
    			\xrightarrow[ n \to \infty ]{}
    				f.
	$$
	
	\noindent
	Consequently, for any probability measure $ \mu $ on $ \limitSpace $, we have the weak convergence
	$
		\mu \intertwiningKernel_n \circ \inclusion\vert_{ \stateSpace_n }^{-1}
			\to
				\mu
	$
\end{prop}

\begin{remark}
	The first convergence provides a constructive companion to Assumption \ref{assumption continuous density functions are dense} by explicitly approximating a function in $ C( \limitSpace ) $ by functions in $ \core $. 
	Note that this approximation resembles classical constructions, such as Bernstein's approximation of continuous functions by polynomials.
\end{remark}

\begin{remark}
	The second convergence implies that $ \intertwiningKernel_n $ can be viewed as an approximating kernel, taking $ x \in \limitSpace $ to a random point in $ \stateSpace_n $ that is close to $ x $ (under the $ \limitSpace $ metric) with high probability.
    This random approximation of $x$ specializes to well-known constructions, 
    namely Vershik--Kerov central measures on partitions (for the down-kernel of Section~\ref{ssec:BO-partitions}), Kingman's paintbox construction~\cite{kingman1978partition} (for the down-kernel of Section~\ref{section crp partitions}), its ordered variant due to Gnedin~\cite{gnedin1997} (for the down-kernel of Section~\ref{ssec:composition}), 
    and permuton/graphon samples when considering the down-kernels
    of \cref{section permutation example,sec:graph}.
\end{remark}

\begin{proof}
	Let $ g \in \spaceOfDensityFunctions $.
	Since $ g $ lies in $ \spaceOfDensityFunctions_n $ for large $ n $, we can apply \eqref{describing intertwining kernel with density functions} to write (recall that $ \intertwiningOperator_n $ is contractive)
	$$
		\norm{
		\intertwiningOperator_n
		\projection_n 
		g^o
		-
		g^o
		}_{ C( \limitSpace ) }
			=
        		\norm{
            	\intertwiningOperator_n
            	( \projection_n g^o - ( g )_n )
        		}_{ C( \limitSpace ) }
			\le
        		\norm{
            	\projection_n g^o - ( g )_n
        		}_{ C( \stateSpace_n ) }
				\qquad
				\text{for large $ n $}
			.
	$$
	Making use of the convergence in \eqref{general reformulation of convergence assumption} establishes the first claim for functions in $ \core $.
	The extension to $ C( \limitSpace ) $ can be handled by a density argument.
	The second claim follows from the first by writing 
	\begin{align*}
		\int_{ \limitSpace }
			f
			\,
			d ( \mu \intertwiningKernel_n \circ \inclusion \vert_{ \stateSpace_n }^{ -1 } )
			= 
				\int_{ \stateSpace_n }
					f \circ \inclusion\vert_{ \stateSpace_n } 
        			d ( \mu \intertwiningKernel_n )
			& = 
        		\sum_{ | s | = n }
        			f( \inclusion( s ) )
    				\int_\limitSpace
						\intertwiningKernel_n( x, s )
    					\,
    					d\mu( x )
			\\
			& = 
				\int_\limitSpace
					\intertwiningOperator_n
					\projection_n f
					d\mu
			\to
            	\int_{ \limitSpace }
            		f
            		d \mu
			,
				\qquad
				f \in C( \limitSpace )
			.\qedhere
	\end{align*}
\end{proof}

\subsection{Large time behavior}
\label{ssec:process_large_time}
We now investigate the behavior of $ \limitProcess $ for large time.
We begin with an analogue of Proposition \ref{prop density estimate}.

\begin{prop}
	\label{prop densities of limit process}
	Let
	$
		f
			= 
				a_\zeroVertex
				\eigenfunction_\zeroVertex^o
			+
				\sum_{ | s | = j }^k
    				a_s
    				\eigenfunction_s^o
	$
	for some $ k \ge j > 0 $.
	Then there exists some $ B_f > 0 $ such that
	$$
		\bigg|
    		\mbb E
    		\left[
    				f( \limitProcess( t ) )
    		\right]
		-
    		a_\zeroVertex
		\bigg|
			\le
        		B_f
                e^{ - t \generatorRates_{ j - 1 } }
			,
				\qquad
				t \ge 0
			,
	$$
	\noindent
	for any initial distribution.
	Consequently, we have the convergence (for any initial distribution)
	\[ 
		\mbb E[ \density_s^o( \limitProcess(t) ) ]
			\xrightarrow[ t \to \infty ]{}   
				\upDownDist_{ | s | }( s )
			,
				\qquad
				s \in \stateSpace
			.
	\]
\end{prop}

\begin{proof}

	We follow the proof of Proposition \ref{prop density estimate}.
	Let $ \nu $ be the initial distribution of $ \limitProcess $, $ s \in \stateSpace $, and $ t \ge 0 $.
	Using (\ref{eq:Tt}), we obtain the estimate 
    \begin{align*}
    	\big|
        	\mbb E
        	\left[
        			\eigenfunction_s^o
        			( 
        				\limitProcess( t )
        			)
        	\right]
    	\big|
    		= 
            	\left|
    				\int_\limitSpace
        				\limitSemigroupOnMetricSpace( t )
        				\eigenfunction_s^o
        				\,
        				d \nu
            	\right|
    		\le 
        		\Vert
                        \limitSemigroupOnMetricSpace( t )
        				\eigenfunction_s^o
        		\Vert_{ C( \limitSpace ) }
    		= 
				e^{ - t \generatorRates_{ | s | - 1 } }
        		\Vert
    				\eigenfunction_s^o
        		\Vert_{ C( \limitSpace ) }
			.
    \end{align*}

    \noindent
    Given now
	$
		f 
			= 
				a_\zeroVertex
				\eigenfunction_\zeroVertex^o
			+
				\sum_{ | s | = j }^k
    				a_s
    				\eigenfunction_s^o
			,
	$
	we can write (recall that $ \eigenfunction_\zeroVertex^o \equiv 1 $)
    \begin{align*}
    	\Big|
    		\mbb E
    		\left[
    				f( \limitProcess( t ) )
    		\right]
    	-
    		a_\zeroVertex
    	\Big|
			& =
            	\bigg|
    				\sum_{ | s | = j }^k
        				a_s
                    	\mbb E
                    	\left[
            				\eigenfunction_s^o
                			( 
                				\limitProcess( t )
                			)
                    	\right]
            	\bigg|
			\le
    				\sum_{ | s | = j }^k
        				| a_s |
                    	\left|
                        	\mbb E
                        	\left[
                				\eigenfunction_s^o
                    			( 
                    				\limitProcess( t )
                    			)
                        	\right]
                    	\right|
			\\
			& \le
				\sum_{ | s | = j }^k
    				| a_s |
    				e^{ - t \generatorRates_{ | s | - 1 } }
            		\Vert
        				\eigenfunction_s^o
            		\Vert_{ C( \limitSpace ) }
			\le
				e^{ - t \generatorRates_{ j - 1 } }
				\sum_{ | s | = j }^k
    				| a_s |
            		\Vert
        				\eigenfunction_s^o
            		\Vert_{ C( \limitSpace ) }
			.
    \end{align*}

	\noindent
    Taking $ B $ to be the above sum establishes the inequality.
    Recalling that $ \generatorRates_{ j - 1 } $ is positive shows that the expectations should converge to $ a_\zeroVertex $ as $ t \to \infty $.
	In the case when $ f = \density_s^o $, this coefficient can be identified from the expansion in (\ref{eq:Do_In_ho_basis}) (recall that $ \coeffDensityInEigBasis_{ 0, j } \equiv 1 $).
\end{proof}

As before, we are interested in identifying these limiting densities as coming from a distribution.
To this end, we first observe that they come from a functional.

\begin{prop}
	
	\label{prop coefficient functional}
    There exists a unique bounded linear functional $ [ \eigenfunction_\zeroVertex^o ] \colon C( \limitSpace ) \to \mbb R $ satisfying
    \begin{equation}
        \label{claim coefficient functional on density functions}
        [ \eigenfunction_\zeroVertex^o ] \density_s^o
        		=
        			\upDownDist_{ | s | }( s )
				,
					\qquad
					s \in \stateSpace
				.
	\end{equation}
	\noindent	
	Moreover, this functional is contractive, positive, and satisfies the following conditions:
    \begin{enumerate}[ label = (\roman*) ]

		\item
		\label{claim coefficient functional extends discrete one}
		$
			[ \eigenfunction_\zeroVertex^o ] f^o
    			=
			[ \eigenfunction_\zeroVertex ]
			f
		$
		for $ f \in \spaceOfDensityFunctions $
        (in particular, $
        	[ \eigenfunction_\zeroVertex^o ] \eigenfunction_s^o
        		=
        			\indicator( s = \zeroVertex )
        $
        for $ s \in \stateSpace $),

        \item
		\label{claim coefficient functional as a limit}
        $
        	[ \eigenfunction_\zeroVertex^o ] f
        		=
        			\lim_{ t \to \infty }
        				{\textstyle \int_\limitSpace }
           					\limitSemigroupOnMetricSpace( t )
           					f
        					\,
        					d \nu
        $
        for $ f \in C( \limitSpace ) $ and any probability measure $\nu$ on $\limitSpace$,
		and
		
		\item
		\label{claim weak convergence of stationary measures}
    	$
			[ \eigenfunction_\zeroVertex^o ] f
				=
        			\lim_{ n \to \infty }
        				{\textstyle \int_{ \stateSpace_n } }
            				\projection_n f \,
            				d \upDownDist_n
    	$
    	for $ f \in C( \limitSpace ) $.
	\end{enumerate}
\end{prop}

\begin{remark}
	We cannot define a functional by specifying its values on $ \{ \density_s^o \}_{ s \in \stateSpace } $ and then extending by continuity since these functions need not be independent.
	The same holds for the family $ \{ \eigenfunction_s^o \}_{ s \in \stateSpace } $.
\end{remark}

\begin{proof}
	Let $ \nu $ be a probability measure on $\limitSpace$.
	\cref{prop densities of limit process} implies that the following formula defines a linear functional from $ \core $ to $ \mbb R $ that is independent of $ \nu $:
    $$
    	[ \eigenfunction_\zeroVertex^o ] f
    		=
    			\lim_{ t \to \infty }
    				{\textstyle \int_\limitSpace }
       					\limitSemigroupOnMetricSpace( t )
       					f
    					\,
    					d \nu
				.
    $$
 	The special case in \cref{prop densities of limit process} then reveals that this functional satisfies \eqref{claim coefficient functional on density functions}.
	A direct computation shows that $ [ \eigenfunction_\zeroVertex^o ] $ inherits the contractivity of the semigroup:
	$$
		\Vert
			f
		\Vert_{ C( \limitSpace ) }
			\ge
					\Vert
        					\limitSemigroupOnMetricSpace( t )
        					f
					\Vert_{ C( \limitSpace ) }
			\ge
        		\left|
        				{\textstyle \int_\limitSpace }
           					\limitSemigroupOnMetricSpace( t )
           					f
    					\,
    					d \nu
        		\right|
			\xrightarrow[ t \to \infty ]{}
        		\left|
            		[ \eigenfunction_\zeroVertex^o ]
            		f
        		\right|
			.
	$$
    Hence, it has a unique contractive extension to $ C( \limitSpace ) $, which we continue to denote by $ [ \eigenfunction_\zeroVertex^o ] $.
    This functional is uniquely defined by \eqref{claim coefficient functional on density functions} since $ \core $ is dense in $ C( \limitSpace ) $. This establishes the first claim.

    We have already seen that \ref{claim coefficient functional extends discrete one} holds whenever $ f $ is given by some $ \density_s $ (see \cref{prop properties of stationary measures}).
    This identity extends to $ \spaceOfDensityFunctions $ by linearity.
    Similarly, we have already established \ref{claim coefficient functional as a limit} for functions in $ \core $, and the extension to $ C( \limitSpace ) $ follows from a density argument.
    Using \eqref{general reformulation of convergence assumption} and \cref{prop properties of stationary measures}\ref{integrating against stationary measure}, we can write for $ f \in \spaceOfDensityFunctions $,
	\begin{align*}
		\int_{ \stateSpace_n }
			\projection_n f^o \,
			d \upDownDist_n
				=
            		\int_{ \stateSpace_n }
            			\projection_n f^o - ( f )_n \,
            			d \upDownDist_n
                    +
            		\int_{ \stateSpace_n }
            			( f )_n \,
            			d \upDownDist_n
				=
            		o( 1 )
                    +
            		[ \eigenfunction_\zeroVertex ] f
				,
					\qquad
					n \to \infty
				.	
	\end{align*}
	This establishes \ref{claim weak convergence of stationary measures} on $ \core $ (recall \ref{claim coefficient functional extends discrete one}) and a density argument (using the contractivity of the $ \projection_n $) extends it to all of $ C( \limitSpace ) $.
    The positivity of $ [ \eigenfunction_\zeroVertex^o ] $ is also inherited from the semigroup: if $ f \in C( \limitSpace ) $ is nonnegative, then so is each $ \limitSemigroupOnMetricSpace( t ) f $ and
	\begin{equation*}
		0
			\le
				{\textstyle \int_\limitSpace }
   					\limitSemigroupOnMetricSpace( t )
   					f
					\,
					d \nu
			\xrightarrow[ t \to \infty ]{}
					[ \eigenfunction_\zeroVertex^o ]
    				f
			.
			\qedhere
	\end{equation*}
\end{proof}

Having identified a suitable functional, we proceed by applying the Riesz--Markov--Kakutani representation theorem \cite[Theorem 2.22]{kallenbergTheBible}.
Since
$
	[ \eigenfunction_\zeroVertex^o ]
$
is a positive, continuous linear functional on $ C( \limitSpace ) $, there is a unique finite measure $ \limitDist $ on $ \limitSpace $ satisfying
\begin{equation}
	\label{claim coefficient description}
	\int_{ \limitSpace }
			f
			\,
		d \limitDist
			=
				[ \eigenfunction_\zeroVertex^o ]
					f
				,
					\qquad
					f \in C( \limitSpace )
				.
\end{equation}
\noindent
The following result verifies that $ \limitDist $ satisfies the desired property and provides some alternative characterizations of it.
\begin{prop}
    \label{prop:Stationary Measure of Limit Process}
    Each of the following statements holds and characterizes the finite measure $ \limitDist $.
    \begin{enumerate}[ label = (\roman*) ]

    	\item
    	\label{claim finite dimensional description}
    	$
        	\int_{ \limitSpace }
        		\density_s^o
    			\,
        		d \limitDist
        			=
                    	\upDownDist_{ | s | }( s )
    	$
    	for $ s \in \stateSpace $,

    	\item
    	\label{claim limit distribution}
    	$ \limitProcess ( t ) $ converges to $ \limitDist $ in distribution as $ t \to \infty $ for any initial condition,
    	\item
    	\label{claim stationary for limit}
    	$ \limitDist $ is a stationary distribution of $ \limitProcess $,

    	\item
    	\label{claim limit of stationary distributions}
    	$ \upDownDist_n \circ \inclusion\vert_{ \stateSpace_n }^{ -1 } \to \limitDist $ weakly,
    	and

    	\item
    	\label{claim projections are stationary}
    	$
        	\limitDist
    		\intertwiningKernel_n
        		=
        			\upDownDist_n
    	$
    	for $ n \ge 0 $.
    \end{enumerate}
\end{prop}

\begin{proof}
	To begin, note that a finite measure on $ \limitSpace $ is uniquely determined by its integration of the density functions on $ \limitSpace $, since their span is dense in $ C( \limitSpace ) $.
    Therefore, the claim regarding \ref{claim finite dimensional description} follows from \eqref{claim coefficient functional on density functions} and \eqref{claim coefficient description}.
	The claim regarding \ref{claim limit distribution} follows from \cref{prop coefficient functional}\ref{claim coefficient functional as a limit} and \eqref{claim coefficient description}.
	The claim regarding \ref{claim stationary for limit} is a well-known consequence of \ref{claim limit distribution}.
	To establish the claim regarding \ref{claim limit of stationary distributions}, we use \cref{prop coefficient functional}\ref{claim weak convergence of stationary measures} and \eqref{claim coefficient description} to write
	\begin{align*}
		\int_{ \limitSpace }
			f
			\,
			d ( \upDownDist_n \circ \inclusion \vert_{ \stateSpace_n }^{ -1 } )
			= 
				\int_{ \stateSpace_n }
					f \circ \inclusion\vert_{ \stateSpace_n } 
        			\,
        			d \upDownDist_n
			= 
				\int_{ \stateSpace_n }
					\projection_n f
        			\,
        			d \upDownDist_n
			\to
                [ \eigenfunction_\zeroVertex^o ] f
			=
            	\int_{ \limitSpace }
            			f
            			\,
            		d \limitDist
			,
				\qquad
				f \in C( \limitSpace )
			.
	\end{align*}

	\noindent
	To establish the claim regarding \ref{claim projections are stationary}, let $ \mu $ be a finite measure on $ \limitSpace $ and use  \ref{claim finite dimensional description} and \eqref{continuous density functions produce kernels} to write
	\begin{equation*}
		\upDownDist_n( s ) - ( \mu \intertwiningKernel_n )( s )
			= 
				\int_{ \limitSpace }
					\density_s^o
        			\,
        			d \limitDist
				-
					\int_\limitSpace
						\intertwiningKernel_n( x, s )
						\,
						d\mu( x )
			= 
				\int_{ \limitSpace }
					\density_s^o
        			\,
        			d \limitDist
				-
					\int_\limitSpace
						\density_s^o
						\,
						d\mu
			,
				\qquad
				s \in \stateSpace_n
			.
			\qedhere
	\end{equation*}
\end{proof}

\subsection{Separation distance}
\label{ssec:separation_distance_continuous}
We conclude our study of $ \limitProcess $ with a result concerning its separation distance, which we denote by
\begin{equation*}
  \sepDist_\limitProcess(t) 
  	= 
		\sup_{ x \in \limitSpace, \, f \in C_+( \limitSpace ) } 
			\bigg( 
				1 - \frac{ ( \limitSemigroupOnMetricSpace( t ) f )( x ) }{ \int_\limitSpace f d\limitDist}
			\bigg)
	,
		\qquad
		t \ge 0
	.
\end{equation*}
{This} result extends \cref{prop recursive structure in sep dist} to the current setting, including an inequality that compliments the one in \cref{thm general sep dist of a limit}.
This allows us to identify $\sepDist_\limitProcess $ as the limit of the separation distances of the continuous-time discrete chains.

\begin{theorem}
	\label{theorem discrete and continuous sep dist}

	Suppose that the $ \{ \upDownChain_n \}_{ n \ge 0 } $ are up-down chains satisfying \ref{assumption finite state spaces} and \ref{assumption:commutation} and that Assumptions \ref{assumption distant elements}, \ref{assumption consistency rn}, and \ref{assumption state space approximation}--\ref{assumption continuous density functions are limits} hold.
	Then we have the identities
    \begin{equation}
    	\label{eqn discrete and continuous sep dist}
        \sepDist_n^*(t) 
        	= 
        		\sup_{ x \in E, f \in H_n^+ } 
        			\bigg( 
        				1 - \frac{ ( \limitSemigroupOnMetricSpace( t ) f^o )( x )
							 }{ [ \eigenfunction_\zeroVertex^o ] f^o }
        			\bigg)
    		,
				\qquad
				n \ge 0,
				\,
				t \ge 0
			.
    \end{equation}
    Consequently, as $ n \to \infty $ we have the monotonic convergence
	$$
        \sepDist_n^*( t )
        \orange{\,\nearrow\,}
				\sepDist_\limitProcess(t) 
			,
				\qquad
				t \ge 0
			.
	$$
\end{theorem}

\begin{remark}
Whenever this theorem and \cref{prop convergence of sep dist} both apply, $ \sepDist_\limitProcess $ is also a scaling limit of the separation distances of the discrete-time chains, and an explicit expression for this limit can be found in \cref{prop convergence of sep dist}.
\end{remark}

\begin{proof}
	\orange{We follow the proof of \cref{prop recursive structure in sep dist}}.
	Let $ t \ge 0 $ and $ n \ge 0 $.
	Equation (\ref{describing intertwining kernel with density functions}) and \cref{prop algebraic identities in limit}\ref{claim semigroup intertwining} give us the identity
	\begin{equation}\label{eq:Tt_fo}
    	\limitSemigroupOnMetricSpace( t ) f^o
			=
            	\limitSemigroupOnMetricSpace( t )
				\intertwiningOperator_n
				( f )_n
        	=
    			\intertwiningOperator_n
    			e^{ t \discreteGenerators_n }
				( f )_n, \qquad f \in \spaceOfDensityFunctions^+_n.
	\end{equation}
	Since $ \intertwiningOperator_n $ is a transition operator, this leads to the bound
		\begin{equation}
			\label{inequality discrete and continuous minima}
    		\inf_{  x\in E } 
        		( \limitSemigroupOnMetricSpace( t ) f^o )( x ) 
              \ge
				\inf_{  u  \in \stateSpace_n }
    				( e^{ t \discreteGenerators_n } ( f )_n )( u )
				, 
					\qquad f \in \spaceOfDensityFunctions^+_n.
		\end{equation}
	Combining this with \cref{prop recursive structure in sep dist} and \cref{prop coefficient functional}\ref{claim coefficient functional extends discrete one}, we obtain the inequality
	\begin{equation*}%
		1 - \sepDist_n^*( t )
			 = 
        		\inf_{ \substack{  f \in \spaceOfDensityFunctions^+_n \\ | u | = n} }
        				\frac{ ( e^{ t \discreteGenerators_n } ( f )_n )( u ) }{ [ \eigenfunction_\zeroVertex ] f }
			\le 
        		\inf_{ \substack{ f \in \spaceOfDensityFunctions^+_n \\ x \in \limitSpace } }
        				\frac{ ( \limitSemigroupOnMetricSpace( t ) f^o ) ( x )
        				}{ 
                    		[ \eigenfunction_\zeroVertex^o ]
                    		f^o
        				}
			.
	\end{equation*}

	We proceed by establishing the reverse inequality.
    Since $ \limitSpace $ is compact, the points $ \{ \inclusion( r_m ) \}_{ m \ge 1 } $ have a subsequential limit point, say $ \inclusion( r_{ m_k } ) \to y $.
    Using the identities in \eqref{continuous density functions produce kernels}, \eqref{equation density as limit down kernel}, and \ref{assumption consistency rn}, we find that
    \begin{align*}
    	\intertwiningKernel_n( y, r_n )
    		=
    			\lim_{ k \to \infty }
    				\downKernel_{ m_k, n }( r_{ m_k }, r_n )
    		=
    			\lim_{ k \to \infty }
    				\indicator( m_k \ge n )
    		=
    			1
			,
    \end{align*}
    and as a result, 
    $
    	(
		\intertwiningOperator_n g 
		)( y )
    		=
    			g( r_n )
	$
	for any $ g $ in $ C( \stateSpace_n )$.
  In particular (recall \eqref{eq:Tt_fo}), 
	$$
    	( \limitSemigroupOnMetricSpace( t ) f^o )( y )
        	=
    			( \intertwiningOperator_n
    			e^{ t \discreteGenerators_n }
				( f )_n)( y )
        	=
    			( e^{ t \discreteGenerators_n }
				( f )_n)( r_n )
			,
				\qquad
				f \in \spaceOfDensityFunctions_n^+
			.
    $$  
  
	\noindent
	Together with \cref{prop coefficient functional}\ref{claim coefficient functional extends discrete one}, \cref{prop properties of stationary measures}\ref{integrating against stationary measure} and \cref{lemma H+ properties}, this implies that
	\begin{equation*}%
					\inf_{ \substack{  f \in \spaceOfDensityFunctions^+_n \\ x \in \limitSpace } }
        				\frac{ ( \limitSemigroupOnMetricSpace( t ) f^o ) ( x )
        				}{ 
                    		[ \eigenfunction_\zeroVertex^o ]
                    		f^o
        				}
			\le 
        		\inf_{ f \in \spaceOfDensityFunctions^+_n }
        				\frac{ ( \limitSemigroupOnMetricSpace( t ) f^o ) ( y )
        				}{ 
                    		[ \eigenfunction_\zeroVertex]
                    		f
        				}  
			= 
        		\inf_{ f \in \spaceOfDensityFunctions^+_n }
        				\frac{ 
							(  
                			e^{ t \discreteGenerators_n }
            				( f )_n
							) ( r_n )
        				}{ 
                				\int_{ \stateSpace_n }
                					( f )_n
                					d \upDownDist_n
        				}
			= 
    			\inf_{ \substack{ g \in C_+(\stateSpace_n)  } } 
            				\frac{ ( e^{ t \discreteGenerators_n } g )( r_n ) 
            				}{ 
                				\int_{ \stateSpace_n }
                					g \,
                					d \upDownDist_n
                			}
			.
	\end{equation*}
	The desired inequality, and the first claim, now follows from \eqref{inequality final bound for reversal}.

	Moving on to the second claim, first observe that, since $\intertwiningOperator_n$ is a positive operator,
    \eqref{describing intertwining kernel with density functions} gives us the containment
    $
        \{ f^o \colon f \in \spaceOfDensityFunctions^+_n \}
        	\subseteq
               C_+( \limitSpace )
			.
    $
	Recalling the first claim and \cref{prop:Stationary Measure of Limit Process}\ref{claim coefficient description}, we can then write
    \begin{equation*}
        \sepDist_n^*(t) 
        =\sup_{ \substack{ x \in \limitSpace, \\ f \in \spaceOfDensityFunctions^+_n }}
        			\bigg( 
        				1 - \frac{ ( \limitSemigroupOnMetricSpace( t ) f^o )( x )
							 }{\int_{ \limitSpace }  f^o 
					d \upDownDist}
        			\bigg)
        	\le
        		\sup_{ \substack{ x \in \limitSpace, \\ f \in C_+( \limitSpace ) }}
        			\bigg( 
        				1 - \frac{ ( \limitSemigroupOnMetricSpace( t ) f )( x )
							 }{ \int_{ \limitSpace } f  d \upDownDist}
        			\bigg)
			=
				\sepDist_\limitProcess( t )
			,
				\qquad
				t \ge 0,
				\,
				n \ge 0
			.
    \end{equation*}
	On the other hand, \cref{prop limiting semigroup and generator} and \cref{thm general sep dist of a limit} supply us with the inequality
	$$
        \sepDist_\limitProcess(t) 
            \le 
            	\liminf_{ n \to \infty }
					\sepDist_n^*( t )
			,
				\qquad
				t \ge 0
			.
	$$
	We must therefore have the equality $ \sepDist_\limitProcess(t) = \lim_{ n \to \infty } \sepDist_n^*( t )$ for any $t \ge 0$.
	The monotonicity of this convergence follows from \cref{prop recursive structure in sep dist}.
\end{proof}

\section{Examples in the literature}
\label{sec:previous_examples}

In this section, we discuss how several up-down chains in the literature fit into our framework.
This will recover many of the main results in \cite{BOpartitions, petrovTwoParameter, petrov2010strict, OlshAddJackParameter, petrovSL2, lohr2020Aldous_chain, krdrDiffusions} while providing new results about these chains and their limiting processes.

   Before the main discussion, let us state a lemma that will simplify checking the commutation relation \ref{assumption:commutation}.
   In words, this result says that it is enough to check \ref{assumption:commutation} on the off-diagonal entries -- the relation for the diagonal entries will follow automatically.
   \begin{lemma}
     \label{lem:avoid_diag}
     Let $ \{ \upKernel_n \}_{ n \ge 0 } $ and $ \{ \downKernel_n \}_{ n \ge  1} $ be transition matrices, respectively from
$\stateSpace_n $ to $ \stateSpace_{n+1} $ and from $\stateSpace_n $ to $ \stateSpace_{n-1} $. 
Assume that for each $n$, there exist  $\beta_n$ in $ ( 0, 1 ) $ such that
\begin{equation}\label{eq:comm_off_diag}
    ( \upKernel_n
	\downKernel_{ n + 1 })(r,s)
		= 
        		\beta_n
    				(\downKernel_n
    				\upKernel_{ n - 1 }) (r,s)
		,
			\qquad
			r \neq s
		.
\end{equation}
    Then Assumption~\ref{assumption:commutation} holds.
   \end{lemma}
   \begin{proof}
     We need to check that, for all $s$ in $\stateSpace_n $, 
     \[( \upKernel_n \downKernel_{ n + 1 })(s, s)= \beta_n              
        (\downKernel_n \upKernel_{ n - 1 }) (s, s) + (1-\beta_n).\]
        \orange{This} follows immediately from \eqref{eq:comm_off_diag} and the identities
        \[( \upKernel_n\downKernel_{ n + 1 })(s, s)=1-\sum_{u \ne s}  (\upKernel_n\downKernel_{ n + 1 })(s,u), \qquad
        (\downKernel_n \upKernel_{ n - 1 }) (s, s)=1-\sum_{u \ne s}(\downKernel_n \upKernel_{ n - 1 }) (s,u).\qedhere\]
   \end{proof}

\subsection{Borodin--Olshanski chains on partitions}
\label{ssec:BO-partitions}

For $n \ge 1$, let $\stateSpace_n$ be the set of integer partitions of $n$ (\orange{equivalently,} Young diagrams with $n$ boxes).
The number of standard Young tableaux of shape $\la$ will be denoted by $\dim(\la)$. 
We write $\la \nearrow \rho$ if a partition $\rho$ can be obtained from a partition $\la$ by adding a single box.
If this new box $\rho \backslash \lambda$ lies in the $i$-th row and $j$-th column of the diagram of $ \rho $, we set $c(\rho \backslash \lambda) = j - i $ (this is known as the \emph{content} of the box).

Let $\la \, \in \stateSpace_n$ and $\rho \, \in \stateSpace_{n+1}$.
The down-steps considered in \cite{BOpartitions} are given by the formula 
\[\downKernel_{n+1}(\rho,\la)=\begin{cases}
  \frac{\dim(\la)}{\dim(\rho)}, & \lambda \nearrow \rho,\\
  0, & \text{ otherwise.}\end{cases}\]

\noindent
The up-steps are given by  
\[\upKernel_{n}(\la,\rho)=\begin{cases} \frac{(z+c(\rho \backslash \lambda))(z'+c(\rho \backslash \lambda))}{zz'+n}\ \frac{\dim(\rho)}{(n+1)\, \dim(\la)},  & \la \nearrow \rho, \\
     0, & \text{ otherwise,}
\end{cases}\]
where $z$ and $z'$ are two complex parameters chosen\footnote{Concretely, they should either be complex conjugates or lie in an interval contained in $ \mbb R \setminus \mbb Z $.} so that the above quantities are nonnegative and $zz'>0 $.
In the following result, we show that these transition kernels satisfy a commutation relation (a variant in which the operators are normalized differently can be found in \cite[Lemma 5.1]{fulmanCommutation}). 
\begin{prop}
\label{prop:BO_Commutation}
  The above matrices $(\upKernel_n)_{n \ge 1}$ and $(\downKernel_{n})_{n \ge 2}$ satisfy Assumption~\ref{assumption:commutation}
  with parameter $\beta_n=\frac{zz'+n-1}{zz'+n} \frac{n}{n+1}$ (for $n \ge 1$).
\end{prop}
\begin{proof}
  From \cref{lem:avoid_diag}, it suffices to check that for $\la,\mu \in \stateSpace_n$, $\la \ne \mu$, we have
  \begin{equation}\label{eq:comm_off_diag_BO}
    ( \upKernel_n
	\downKernel_{ n + 1 })(\la,\mu)
		= \frac{zz'+n-1}{zz'+n} \frac{n}{n+1}
    				(\downKernel_n
    				\upKernel_{ n - 1 }) (\la,\mu).
                  \end{equation}
    If there are two boxes or more that are in $\la$ but not in $\mu$, then both sides are equal to $0$ and the equality holds trivially.
    We can therefore assume that there is exactly one box that is in $\la$ but not in $\mu$.    
    Let $\rho=\la \cup \mu$ and $\tau=\la \cap \mu$, where $\la$ and $\mu$ should be viewed as sets of boxes.
    Then $\rho$ and $\tau$ are Young diagrams of size $n+1$ and $n-1$, respectively, \orange{and}
\[      ( \upKernel_n                                                                  
    \downKernel_{ n + 1 })(\la,\mu)= \upKernel_n(\la,\rho) \downKernel_{ n + 1 }(\rho,\mu) 
    = \frac{(z+c(\rho \backslash \la))(z'+c(\rho \backslash \la))}{zz'+n} \frac{\dim(\rho)}{(n+1)\, \dim(\la)} \frac{\dim(\mu)}{\dim(\rho)}.
    \]
    On the other hand,
    \[
             (\downKernel_n                                                                                              
                    \upKernel_{ n - 1 }) (\la,\mu)= \downKernel_n(\la,\tau)  \upKernel_{ n - 1 }(\tau,\mu)
                    =\frac{\dim(\tau)}{\dim(\la)} \frac{(z+c(\mu \backslash \tau))(z'+c(\mu \backslash \tau))}{zz'+n-1} \frac{\dim(\mu)}{n \, \dim(\tau)}.
      \]
     Observing that $\rho \backslash \la=\mu \backslash \tau$ establishes \eqref{eq:comm_off_diag_BO} and concludes the proof.
\end{proof}

To address Assumption~\ref{assumption finite state spaces}, we simply observe that $\stateSpace_1$ is a singleton.
Notice that this involves a shift of index since the state spaces in this example are indexed by $n \ge 1$, not $ n \ge 0 $.
It then follows from \cref{prop up-down chains give intertwining} that our chains satisfy the hypotheses \allAssumptions with $c_n=(n+1)(zz'+n)$.

We now discuss the hypotheses \ref{assumption distant elements} and \ref{assumption consistency rn}, which are needed for the results on separation distance.
Under the shift of index, \ref{assumption distant elements} requires two elements $r_n$ and $s_n$ in $\stateSpace_n$ that are at distance $n-1$ from each other. 
This condition is satisfied by taking $r_n=(n)$, a one-row diagram, and $s_n=(1^n)$, a one-column diagram: 
{$ s_n $ can be obtained from $ r_n $ by $ n - 1 $ up-down steps that create a new part and remove from the largest part each time}, and this is the fastest way to obtain $ s_n $ from $ r_n $ since each up-down step can increase the number of parts by at most $ 1 $.
The condition $M_n(s_n)>0$ is vacuous here because $M_n$ has full support (this is a consequence of \eqref{defn stationary measures} or of its explicit expression given in~\cite{BOpartitions}).
The assumption \ref{assumption consistency rn} is clearly satisfied.
The additional hypotheses of \cref{prop convergence of sep dist} are also satisfied -- namely, the rates satisfy $ \sum_{ n \ge 1 } \tfrac{1}{\generatorRates_n} < \infty $ and $ \{ \generatorRates_{n+1} - \generatorRates_{n} \}_{ n \ge 1 } $ is an unbounded, nondecreasing sequence.
\bigskip

Finally, we discuss \orange{our analytic hypotheses}. 
\orange{Again, we follow the work in \cite{BOpartitions}.}
For the limit space, we take the so-called Thoma simplex, the subspace $\Omega \subset [0, 1]^\infty \times [0, 1]^\infty $ consisting of pairs $(\bm a, \bm b)$, where $\bm a$ and $\bm b$ are infinite nonincreasing sequences of 
real numbers
such that $\sum ( a_i + b_i ) \le 1$.
The topology on $\Omega$ is that of pointwise convergence, under which it is a compact metric space.
A partition $\la \in \stateSpace_n$ can be seen as an element of $\Omega$ via its {\em modified Frobenius coordinates} -- namely $\iota(\la)=\frac1n (\bm a,\bm b)$, where
\begin{equation}\label{eq:Frobenius_coord}
 a_i= \max\big(\la_i -i +\tfrac12 ,0\big),
 \qquad b_i= \max\big(\la'_i -i +\tfrac12,0 \big) .
 \end{equation}

\noindent
We can check that this \orange{injection} satisfies \ref{assumption state space approximation}, i.e.~any point  $(\bm a, \bm b)$ in $\Omega$ can be approximate by a sequence
 of the form $\iota(\la_n)$: indeed, \orange{one can construct} $\la_n$ so that, for each fixed $i$ and $n$ sufficiently large,
 its $i$-th row (resp.~column) has length $\lfloor a_i n \rfloor$ (resp.~$\lfloor b_i n \rfloor$) if $a_i \ne 0$ (resp. $b_i \ne 0$), \orange{and length $o(n)$ if $a_i = 0$ (resp. $b_i = 0$)}.

We proceed by considering the density functions. 
It can be verified that they are given by
\begin{equation}\label{eq:dmu}
d_\mu(\la) = 
	\begin{cases}
		\dim(\la /\mu) \tfrac{\dim(\mu)}{\dim(\la)},
			& 	\mu \subseteq \la,
			\\
		0,
			& 	\text{else},
	\end{cases}
\end{equation}
where $\dim(\la / \mu)$ is the number of standard Young tableaux of skew shape $\la / \mu$.
We will also need to consider the action of a power sum symmetric function on a partition -- this is given by
\begin{equation}\label{eq:pk_la}
p_k(\la)= \sum_{i \ge 1} a_i^k +(-1)^{k-1} \sum_{i \ge 1} b_i^k,
	\qquad
		k \ge 1,
		\,
		\la \in \stateSpace,
\end{equation}
where $a_i$ and $b_i$ are defined by \eqref{eq:Frobenius_coord}.
This definition can be extended to $F(\la)$ for any symmetric function $F$, by requiring
that $F \mapsto F(\la)$ is an algebra morphism.
With this definition, one can prove (see \cite[eq. (4.1)]{BOpartitions}) that, for any $\mu \vdash k$, there exists
a symmetric function $FS_\mu$, called Frobenius--Schur function, such that
\begin{equation}\label{eq:FSmu}
FS_\mu(\la)= 
|\la| (|\la|-1) \cdots (|\la|-k+1) \frac{\dim(\la /  \mu)}{\dim(\la)} 
\end{equation}
Moreover $FS_\mu$ is an inhomogeneous function of degree $\mu$,
and its top degree component is the standard Schur function $s_\mu$.

We note that formula \eqref{eq:pk_la} can be used to define $p_k(\omega)$ 
for any $k \ge 2$ and any point $\omega$ of the Thoma simplex. By convention $p_1(\omega)=1$ for all $\omega$. Once again, we extend this definition to any $F(\omega)$, where $F$ is a symmetric function, by requiring that $F \mapsto F(\omega)$ is an algebra morphism.
Then we set $d_\mu^o(\omega) \ldef \dim(\mu)\, s_\mu(\omega)$.
We claim that this family of functions satisfies
\ref{assumption continuous density functions are dense}
and \ref{assumption continuous density functions are limits}.

The assumption \ref{assumption continuous density functions are dense} follows from Stone-Weierstrass theorem (recall that $\Omega$ is compact).
 The span of $\{d_\mu^o, \mu \in \stateSpace\}$
is the set of functions $\{\omega \mapsto F(\omega)\}$, where $F$ runs over the algebra of symmetric functions. It is thus a subalgebra of $C(\Omega)$.
Moreover, it contains the constant function $\omega \mapsto p_1(\omega)=1$
and separates point.
Indeed, $p_k(\omega)$ can be interpreted as the $k-1$-st moment of the discrete probability measure on $[0,1]$, giving weight $a_i$ to $a_i$, weight $b_i$ to $-b_i$, and the remaining weight $1-\sum_i a_i - \sum_i b_i$ to $0$.
Since bounded probability measures are determined by their moments,
$\omega$ is determined by the values of $(p_k(\omega))_{k \ge 1}$, as claimed.
We conclude that the span of $\{d_\mu^o, \mu \in \stateSpace\}$ is dense in $C(\Omega)$,
showing \ref{assumption continuous density functions are dense}.

For the remaining assumption \ref{assumption continuous density functions are limits},
we observe that 
\[ |p_k(\la)| \le \Big(\sum\nolimits_i a_i + \sum\nolimits_i b_i\Big)^k  \le n^k\]
 for any diagram $\la$ of size $n$.
This implies that, for any symmetric function $F$,
one has the bound $F(\la)=O(|\la|^{\deg(F)})$ uniformly in $\la$ (the constant in the $O$ symbol however depends on $F$). Since the top-homogeneous part of $FS_\mu$ is $s_\mu$, we get
\[ \big| FS_\mu(\la)-s_\mu(\la) \big| = O(|\la|^{|\mu|-1}) \]
Dividing by $|\la|(|\la|-1)\cdots (|\la|-|\mu|+1)=|\la|^{|\mu|}(1+O(|\la|^{-1}))$, and using \eqref{eq:FSmu}, we have
\[ \left| \frac{\dim(\la /  \mu)}{\dim(\la)} - \frac{s_\mu(\la)}{|\la|^{|\mu|}} \right| =O(|\la|^{-1}).\]
But 
\[\frac{s_\mu(\la)}{|\la|^{|\mu|}}=s_\mu(\iota(\la))=\frac1{\dim(\mu)} d_\mu^o(\iota(\la)),\]
while  $\frac{\dim(\la /  \mu)}{\dim(\la)}=\frac1{\dim(\mu)} d_\mu(\la)$ (see \eqref{eq:dmu}).
Since all error terms are uniform in $\la$, this proves \ref{assumption continuous density functions are limits}.

Therefore, all results of \cref{section discrete framework} apply and we recover the results from \cite{fulmanCommutation} on the separation distance of this up-down chain.

To sum up, all hypotheses \ref{assumption finite state spaces}, \ref{assumption:commutation},  \ref{assumption distant elements}-\ref{assumption consistency rn} and \ref{assumption state space approximation}-\ref{assumption continuous density functions are limits} are satisfied, hence all our results apply.
We recover the scaling limit results from \cite{BOpartitions} (see Theorem 5.1 there) and the triangular form of the limiting generator \cite[Lemma 5.4]{BOpartitions}.
Other results, like the large time asymptotics of the limiting process (\cref{prop densities of limit process}), the intertwining between the discrete and the continuous processes
(Proposition~\ref{prop algebraic identities in limit})
or the computation of its separation distance to equilibrium (\cref{theorem discrete and continuous sep dist}) seem to be new.

\begin{remark}
A similar chain on strict partitions has been defined and analyzed by Petrov in \cite{petrov2010strict}.
Our general theory also applies to this example, but since most of the discussion
would be similar
to that of Borodin--Olshanski's chain, we do not give any detail here.
\end{remark}

\subsection{Olshanski's generalization through Jack polynomials}
The chain in the previous section can be generalized, introducing a parameter $\theta$
linked to Jack polynomials. We follow here the paper \cite{OlshAddJackParameter}
and use its notation.

In particular, if $\la$ is a partition and $\theta$ a positive real number,
 $P^{\theta}_\la$ denotes the Jack symmetric function indexed by $\la$ of parameter $\theta$. It is a homogeneous symmetric function of degree $|\la|$ and they form a basis of the symmetric function algebra. 
 The Pieri rule states that, for any partition $\mu$,
 \[P^{\theta}_\mu \, p_1=\sum_\la \kappa^\theta(\mu,\la) P^{\theta}_\la,\]
 where the sum runs over partitions $\la$ obtained by adding one box to $\mu$ (i.e.~$\mu \nearrow \la$) and $\kappa^\theta(\mu,\la)$ are some combinatorial quantities whose precise definition is irrelevant here.
Using the coefficients  $\kappa^\theta(\mu,\la)$, we can introduce a $\theta$-deformation of dimensions as follows: fixing $\mu$, we define $\dim^\theta(\la / \mu)$ recursively on $\la$ by
\[\dim^\theta(\la / \mu) = \begin{cases} 0 & \text{ if } \la\not\supseteq \mu;\\
1 & \text{ if } \la=\mu;\\
 \sum_{\tau \nearrow \la} \dim^\theta(\tau / \mu)  \kappa^\theta(\tau,\la) & \text{ if } \la \supsetneq \mu. \end{cases}\]
 We simply write $\dim^\theta(\la)$ for $\dim^\theta(\la / \varnothing)$, 
 where $\varnothing$ is the empty partition. For $\theta=1$, this coincides with the usual
 dimension of (skew) Young diagrams.
 Still following Olshanski~\cite{OlshAddJackParameter}, we can now define a down transition matrix:
 for $\rho \vdash n+1$ and $\la \vdash n$,
\begin{equation}\label{eq:pdowntheta}
p^{\downarrow,\theta}_{n+1}(\rho,\la)=\begin{cases}
  \frac{\dim^\theta(\la) \kappa^\theta(\tau,\la) }{\dim^\theta(\rho)} 
  & \text{ if }\lambda \nearrow \rho,\\
  0 & \text{ otherwise.}\end{cases}
\end{equation}
It turns out that there is an equivalent description of $p^{\downarrow,\theta}_{n+1}$ which will
be useful for establishing the commutation relation.
Given a Young diagram $\rho$, we consider its Russian representation with boxes of length $1$, but height $\theta$ (this is sometimes called anisotropic Young diagrams and is standard in the context of Jack polynomials); see \cref{fig:anisotropic}, left.
We then call $x^\rho_0 <y^\rho_1 < \dots < y^\rho_m < x^\rho_m$ the local extrema of the outer border of the Young diagram (the $x^\rho_i$'s are local minima, while the $y^\rho_i$'s are local maxima).
These are called \orange{the} ($\theta$-dependent) interlacing coordinates of the diagram.
 
\begin{figure}
\[\includegraphics[scale=.6]{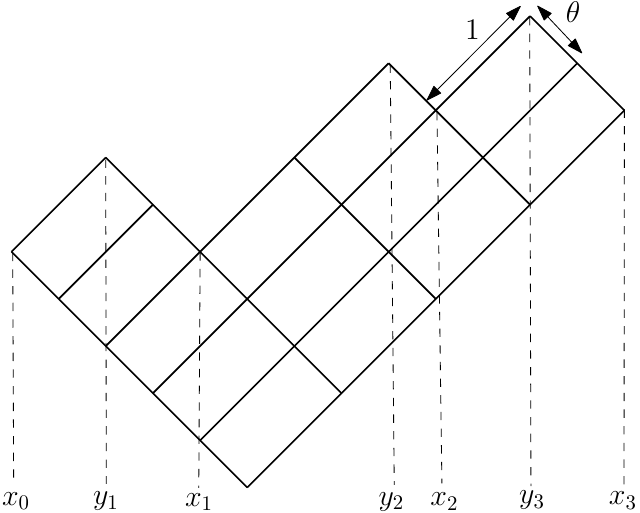} \qquad
\includegraphics[scale=.6]{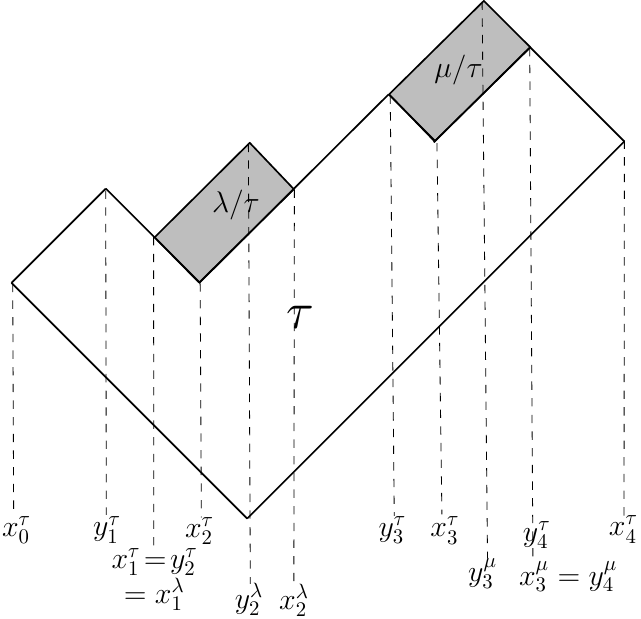}\]
\caption{Left: the anisotropic Young diagram $\tau=(4,4,3,1,1)$ (for $\theta=.5$)
and its interlacing coordinates. Right: the same Young diagram $\tau$, together with two extra boxes, defining diagrams $\la$ and $\mu$, each with one more box than $\tau$.
We indicated the interlacing coordinates of $\tau$ in which we have inserted an artificial pair $x_1^\tau=y_2^\tau$ for convenience. We also indicated the interlacing coordinates of $\la$ and $\mu$ which are different from that of $\tau$. Note that the interlacing coordinates
of $\mu$ also contain an artificial pair $x_3^\mu=y_4^\mu$.
These coordinates satisfy  \eqref{eq:xj1} and \eqref{eq:xj2}.  }
\label{fig:anisotropic}
\end{figure}

 Now if $\lambda \nearrow \rho$, the unique box of $\rho / \la$ corresponds
 to $y^\rho_j$ for some $j \le m$. Then it can be proven \cite[Proposition 4.2]{OlshAddJackParameter} that
 \begin{equation}\label{eq:pdowntheta_interlacing}
 p^{\downarrow,\theta}_{n+1}(\rho,\la)=\frac{-1}{\theta (n+1)} \, \frac{\prod_{i=0}^m (y^\rho_j-x^\rho_i)}{\prod_{ 1 \le i \le m, \, i \ne j} (y^\rho_j-y^\rho_i)}.
 \end{equation}
 
 To define the up transition matrix, we fix parameters $z,z'$ as before
 and use the interlacing coordinates $x^\la_0 <y^\la_1 < \dots < y^\la_m < x^\la_m$ of $\la$.
 The box $\rho / \la$ corresponds to $x^\rho_j$ for some $j \in  \{0,\dots,m\}$.
 Then we set
  \begin{equation}\label{eq:puptheta}
p^{\uparrow,\theta}_{n}(\la,\rho) = \frac{(z+x^\la_j)(z'+x^\la_j)}{zz'+\theta\, n} \, 
 \frac{\prod_{i=1}^m (x^\la_j-y^\la_i)}{\prod_{ 0 \le i \le m, \, i \ne j} (x^\la_j-x^\la_i)},
  \end{equation}
  where, as before, $z$ and $z'$ are complex parameters, either conjugate or in the same interval $(N,N+1)$ for some integer $N$.

We now establish the following commutation relation, 
which does not appear in the paper
 \cite{OlshAddJackParameter}.
 This is the only example where proving Assumption~\ref{assumption:commutation}
 requires a delicate computation.
\begin{prop}
  The above matrices $(p^{\uparrow,\theta}_n)_{n \ge 1}$ and $(p^{\downarrow,\theta}_{n})_{n \ge 2}$ satisfy Assumption \ref{assumption:commutation}
  with parameter $\beta_n=\frac{n\, (zz'+\theta\, (n-1))}{(n+1)\, (zz'+\theta\, n)}$ (for $n \ge 1$).
\end{prop}
\begin{proof}
As for \cref{prop:BO_Commutation}, it suffices to consider
the case where $\la \ne \mu$, and they differ from exactly one box,
so that $\tau\ldef \la \cap \mu$ and $\rho\ldef \la \cup \mu$ have respectively $n-1$
and $n+1$ boxes.
We note that the formulae \eqref{eq:pdowntheta_interlacing} and \eqref{eq:puptheta}
are not changed if we insert a pair of artificial equal coordinates $x_i=y_i$ or $y_i=x_{i+1}$ .
Using this, we can assume that the numbers $2m+1$ of interlacing coordinates
for $\la$, $\mu$, $\tau$ and $\rho$ are all the same, and that they have the same
interlacing coordinates, except that there exists $j_1 \ne j_2$ in $\{1,\dots,m\}$ such that
$x_{j_1}^\tau =y_{j_1}^\tau +\theta$,  $x_{j_2}^\tau =y_{j_2}^\tau +\theta$,
\begin{align}
  x_{j_1}^\rho&=x_{j_1}^\la= x_{j_1}^\tau +1=x_{j_1}^\mu +1,\  y_{j_1}^\rho=y_{j_1}^\la= y_{j_1}^\tau +1=y_{j_1}^\mu +1  \label{eq:xj1}\\
  \text{and }
x_{j_2}^\rho&=x_{j_2}^\mu= x_{j_2}^\tau +1=x_{j_2}^\la +1,\  y_{j_2}^\rho=y_{j_2}^\mu= y_{j_2}^\tau +1=y_{j_2}^\la +1.\label{eq:xj2}
\end{align}
See \cref{fig:anisotropic}, right. Then
\begin{multline}
 ( p^{\uparrow,\theta}_n                                                                  
    p^{\downarrow,\theta}_{ n + 1 })(\la,\mu)= p^{\uparrow,\theta}_n(\la,\rho) p^{\downarrow,\theta}_{ n + 1 }(\rho,\mu) \\
    =\frac{-(z+x^\la_{j_2})(z'+x^\la_{j_2})}{\theta (n+1)\, (zz'+\theta\, n)} \, 
 \frac{\prod_{i=1}^m (x^\la_{j_2}-y^\la_i)}{\prod_{ 0 \le i \le n, \, i \ne {j_2}} (x^\la_{j_2}-x^\la_i)} 
 \, \frac{\prod_{i=0}^m (y^\rho_{j_1}-x^\rho_i)}{\prod_{ 1 \le i \le n, \, i \ne {j_1}} (y^\rho_{j_1}-y^\rho_i)}.
 \label{eq:UDtheta}
\end{multline}
On the other hand,
\begin{multline}
(p^{\downarrow,\theta}_n                                                                                              
                    p^{\uparrow,\theta}_{ n - 1 }) (\la,\mu)= p^{\downarrow,\theta}_n(\la,\tau)  p^{\uparrow,\theta}_{ n - 1 }(\tau,\mu) \\
                    =\frac{-(z+x^\tau_{j_2})(z'+x^\tau_{j_2})}{\theta \, n\, (zz'+\theta\, (n-1))}
                      \frac{\prod_{i=0}^m (y^\la_{j_1}-x^\la_i)}{\prod_{ 1 \le i \le m, \, i \ne {j_1}} (y^\la_{j_1}-y^\la_i)}
                     \frac{\prod_{i=1}^m (x^\tau_{j_2}-y^\tau_i)}{\prod_{ 0 \le i \le m, \, i \ne {j_2}} (x^\tau_{j_2}-x^\tau_i)} 
 \,
  \label{eq:DUtheta}
\end{multline}
Using \cref{eq:xj1,eq:xj2}, we have
\[\frac{\prod_{i=0}^m (y^\rho_{j_1}-x^\rho_i) \Big/ \prod_{ 1 \le i \le n, \, i \ne {j_1}} (y^\rho_{j_1}-y^\rho_i)}
{\prod_{i=0}^m (y^\la_{j_1}-x^\la_i) \Big/ \prod_{ 1 \le i \le n, \, i \ne {j_1}} (y^\la_{j_1}-y^\la_i)} = \frac{ (y^\rho_{j_1}-x^\rho_{j_2})(y^\la_{j_1}-y^\la_{j_2})}{ (y^\rho_{j_1}-y^\rho_{j_2})(y^\la_{j_1}-x^\la_{j_2})} 
= \frac{ (y^\tau_{j_1}-y^\tau_{j_2}-\theta)(y^\tau_{j_1}-y^\tau_{j_2}+1)}{ (y^\tau_{j_1}-y^\tau_{j_2})(y^\tau_{j_1}-y^\tau_{j_2}+1-\theta)}\]
and
\[ \!\! \frac{\prod_{i=1}^m (x^\la_{j_2}-y^\la_i) \Big/ \prod_{ 0 \le i \le m, \, i \ne {j_2}} (x^\la_{j_2}-x^\la_i)} {\prod_{i=1}^m (x^\tau_{j_2}-y^\tau_i) \Big/ \prod_{ 0 \le i \le m, \, i \ne {j_2}} (x^\tau_{j_2}-x^\tau_i)} = \frac{ (x^\la_{j_2}-y^\la_{j_1})(x^\tau_{j_2}-x^\tau_{j_1})}{(x^\la_{j_2}-x^\la_{j_1})(x^\tau_{j_2}-y^\tau_{j_1})}=\frac{ (y^\tau_{j_2}-y^\tau_{j_1}-1+\theta)(y^\tau_{j_2}-y^\tau_{j_1})}{(y^\tau_{j_2}-y^\tau_{j_1}-1)(y^\tau_{j_2}-y^\tau_{j_1}+\theta)}. \]
We note that the product of the expressions in the last two displays is equal to $1$
so that the product of the last two fractions in \eqref{eq:UDtheta} and \eqref{eq:DUtheta} are the same. Since $x^\tau_{j_2}=x^\la_{j_2}$, we conclude that
\[ ( p^{\uparrow,\theta}_n                                                                  
    p^{\downarrow,\theta}_{ n + 1 })(\la,\mu)= \frac{n\, (zz'+\theta\, (n-1))}{(n+1)\, (zz'+\theta\, n)} (p^{\downarrow,\theta}_n                                                                                              
 p^{\uparrow,\theta}_{ n - 1 }) (\la,\mu),\]
 as wanted.
  \end{proof}
The hypotheses  \ref{assumption distant elements} and  \ref{assumption consistency rn}
are satisfied (with a shift of index) for the exact same reason than in the case $\theta=1$.
The additional assumptions of \cref{prop convergence of sep dist} are also satisfied.
Our results on separation distance of the discrete chain therefore apply:
in particular, \cref{thm:separation_distance} gives an exact expression for the separation
distance that generalizes some results of \cite[Section 5]{fulmanCommutation} (\orange{the latter corresponding
to the case $\theta=1$}).

Regarding the scaling limit, we still use the Thoma simplex $\Omega$ as limiting space
but the embedding $\iota$ needs to be twisted. Namely for a Young diagram $\la$,
we set $\iota_\theta(\la)=\frac1n (\bm a,\bm b)$, where $a_i$ (resp. $b_i$) is the area of the intersection of the $i$-th row (resp.~$i$-th column) of the diagram in French representation and the half-plane $\{(r,s): \theta r \ge s\}$ (resp. $\{(r,s): \theta r \le s\}$):
see \cref{fig:theta_embedding}. It is clear that $\Omega$ with this embedding 
still satisfies \ref{assumption state space approximation}.
\begin{figure}
\[\includegraphics[scale=.7]{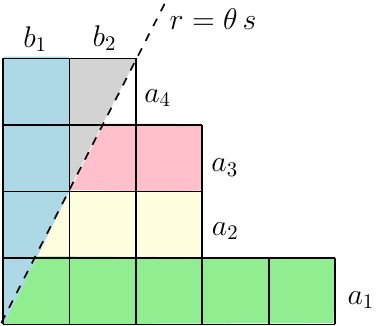} \]
\caption{The $\theta$-embedding of a Young diagram $\lambda$ in $\Omega$.
The $a_i$'s and $b_i$'s are the areas of the colored regions of the diagram.}
\label{fig:theta_embedding}
\end{figure}

Let us now look at the density functions.
Iterating \cref{eq:pdowntheta}, we get
\[d_\mu(\la)=p^{\downarrow}_{|\la|,|\mu|} (\la,\mu)=\dim^\theta(\mu) \frac{\dim^\theta(\la / \mu)}{\dim^\theta(\la)}.\]
It turns out \cite[Eq. (5.2)]{okounkov1997shifted_Jack} that there exists some functions, analogue to the Frobenius--Schur functions above, called shifted Jack polynomials, and denoted $P^{\#,\theta}_\mu$ such that for any partition $\la$,
\[P^{\#,\theta}_\mu(\la)=|\la| (|\la|-1) \dots  (|\la|-|\mu|+1)  \frac{\dim^\theta(\la / \mu)}{\dim^\theta(\la)}.\]
The top homogeneous component of shifted Jack polynomials
are the Jack symmetric functions $P^{\theta}_\mu$, which invites us to set 
$d^o_\mu(\omega)=\dim^\theta(\mu)\, P^{\theta}_\mu( \omega)$,
where we define $F(\omega)$ as before (for $F$ a symmetric function and $\omega$ a point in $\Omega$). 
Then Assumption~\ref{assumption continuous density functions are dense} is proved in the same way as before, and Assumption~\ref{assumption continuous density functions are limits} is a consequence of \cite[Theorem 9.5]{OlshAddJackParameter}.

The scaling limit result of \cite{OlshAddJackParameter} follows from our general theory,
together with some additional results: the triangular description of the generator via density functions seem to be new in this case (\cref{prop algebraic identities in limit} item \ref{claim generator on density functions}), as well as the intertwining property between the discrete and continuous processes (\cref{prop algebraic identities in limit}), and the result on separation distances
(\cref{prop convergence of sep dist,theorem discrete and continuous sep dist}).
However, as previously, we cannot recover the differential expression of the generator.

\subsection{Partition chains arising from the Chinese restaurant process}
\label{section crp partitions}

Moving on to the chains in \cite{petrovTwoParameter}, we will once again deal with partitions, but with different dynamics.
Interestingly, we will need to use a different limiting space
(since the down-operator is different, the density functions are different, and
we need to consider a different space for \ref{assumption continuous density functions are dense}--\ref{assumption continuous density functions are limits} to be satisfied).
We use the French convention to draw Young diagrams,
and let $ \ell( \lambda ) $ denote the number of rows in a partition $ \lambda $.

In Petrov's chains, a down-step from $ \lambda \in \stateSpace_n $ deletes \orange{the last box} from row $ j $ with probability $ \frac{\lambda_j}{n} $.
\orange{Boxes above the deleted box then need to be shifted down} so that the resulting diagram is once again a partition.
The up-steps depend on two parameters $(\alpha,\theta)$ satisfying $0\leq \alpha<1$ and $\alpha+\theta>0$, and will add a box to a partition.
In particular, an up-step from $ \lambda \in \stateSpace_n $ will
\begin{itemize}

	\item
	add a box to row $ j $ with probability $ \frac{ \lambda_j - \alpha }{ n + \theta } $,
	and
	
	\item
	create a one-box row with probability $ \frac{ \theta + \alpha \ell( \lambda ) }{ n + \theta } $.
\end{itemize}

\noindent
If necessary, the rows in the resulting diagram should be reordered to obtain a partition.
\orange{Examples of up and down transition probabilities are shown on \cref{fig:transition_CRP}.}
\orange{When $\alpha=0$, these up-steps correspond to the so-called}
{\em Chinese restaurant process} \cite{pitmanBook}. 
\orange{We denote $(\upKernel_n)_{n \ge 1}$ and $(\downKernel_{n})_{n \ge 2}$ the corresponding transition matrices.
We now verify the commutation assumption.}

\begin{figure}
\[\includegraphics[height=5cm]{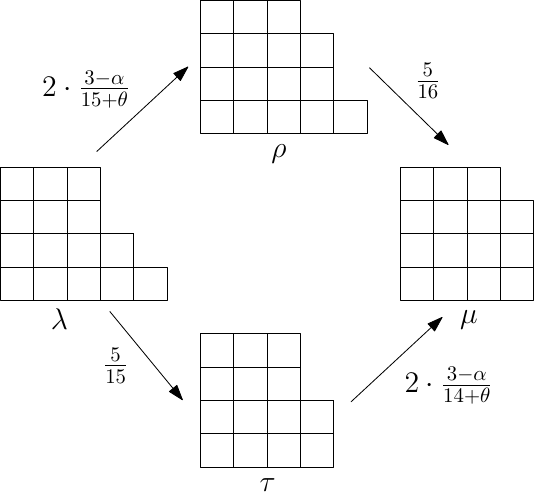}\]
\caption{Here, we show four partitions $\la$, $\mu$, $\rho$ and $\tau$ represented by their Young diagrams in French convention. We also printed the up transition probabilities
$\upKernel_n(\la,\rho)$ and $\upKernel_{n-1}(\tau,\mu)$, as well as
the down transition probabilities $\downKernel_n(\la,\tau)$ and $\downKernel_{n+1}(\rho,\mu)$.
Note the multiplicity factor $2$ in $\upKernel_n(\la,\rho)$ which comes from the fact that $\rho$ can be obtained from $\la$ by adding a box in row 3 or in row 4 (the same factor appears for the same reason in  $\upKernel_{n-1}(\tau,\mu)$.
}
\label{fig:transition_CRP}
\end{figure}

\begin{prop}
\label{prop:CRP_Commutation}
  \orange{The above matrices $(\upKernel_n)_{n \ge 1}$ and $(\downKernel_{n})_{n \ge 2}$ satisfy Assumption \ref{assumption:commutation}
  with parameter $\beta_n=\frac{ n ( n - 1 + \theta )}{ ( n + 1 ) ( n + \theta ) }$ (for $n \ge 2$).}
\end{prop}
\begin{proof}
\orange{Again, it suffices to consider
the case where $\la \ne \mu$, and they differ from exactly one box,
so that $\tau\ldef \la \cap \mu$ and $\rho\ldef \la \cup \mu$ have respectively $n-1$
and $n+1$ boxes. Call $j_1$ (resp.~$j_2$) the row index of the unique box in $\la / \tau$
(resp.~in $\mu / \tau$). Note that $j_1\ne j_2$ since $\la \ne \mu$. 
We first assume that $\tau_{j_2}$ is positive.
Let $m_i(\pi)$ denote the multiplicity of $i$ in a partition $\pi$.
Then we have
 \[( p^{\uparrow,\theta}_n                                                                  
    p^{\downarrow,\theta}_{ n + 1 })(\la,\mu)= p^{\uparrow,\theta}_n(\la,\rho) \, p^{\downarrow,\theta}_{ n + 1 }(\rho,\mu)  =\frac{m_{\la_{j_2}}(\la) (\la_{j_2}-\alpha)}{n+\theta} \frac{m_{\rho_{j_1}}(\rho) \cdot \rho_{j_1} }{n+1}. \]
    On the other hand,
\[
(p^{\downarrow,\theta}_n                                                                                              
                    p^{\uparrow,\theta}_{ n - 1 }) (\la,\mu)= p^{\downarrow,\theta}_n(\la,\tau)  \, p^{\uparrow,\theta}_{ n - 1 }(\tau,\mu)  = \frac{m_{\la_{j_1}}(\la) \cdot \la_{j_1} }{n}\frac{m_{\tau_{j_2}}(\tau) (\tau_{j_2}-\alpha)}{n-1+\theta}.
                    \]
 To explain these formulas, let us note that to go from $\la$ to $\rho$ in an up-step,
 we may choose to add a box in any row which has the same length than $\la_{j_2}$,
 whence the factor $m_{\la_{j_2}}(\la)$.
 By construction, we have $\la_{j_2}=\tau_{j_2}$ and $\rho_{j_1}=\la_{j_1}$.
 Also, except when $\la_{j_1}=\la_{j_2}$ (the added and removed boxes are in neighbouring columns), we have $m_{\la_{j_2}}(\la)=m_{\tau_{j_2}}(\tau)$ and $m_{\rho_{j_1}}(\rho)=m_{\la_{j_1}}(\la)$. In the case $\la_{j_1}=\la_{j_2}$, we have $m_{\la_{j_1}}(\la)=m_{\la_{j_2}}(\la)=m_{\tau_{j_2}}(\tau)+1=m_{\rho_{j_1}}(\rho)+1$.
 In both cases, it holds that
 \[(p^{\uparrow,\theta}_n                                                                  
    p^{\downarrow,\theta}_{ n + 1 })(\la,\mu)=
    \frac{n(n-1+\theta)}{(n+1)(n+\theta)} (p^{\downarrow,\theta}_n                                                                                              
                    p^{\uparrow,\theta}_{ n - 1 }) (\la,\mu).\]
   The case where $\tau_{j_2}=0$ is similar.}
\end{proof}
\orange{Since the underlying graph is the same as in \cref{ssec:BO-partitions} (only the transition
probabilities differ)}, the conditions \ref{assumption distant elements} and \ref{assumption consistency rn} (keeping the change of index in mind) are also satisfied here.
It can be verified that the hypotheses of Proposition \ref{prop convergence of sep dist} are also satisfied.
Therefore, all of the results in \cref{section discrete framework} apply to these chains.
Here, we recover 
	the complete triangular description in \cite[Proposition 3.1]{petrovTwoParameter}, 
	the description of the spectrum in \cite[Corollary 7.2]{fulmanCommutation},
	and \cite[Theorem 7.5]{fulmanCommutation}, which provides an explicit formula for the separation distance and its scaling limit when $ \theta = 1 $.
Our description of the eigenfunctions in Proposition \ref{prop eigenbasis}, 
the asymptotic descriptions in Propositions \ref{prop density estimate} and \ref{prop initial convergence},
and the analysis of the separation distance in \cref{ssec:separation_distance_discrete} (for all values of $ \theta $), 
appear to be new.
\medskip

To discuss the scaling limit of these chains, we follow \cite{petrovTwoParameter}.
The limiting space will be taken to be the Kingman simplex, defined as 
$$ 
	\overline{\nabla}_\infty 
		= 
			\left\{
				\mathbf{x} = (x_1,x_2,\dots) \ 
				: 
				\ x_1\ge x_2\ge \cdots\ge 0, 
				\sum_{i \ge 1} 
					x_i
						\leq 
								1
			\right\}
		.
$$
This space becomes a compact metrizable space when equipped with the topology of coordinatewise convergence.
A partition $ \lambda $ can be identified as the following element of $ \overline{\nabla}_\infty $:
$$
	\iota( \lambda ) 	
		= 	
				\bigg( 
						\frac{ \lambda_1 }{ | \lambda | }, 
						\frac{ \lambda_2 }{ | \lambda | }, 
						\ldots 
						\frac{ \lambda_{ \ell(\lambda) } }{ | \lambda | }, 
						0, 0, \ldots
				\bigg)
	.
$$
It then follows from \cite[Remark 2.1]{petrovTwoParameter} that \ref{assumption state space approximation} holds.

To extend the density functions to $ \overline{\nabla}_\infty $, we first observe that,
using \cite[equation (6)]{petrovTwoParameter}, it can be shown that the density functions in this context are given by
\begin{equation*}
	\density_\mu ( \lambda )
		=
               \frac{1}{ \binom{|\lambda|}{|\mu|}}
    			\sum_{ \substack{1 \le i_1, i_2, \ldots, i_{ \ell( \mu ) } \le \ell( \lambda ) \\ \text{distinct}} }
    				\prod_{ r = 1 }^{ \ell( \mu) }
    					\binom{ \lambda_{ i_r } }{ \mu_r }
		,
			\qquad
				| \lambda | \ge | \mu |
		.
\end{equation*}

Then we make use of \cite[Remark 5.2]{petrovTwoParameter}.
This remark implies that there is a unique continuous function $ d_\mu^o \in C( \overline{\nabla}_\infty  ) $ that satisfies
$$
	d_\mu^o( \mb x ) 
		=
			\frac{|\mu|!}{ \prod_{ r = 1 }^{ \ell( \mu ) }  \mu_r !}
			\sum_{ \substack{1 \le i_1, i_2, \ldots, i_{ \ell( \mu ) } \\ \text{distinct}} }
				\prod_{ r = 1 }^{\ell(\mu)} 
					x_{ i_r }^{ \mu_r }
		,
			\qquad \text{whenever}
			\sum_{ k \ge 1 }
				x_k
					=
						1
		.
$$

\noindent
This function would be called $ g(\mu) m_\mu^o $ in \cite{petrovTwoParameter}
(see Proposition 5.1 there).
Assumption \ref{assumption continuous density functions are dense} now follows from the discussion in \cite[Section 2.2]{petrovTwoParameter} and Assumption \ref{assumption continuous density functions are limits} can be established using \cite[equation (6)]{petrovTwoParameter} and the proof of \cite[Lemma 4.1]{petrovTwoParameter}.
The hypothesis of Proposition \ref{prop:path_continuity} also holds: an up-down step from $ \lambda $ will cause at most two coordinates of $ \iota( \lambda ) $ to change and these changes will be at most $ \frac{1}{|\lambda|} $.

We have shown that all of the results in \cref{section convergence} apply.
This recovers some of the main results in \cite{petrovTwoParameter} (his Main theorem,: his item (3) follows from \eqref{defn stationary measures} and Proposition \ref{prop:Stationary Measure of Limit Process} here), \orange{except for the differential form of the generator.}
Some of our results appear to be new -- these include
	the diagonal descriptions in Proposition \ref{prop limiting semigroup and generator},
	the intertwining result in Proposition \ref{prop algebraic identities in limit},
	the asymptotic description in Proposition \ref{prop densities of limit process},
	and
	the description of the separation distance of the limiting process
	in Theorem \ref{theorem discrete and continuous sep dist}.

\subsection{Composition chains arising from the Chinese restaurant process}
\label{ssec:composition}

In the up-down chains considered in \cite{krdrDiffusions}, the state space $ \stateSpace_n $ consists of integer compositions of $ n $.
A composition of $ n \ge 1 $ is a tuple $ \sigma = ( \sigma_1, ..., \sigma_k ) $ of positive integers that sum to $ n $. 
It will be helpful to think of these objects in terms of their associated box diagrams. 
The diagram for a composition $ \sigma \in \stateSpace_n $ contains $ n $ boxes arranged into columns so that there are $ \sigma_j $ boxes in the $ j $-th column. %
We denote the number of boxes and columns in a composition $ \sigma $ by $ | \sigma | $ and $ \ell( \sigma) $, respectively.
The down-steps considered in these chains delete a box uniformly at random from a composition. 
If this box lied below other boxes, those boxes are shifted down so that the resulting diagram still describes a composition. 
Similarly, if this box was the only box in its column, the remaining columns should be shifted so that there is no empty column.

The up-steps in these chains add a box to a composition and depend on two parameters $(\alpha,\theta)$ satisfying $\theta\geq 0$, $0\leq \alpha<1$, and $\alpha+\theta>0$.
Given an initial composition $ \sigma \in \stateSpace_n $, an up-step from $ \sigma $ will 
\begin{itemize}

	\item
	add a box on top of column $ j $ with probability $ \frac{ \sigma_j - \alpha }{ n + \theta } $,
	
	\item
	create a one-box column to the left of the first column with probability $ \frac{ \theta }{ n + \theta } $,
	and
	
	\item
	create a one-box column immediately to the right of a given column with probability $ \frac{ \alpha }{ n + \theta } $.
\end{itemize}

\noindent
These up-steps are inspired by an ordered variant of the Chinese restaurant process introduced in~\cite{pitmanOCRP}.
The resulting up-down chains can thus be viewed as ordered variants of the chains in \cref{section crp partitions}.
\orange{Again, we denote $(\upKernel_n)_{n \ge 1}$ and $(\downKernel_{n})_{n \ge 2}$ the corresponding transition matrices.
Examples of composition diagrams and up- and down-transition probabilities can be seen
in \cref{fig:Ordered_CRP}.}

\begin{figure}
\[\includegraphics[height=5cm]{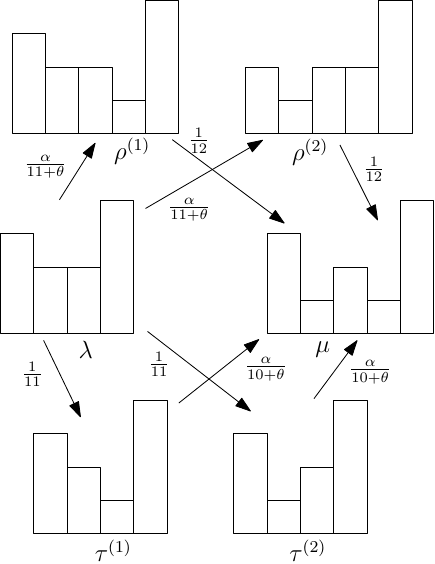}\]
\caption{Here, we show six compositions partitions $\la$, $\mu$, $\rho^{(1)}$, $\rho^{(2)}$, $\tau^{(1)}$ and $\tau^{(2)}$ represented by their box diagrams. We also printed the up transition probabilities
$\upKernel_n(\la,\rho^{(i)})$ and $\upKernel_{n-1}(\tau^{(i)},\mu)$, as well as
the down transition probabilities $\downKernel_n(\la,\tau^{(i)})$ and $\downKernel_{n+1}(\rho^{(i)},\mu)$ (each time for $i$ in $\{1,2\}$).
}
\label{fig:Ordered_CRP}
\end{figure}

\begin{prop}
\label{prop:Ordered_CRP_Commutation}
  \orange{The above matrices $(\upKernel_n)_{n \ge 1}$ and $(\downKernel_{n})_{n \ge 2}$ satisfy Assumption \ref{assumption:commutation}
  with parameter $\beta_n=\frac{ n ( n - 1 + \theta )}{ ( n + 1 ) ( n + \theta ) }$ (for $n \ge 2$).}
\end{prop}

\begin{proof}
\orange{The proof is similar to that of \cref{prop:CRP_Commutation}, except that we do not have multiplicity factors (\orange{except when adding a part of size $1$ in a series of parts of size $1$}), and with the following modification. (We use here exponents in compositions to indicate repeated parts.)}
\begin{itemize}
\item \orange{If $\la$ and $\mu$ are of the form $(u,2^{k+1},v)$ and $(u,1,2^k,1,v)$ for some $k \ge 0$ and some compositions $u$ and $v$, then there are two compositions $\rho^{(1)}$ and $\rho^{(2)}$ of size $n+1$
which can both be obtained  from $\la$ and $\mu$ by adding a single box,
namely $\rho^{(1)}=(u,2^{k+1},1,v)$ and $\rho^{(2)}=(u,1,2^{k+1},v)$.
See again \cref{fig:Ordered_CRP} for an example.}

\orange{Similarly, if $\la$ and $\mu$ are of the form $(u,2^{k},1,v)$ and $(u,1,2^{k},v)$,
then $\rho^{(1)}=(u,1,2^{k},1,v)$ and $\rho^{(2)}=(u,2^{k+1},v)$
can both be obtained  from $\la$ and $\mu$ by adding a single box.}
\item \orange{Conversely, for some $\la$ and $\mu$, there might be two compositions $\tau^{(1)}$ and $\tau^{(2)}$ of size $n-1$
which can be obtained both from $\la$ and $\mu$ by removing a single box
(just "revert" the previous example).}
\end{itemize}
\orange{In the first (resp.~second) case, the probability $( p^{\uparrow,\theta}_n                                                                  
    p^{\downarrow,\theta}_{ n + 1 })(\la,\mu)$ (resp.~$(p^{\downarrow,\theta}_n                                                                                              
p^{\uparrow,\theta}_{ n - 1 }) (\la,\mu)$) is a sum of two terms.
Nevertheless, the identity,
 \[(p^{\uparrow,\theta}_n                                                                  
    p^{\downarrow,\theta}_{ n + 1 })(\la,\mu)=
    \frac{n(n-1+\theta)}{(n+1)(n+\theta)} (p^{\downarrow,\theta}_n                                                                                              
                    p^{\uparrow,\theta}_{ n - 1 }) (\la,\mu)\]
                    is straightforward to check in all cases.}
\end{proof}

Let us verify the conditions \ref{assumption distant elements} and \ref{assumption consistency rn} (keeping the \orange{shift} of index in mind).
Given $ n \ge 2 $, we will take $ r_n = ( n ) $ and $ s_n = ( 1, \ldots, 1 ) $.
These compositions are indeed at distance $ n - 1 $ from each other: $ s_n $ can be obtained from $ r_n $ by $ n - 1 $ up-down steps that create a new column while removing a box from the largest column each time, and this is the fastest way to obtain $ s_n $ from $ r_n $ since each up-down step can increase the length of a composition by at most $ 1 $.
To see that $ s_n $ is in the support of the stationary measure, it suffices to show that it can be obtained from the composition $ \zeroVertex = (1) $ through up-steps (see \eqref{defn stationary measures}).
It should be clear that this is the case.
To verify \ref{assumption consistency rn}, we simply observe that a down-step from $ r_n = ( n ) $ will always lead to $ r_{n-1} = ( n-1 ) $.

The extra hypotheses of Proposition \ref{prop convergence of sep dist} are also satisfied.
Therefore, all of the results in \cref{section discrete framework} apply to these chains.
We highlight in particular
	the spectral decomposition in Proposition \ref{prop eigenbasis},
	the asymptotic descriptions in Propositions \ref{prop density estimate} and \ref{prop initial convergence}, 
	and 
	the analysis of the separation distance in \cref{ssec:separation_distance_discrete},
which we believe are new.
\medskip

Moving beyond the discrete setting, we will take as the limiting space $ \mc U $, the collection of open subsets of $ ( 0, 1 ) $.
This space becomes a compact metric space when equipped with the metric obtained from applying the Hausdorff metric on the complements of sets (see \cite{gnedin1997}). 
We regard a composition as an element of $ \mc U $ by identifying $ \sigma $ with the open set
$$
	\iota( \sigma ) 	
		= 	
				\bigg( 
						0, 
						\frac{ \sigma_1 }{ | \sigma | } 
				\bigg ) 	
			\cup 	
				\bigg( 
						\frac{ \sigma_1 }{ | \sigma | }, 
						\frac{ \sigma_1 + \sigma_2 }{ | \sigma | } 
				\bigg) 	
			\cup 
				\ldots 
			\cup	
				\bigg( 
						\frac{ 
								| \sigma | - \sigma_{ \ell( \sigma ) }
							}{ 
								| \sigma | 
							}, 
						1 
				\bigg)
	.
$$
It then follows from \cite[Proposition 6.1]{krdrDiffusions} that \ref{assumption state space approximation} holds.

\orange{To go further, we need to consider the density functions associated with the down-operators}.
Using \cite[Proposition 4.1]{krdrDiffusions}, it can be shown that they are given by
\begin{equation*}
	\density_\sigma ( \tau )
		=
               \frac{1}{ \binom{|\tau|}{|\sigma|}}
    			\sum_{ 1 \le i_1 < i_2 < \ldots < i_{ \ell( \sigma ) } \le \ell( \tau ) }
    				\prod_{ r = 1 }^{ \ell( \sigma) }
    					\binom{ \tau_{ i_r } }{ \sigma_r}
		,
			\qquad
				| \tau | \ge | \sigma |
		.
\end{equation*}
To extend the density functions on $ \mc U $, we make use of \cite[Proposition 10]{gnedin1997}.
This result implies that there is a unique continuous function $ \density_\sigma^o \in C( \mc U ) $ that satisfies
$$
	\density_\sigma^o( U ) 
		=
			\frac{|\sigma|!}{ \prod_{ r = 1 }^{ \ell( \sigma ) }  \sigma_r !}
			\sum_{ 1 \le i_1 < i_2 < \ldots < i_{ \ell( \sigma ) } } 
				\prod_{ r = 1 }^{ \ell( \sigma ) } 
					x_{ i_r }^{ \sigma_r }
$$
\noindent
for open sets of the form
$$
	U 
		= 
				( 0, x_1 ) 
			\cup 
				( x_1, x_1 + x_2 ) 
			\cup 
				( x_1 + x_2, x_1 + x_2 + x_3 ) 
			\cup \ldots,
$$

\noindent
where $ \{ x_i \} $ is a sequence in $ [ 0, 1 ] $ summing to 1.
Assumption \ref{assumption continuous density functions are dense} now follows from \cite[Proposition 6.4]{krdrDiffusions} and Assumption \ref{assumption continuous density functions are limits} can be established using \cite[Propositions 4.1, 6.3, and 6.5]{krdrDiffusions}.
The additional hypothesis needed for Proposition \ref{prop:path_continuity} can also be seen to hold (an up-down step will move a composition of $ n $ by at most $\frac{1}{n} $ in the $\mc U $ metric).
Therefore, all of the results in \cref{section convergence} apply.
Here, we recover the main results of \cite{krdrDiffusions} and obtain some new results.
These include
	the diagonal descriptions in Proposition \ref{prop limiting semigroup and generator},
	the intertwining result in Proposition \ref{prop algebraic identities in limit},
	the uniqueness of the stationary distribution (Proposition \ref{prop:Stationary Measure of Limit Process}),
	the asymptotic description in Proposition \ref{prop densities of limit process},
	and
	the description of the separation distance in Theorem \ref{theorem discrete and continuous sep dist} and Proposition \ref{prop convergence of sep dist}.

\subsection{Up-down chains arising from Kerov's operators}

The \orange{paper} \cite{petrovSL2} does not focus on a specific example of up-down chains but is concerned with a certain class of up-down chains that follow a general construction.
The starting point for this construction is a branching graph, a graph $ \mbb G = \bigsqcup_{ n = 0 }^\infty \mbb G_n $ in which
\begin{itemize}
	
	\item
	each $ \mbb G_n $ is finite,
	
	\item 
	$ \mbb G_0 $ is a single vertex, denoted by $ \zeroVertex $,
	
	\item
	each edge is directed and moves from some $ \mbb G_n $ to $ \mbb G_{ n + 1 } $,
	
	\item
	each vertex has an outgoing edge, 
	and 
	
	\item
	each vertex has an incoming edge (except for $ \zeroVertex $).
\end{itemize}

\noindent
The notation $ \mu \nearrow \lambda $ or $ \lambda \searrow \mu $ is used to denote that there is an edge from $ \mu $ to $ \lambda $.

The graph $ \mbb G $ is to be equipped with an edge multiplicity function $ \kappa $ that assigns a positive number $ \kappa( \mu, \lambda ) $ to the edge from $ \mu $ to $ \lambda $.
This function induces a weight on each path in $ \mbb G $, given by the product of the multiplicities of its edges.
These weights then lead to the relative combinatorial dimension $ \dim( \mu, \lambda ) $, defined as the sum of the weights of all of the paths from $ \mu $ to $ \lambda $. %
The special case $ \mu = \zeroVertex $ results in the combinatorial dimension $ \dim \lambda \ldef \dim( \zeroVertex, \lambda ) $.
These quantities are used to define down-kernels as follows:
$$
	\downKernel_n( \lambda, \mu )
		=
			\frac{ \dim \mu \cdot \kappa( \mu, \lambda ) }{ \dim \lambda }
		,
			\qquad
			\lambda \in \mbb G_n,
			\,
			\mu \in \mbb G_{ n - 1 }
		.
$$

To continue with the construction, the following assumptions \orange{(among others)} are required:
\begin{itemize}
	\item
      the vertices of $ \mbb G $ can be identified with the lattice of finite order ideals of a poset $ P $, under which $ \lambda \in \mbb G_n $ corresponds to an ideal with $ n $ elements and $ \mu \nearrow \lambda $ \orange{means} that $ \mu \subset \lambda $ 
and $|mu|=|\lambda|+1$
      (see \cite[Section 2.3]{petrovSL2} for details),
	
	\item For $\la$, $\mu$ in $\mbb G_n$ such that $\tau\ldef  \la \cap \mu$ has size $n-1$ (intersections,
      unions and set differences should be understood after identification of the elements with ideals of $P$), we have
      \begin{equation}
		\orange{\kappa( \la, \rho )
		\kappa( \mu, \rho )
			=
				\kappa( \tau, \lambda)
                \kappa( \tau, \mu ),}
        \label{eq:identity_kappas}
      \end{equation}
    where $\rho=\la \cup \mu$ (see \cite[Eq.~(12)]{petrovSL2}).
\end{itemize}

\noindent
We note here that the second %
condition is related to a certain algebraic structure called Kerov's operators (see \cite[Proposition 4]{petrovSL2}).
This structure is fundamental in \cite{petrovSL2} but is not important for our purposes.
Given this, the up-kernels considered by Petrov are of the form\footnote{\orange{For $\upKernel_{n}$
to be a transition kernel (i.e.~for its row sums to be equal to $1$), we need additional assumption on
$\gamma$ and $Q$; see \cite[Proposition 3.9]{petrovSL2}.}}
\begin{equation}
	\label{Kerov up-kernels}
	\upKernel_{n}( \lambda, \nu )
		=
			\frac{ \dim \nu \cdot \kappa( \lambda, \nu ) }{ \dim \lambda }	
			\frac{ Q( \nu \setminus \lambda ) }{ ( n + 1 ) ( n + \gamma ) }
		,
			\qquad
			\lambda \in \mbb G_n,
			\,
			\nu \in \mbb G_{ n + 1 }
		,
\end{equation}
for some nonnegative parameter $\gamma$ and some function $Q$ on the poset $P$;
see \cite[Definition 9 and Eq.~(18)]{petrovSL2}.

\begin{prop}
\label{prop:KerovOper_Commutation}
  \orange{The above matrices $(\upKernel_n)_{n \ge 1}$ and $(\downKernel_{n})_{n \ge 2}$ satisfy Assumption \ref{assumption:commutation}
  with parameter $\beta_n=\frac{ n ( n - 1 + \gamma )}{ ( n + 1 ) ( n + \gamma ) }$ (for $n \ge 2$).}
\end{prop}
\begin{proof}
  The proof is similar to that of \cref{prop:BO_Commutation}, making use of the relation \eqref{eq:identity_kappas}.
\end{proof}

Noting that \ref{assumption finite state spaces} is  satisfied \orange{by assumption},
it follows that the above chains fit into our framework.
\orange{Moreover, our framework is more general: 
the above example on integer compositions, as well as the next example on trees, and our novel examples
on permutations and graphs do not fit in the framework proposed by Petrov, since the objects are not in bijection
with the order ideals of some posets (in particular, we have seen in the previous section that integer compositions 
do not form a lattice since two elements can have several joins and/or meets).
Also, the chains considered by Petrov are all reversible (as a consequence of \cite[Definition 9]{petrovSL2}),
while the composition chain above, and the permutation and graph chains of the next sections are not.}

An application of our theory recovers the complete triangular description in \cite[Proposition 7]{petrovSL2} and the description of the spectrum in \cite[Proposition 8]{petrovSL2}.
Our description of the eigenfunctions in Proposition \ref{prop eigenbasis} and the asymptotic descriptions in Propositions \ref{prop density estimate} and \ref{prop initial convergence} appear to be new.

Regarding the separation distance or the scaling limit, this is neither considered in \cite{petrovSL2} nor immediately addressed by our results.
Indeed, the verification of \ref{assumption distant elements}, \ref{assumption consistency rn}, and \ref{assumption state space approximation}--\ref{assumption continuous density functions are limits} cannot be handled in this general context.
For specific examples, however, we expect some of these conditions to be quite approachable.
For example, \ref{assumption consistency rn} is equivalent to the existence of a sequence of vertices $ \zeroVertex \nearrow \lambda_1 \nearrow \lambda_2 \nearrow \ldots $ in which each $ \lambda_i $ has exactly one incoming edge.
Moreover, a second sequence of vertices satisfying the same property, \orange{disjoint from the first one,} would then suffice to establish \ref{assumption distant elements}.

\subsection{Aldous' chain on cladograms}
\label{ssec:cladograms}

We now move from partitions and related objects to trees. More precisely,
we will considered unrooted non-plane unlabelled trees, whose internal nodes
all have degree 3. These are sometimes called (unlabelled) {\em cladograms} in the literature,
and we will use this terminology here.
The following chain is an unlabelled version of a chain introduced by Aldous~\cite{aldous2000mixing_cladograms}. 

The size is the number of leaves, i.e.~we let $(\stateSpace_n)_{n \ge 3}$ be the set of cladograms with $n$ leaves.
Our up step will consist in selecting uniformly at random an edge of a cladogram,
and attaching a new leaf to it. Equivalently, it $C$ and $C'$ are cladograms of size $n$ and $n+1$ respectively, then
\[\upKernel_n(C,C')=\frac{1}{2n-3} e(C,C'),\]
where $e(C,C')$ is the number of edges in $C$ such that attaching a new leaf to one of those edges yields $C'$.
Similarly the downstep consists in selecting uniformly at random a leaf in a cladogram, and erasing it.
This creates an internal node of arity $2$, which is erased as well, its two incident edges being merged in a single one.
In formula, this writes as
\[\downKernel_{n+1}(C',C)=\frac{1}{n+1} \ell(C',C),\]
where $\ell(C',C)$ is the number of leaves of $C'$ whose removal yields $C$.

\orange{Examples of up- and down-transition probabilities are given on \cref{fig:Transitions_Cladograms}.}
We now check the condition \ref{assumption:commutation}.
\begin{figure}
\[\includegraphics[height=5cm]{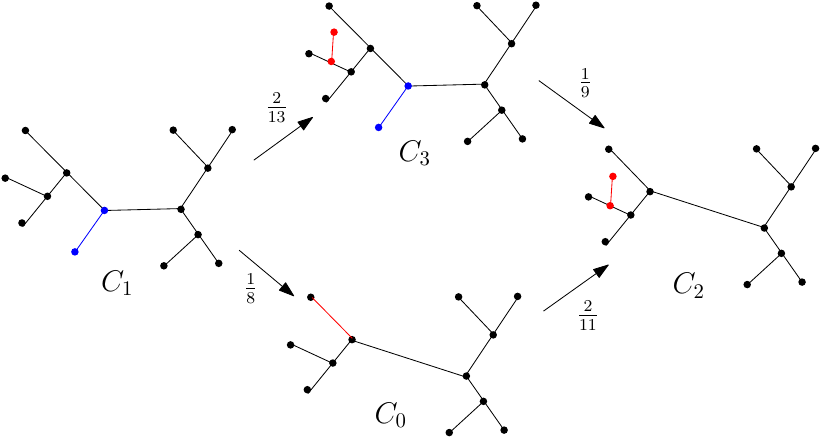}\]
\caption{Here, we show four cladograms partitions $C_0$, $C_1$, $C_2$ and $C_3$,
with the up-transition probabilities
$\upKernel_n(C_1,C_3)$ and $\upKernel_{n-1}(C_0,C_2)$, and
the-down transition probabilities $\downKernel_n(C_1,C_0)$ and $\downKernel_{n+1}(C_3,C_2)$. To help the reader, the added leaf is painted in red in $C_2$ and $C_3$,
while the leaf to be removed in $C_0$ and $C_1$ is in blue, as well as the contracted edge in $C_2$ and $C_3$. Note the multiplicity factor $2$ in the up-transition matrix coming from the fact that the edge on which we graft a new leaf has an orbit of size $2$ under the action 
of the automorphism group of the tree.
}
\label{fig:Transitions_Cladograms}
\end{figure}
\begin{prop}
 The above matrices $(\upKernel_n)_{n \ge 3}$ and $(\downKernel_{n})_{n \ge 4}$ satisfy Assumption \ref{assumption:commutation}
  with parameter $\beta_n=\frac{n(2n-5)}{(n+1)(2n-3)}$ (for $n \ge 4$).
  \end{prop}
  \begin{proof}
  As in \cref{prop:BO_Commutation},
  it suffices to check that for cladograms $C_1 \ne C_2$, both of size $n$, we have
  \begin{equation}\label{eq:comm_off_cladograms}
    ( \upKernel_n
	\downKernel_{ n + 1 })(C_1,C_2)
		= \frac{n(2n-5)}{(n+1)(2n-3)}  (\downKernel_n                                                                                              
                    \upKernel_{ n - 1 })
    				 (\la,\mu).
                  \end{equation}
\orange{To avoid dealing with symmetries, we consider here cladograms with labels on the leaves.
Let $C_1^\ell$ and $C_2^\ell$ be labeled versions of $C_1$ and $C_2$ with labels from $1$ to $n+1$, such that each label is used at most once, and $n+1$ is not used in 
$C_1^\ell$. We will also denote $\upKernelell_{n-1}$, $\upKernelell_n$, $\downKernelell_n$ and $\downKernelell_{n+1}$ to be labeled liftings of the up and down transition matrices, where
for $\upKernelell_{n-1}$ and $\upKernelell_n$, the new added leaf has label $n+1$.
We assume that $C_2^\ell$ is different from $C_1^\ell$ can be obtained from $C_1^\ell$ by erasing a leaf $\ell_0$ and grafting a leaf labelled $n+1$ on some edge $e_0$.
We then let $C^\ell_0$ be the cladogram of size $n-1$ obtained 
from $C^\ell_1$ by erasing $\ell_0$,
and $C^\ell_3$ the cladogram of size $n+1$ obtained from $C^\ell_1$ by grafting a new leaf on $e_0$.
Then we have 
\begin{align*}
( \upKernelell_n                                                                  
    \downKernelell_{ n + 1 })(C^\ell_1,C^\ell_2)&= \upKernelell_n(C^\ell_1,C^\ell_3) \downKernelell_{ n + 1 }(C^\ell_3,C^\ell_2) 
    = \frac{1 }{(2n-3)(n+1)};\\
             (\downKernelell_n                                                                                              
                    \upKernelell_{ n - 1 }) (C^\ell_1,C^\ell_2)&= \downKernelell_n(C^\ell_1,C^\ell_0)  \upKernelell_{ n - 1 }(C^\ell_0,C^\ell_2)
                    =\frac{1 }{n(2n-5)}.
      \end{align*}
Fixing a labeled version $C_1^\ell$ of $C_1$ and summing over all possible labeled versions $C^\ell_2$ of $C_2$ yields \eqref{eq:comm_off_cladograms} and conclude the proof.}
  \end{proof}
  Assumption~\ref{assumption finite state spaces} is clearly satisfied, \orange{this time with a shift of index by 3.
  Hence,} the results
of Sections \ref{ssec:discrete_triangular}--\ref{ssec:large_time_large_size} apply:
in particular, we have an explicit description of the eigenvectors of the transition operator,
which seems to be new.
More interestingly, let us discuss the existence of a scaling limit,
a question originally asked by Aldous~\cite{aldous1999diffusion}, and solved in different ways
by  Löhr--Mytnik--Winter~\cite{lohr2020Aldous_chain} and 
Forman--Pal--Rizzolo--Winkel~\cite{forman2023aldousdiffusion}.

In \cite{lohr2020Aldous_chain}, Löhr, Mytnik and Winter introduced a compact metric space $\mathbb T_2^{\cont}$ of binary algebraic measure trees, \orange{which plays the role of our limiting space $E$.
The space $\mathbb T_2^{\cont}$ does not satisfy our assumption \ref{assumption state space approximation}
since there is no canonical embedding of finite cladograms in $\mathbb T_2^{\cont}$,
but both finite cladograms and $\mathbb T_2^{\cont}$ can be embedded in a larger space (denoted $\mathbb T$ in \cite{lohr2020Aldous_chain}),
so that each element of $\mathbb T_2^{\cont}$ is a limit of finite cladograms~\cite[Proposition 2.9]{lohr2020Aldous_chain}.
This property can replace Assumption \ref{assumption state space approximation} in our proofs.}
For Assumption~\ref{assumption continuous density functions are dense}
and \ref{assumption continuous density functions are limits},
we note that for a cladogram $C$ of size $k$ and a cladogram $C'$ of size $n$,
$d_C(C')$ is the probability that $k$ random leaves of $C$ induce the cladogram $C'$
in the sense of \cite[Definition 2.5]{lohr2020Aldous_chain}.
We therefore set, for $(T,c,\mu) \in \mathbb T_2^{\cont}$
\begin{equation}\label{eq:dco}
d_C^o\big((T,c,\mu) \big)= \mathbb P\big(\mathfrak s_{(T,c)}(x_1,\dots,x_k) = C\big),
\end{equation}
where $x_1,\dots,x_k$ are i.i.d.~with measure $\mu$.
Then Assumption \ref{assumption continuous density functions are dense}
 follows from \cite[Lemma 2.12]{lohr2020Aldous_chain}.
 Let us check \ref{assumption continuous density functions are limits}.
  For $(T,c,\mu)=\iota(C')$, we have $\mu=\frac1n \sum \delta_u$, where the sum is taken over leaves of $C'$ (see \cite[eq. (2.18)]{lohr2020Aldous_chain}).
  Thus $d_C^o(\iota(C'))$ is given by \eqref{eq:dco}, where
$x_1, \dots, x_k$ are i.i.d.~random leaves of $C'$.
 In comparison, $d_C(C')$ is defined the same way, but we take $\{x_1, \dots, x_k\}$
 to be a uniform random subset of $k$ leaves of $C'$. The difference is bounded by the probability
 to have a repetition in an i.i.d. sample, i.e.
 \[\Big|d_C(C')-d_C^o\big(\iota(C')\big) \Big| \le \frac1n \binom{k}2.\]
Since this bound is uniform on all cladograms $C'$ of size $n$,
Assumption \ref{assumption continuous density functions are limits} is verified.

Therefore our results of Sections \ref{ssec:limiting_process}--\ref{ssec:process_large_time} apply. In particular, this proves the convergence of the chain to a limiting diffusion
in $ \mathbb T_2^{\cont}$, which is one of the main result of \cite{lohr2020Aldous_chain}\footnote{As said above, our chain is an unlabelled version of that defined by Aldous and
considered in \cite{lohr2020Aldous_chain}. But since the embedding of cladograms
in $ \mathbb T_2^{\cont}$ is invariant under relabelling, proving the convergence of the labelled chain or its unlabelled version are equivalent.}.
We also obtain diagonal and triangular descriptions of the limiting generator and 
estimates for the convergence of the limiting process to its equilibrium measure (the algebraic Brownian Continuum Random Tree),
\orange{which can be found under slightly different forms in Gambelin's thesis
\cite{gambelin:tel-05202989} (see in particular Theorem 4.4.3.2 and Corollary 4.4.4.1 there)}.
On the other hand, the differential description of the generator acting on the subtree mass vector given
in  \cite{lohr2020Aldous_chain} does not follow from our general framework.

To conclude this section, we discuss the separation distance of the discrete chain.
It follows from \cref{thm general sep dist of a limit} and the scaling limit result that,
if $\Delta_n(m)$ denotes the separation distance of the chain on cladograms of size $n$
after $m$ steps, then
\begin{equation}\label{eq:lower_bound_Deltan_clado}
\liminf \Delta_n(\lfloor c_n t \rfloor) \ge \Delta(t),
\end{equation}
where $c_n=\beta_3^{-1} \dots \beta_n^{-1}=\Theta(n^2)$,
and $\Delta$ is the separation distance of the limiting process at time $t$.
This gives a lower bound $\Theta(n^2)$ on the separation mixing time.
Unfortunately, Assumption \ref{assumption distant elements} is not satisfied in this context.
Indeed, any cladogram of size $n$ contains a path of logarithmic length and, consequently, one can go from any cladogram to another by moving at most $n-\Theta(\log(n))$ leaves.
Hence the results of \cref{ssec:separation_distance_discrete,ssec:separation_distance_continuous} do not apply and we do not have a formula for $\Delta(t)$, nor the reverse inequality in \eqref{eq:lower_bound_Deltan_clado}.
As a comparison, the analogue chain on labeled cladograms
is known to have a mixing time in total variation distance \orange{bounded below by $\Theta(n^2)$ 
and above by $\Theta(n^3)$~\cite{aldous2000mixing_cladograms},
and a relaxation time\footnote{The relaxation time is the inverse of the spectral gap of the transition matrix, the spectral gap being the difference between the two largest eigenvalues (in absolute value). The relaxation time is a lower bound for the mixing time.} of order $\Theta(n^2)$ \cite{schweinsberg2002cladograms}.}

\begin{remark}
\orange{A generalization of Aldous' chain where the role of the Brownian Continuum Random Tree
is played by stable trees has been introduced and studied by Gambelin in his Phd thesis~\cite{gambelin:tel-05202989}.}
\end{remark}

\section{{The} permutation example}
\label{section permutation example}

In this section, we apply our theory to the permutation-valued chains of the introduction.
We begin with some background on permutations and permutons.

\subsection{{Background}} %
\label{ssec:prel_permutations}

\subsubsection{Patterns, densities, and the down-steps}
\label{ssec:basics_patterns}

For $ n \ge 1 $, let $ \stateSpace_n $ be the set of permutations of $[n]= \{1,2,\ldots, n\}$.
Each $ \sigma $ in $ \stateSpace_n $ can be written in one-line notation as $ \sigma(1) \sigma(2) \dots \sigma(n)$.
It can also be represented in the Cartesian plane by its \emph{diagram}: the set $ \{ ( i, \sigma(i) ) : i \in [n] \} $ or any variant of it obtained by applying a map that preserves the order of $ x $-coordinates and $ y $-coordinates. 
An example is given in Figure~\ref{fig diagram and pattern}.

In such a diagram, every subset of indices $ I \subset [n] $ yields a subdiagram $ \{ ( i, \sigma(i) ) : i \in I \} $.
This set is the diagram of a unique permutation that is denoted by $\pat_I(\sigma) $.
Such a permutation is called a \emph{pattern} since it is obtained from a larger permutation.
Restricting the one-line notation of $ \sigma $ to the index set $ I $, we obtain an \emph{occurrence} of the pattern $\pat_I(\sigma)$ in $\sigma$. 
For example, if $\sigma=65831247$ and $I=\{2,5,7\}$, the one-line notation restricts to $\sigma(2) \sigma(5) \sigma(7)=514$, which is an occurrence of the pattern induced by $ I $,
$$
	\pat_{\{2,5,7\}}\left(6\mb{5}83\mb{1}2\mb{4}7\right)=312.
$$
The relevant diagrams are depicted in Figure~\ref{fig diagram and pattern}.
The case when $ I $ has size $n-1$ is of particular interest for us.
Here we will say that the pattern $\pat_{ [n] \setminus \{i\} }(\sigma)$ is obtained from $\sigma$
by removing the point $(i,\sigma(i))$.
\smallskip
\begin{figure}[ t ]
    \centering
	\includegraphics[height=35mm]{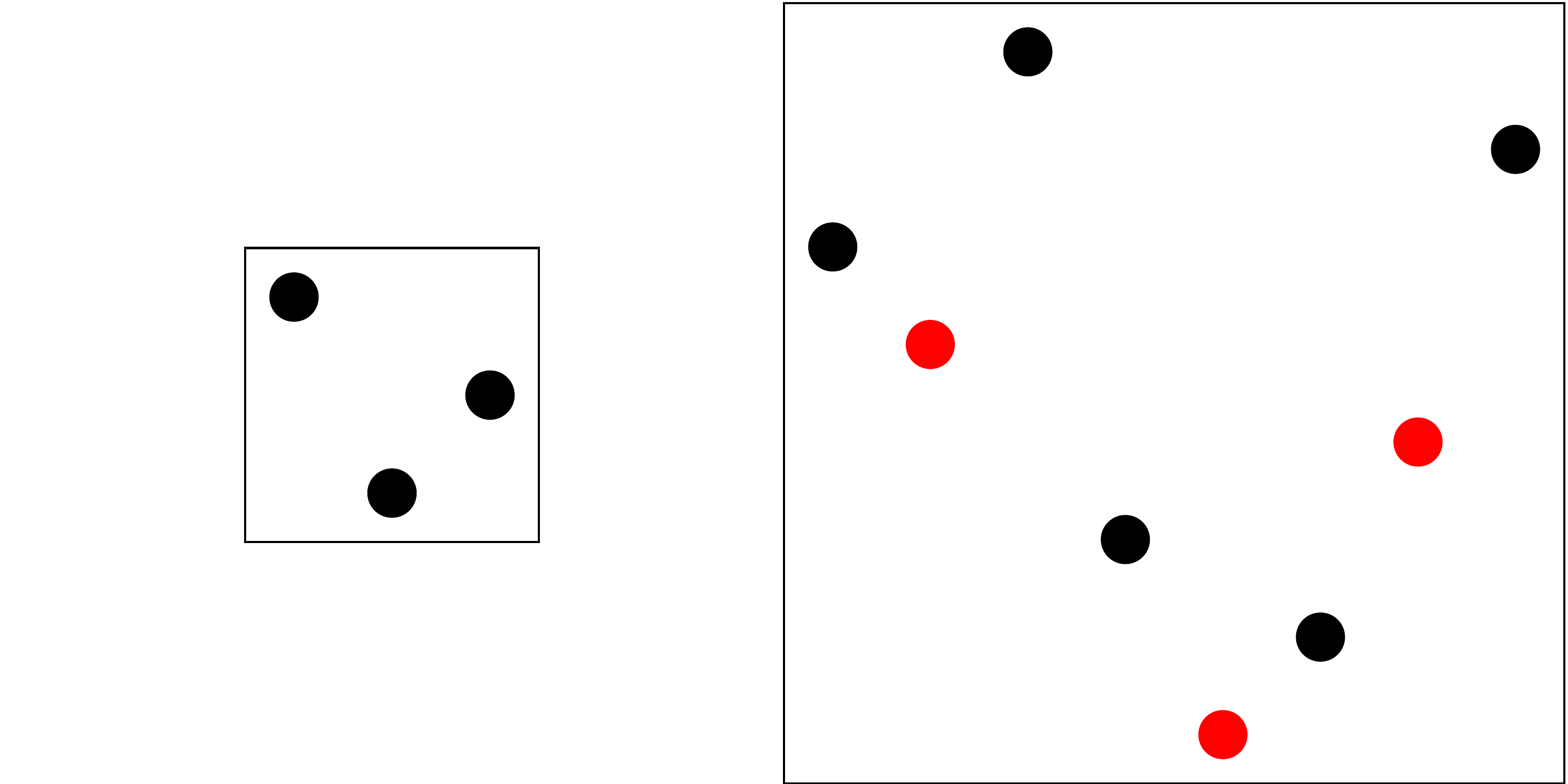}
    \captionsetup{width = 0.81 \textwidth}
    \caption{
    	\emph{Left}. 
		The diagram of the permutation $ \pi = 312 $.
		\emph{Right}. 
		The diagram of the permutation $\sigma=65831247$.
		The red points indicate the \orange{subdiagram induced} by the index set $I=\{2,5,7\}$.
		These points form the diagram of $ \pi $.
    }
    \label{fig diagram and pattern}
\end{figure}

Let $ n \ge 2 $.
Consider the pattern obtained by removing a uniformly random point \orange{from $ \sigma \in \stateSpace_n $}.
The probability that this pattern is equal to $ \pi \in \stateSpace_{ n - 1 } $ is proportional to the
number of occurrences of $\pi$ in $\sigma$, denoted by $\occ(\pi,\sigma)$.
We can express this through a kernel from $ \stateSpace_n $ to $ \stateSpace_{ n - 1 } $ given by
\[ \downKernel_{ n }(\sigma,\pi) =\frac{\occ(\pi,\sigma)}{n}.\]

\noindent
These will play the role of our down-kernels.
\orange{Using the notation $ | \cdot | $ from the general setting}, their products are given by
$$
	\downKernel_{ n, k }( \sigma, \pi )
		=
			\frac{\occ( \pi, \sigma )}{\binom{n}{k}}
		,
			\qquad
			| \sigma | = n \ge k = | \pi |
		,
$$

\noindent
from which we obtain the density functions
\begin{equation}\label{eq:def_pattern_densities}
	\density_\pi (\sigma)
		=
			\begin{cases}
                \frac{ \occ(\pi,\sigma) }{ \binom{|\sigma|}{|\pi|} },
						&
							| \sigma | \ge | \pi |,
						\\
					0,
						&
							\text{else}.
				\end{cases}
\end{equation}
\noindent
\orange{These functions are known in the literature as} \emph{pattern densities},
\orange{and notably play a central role in the theory of permutons (see next section).}

\subsubsection{Permutons}
\label{ssec:permutons}

Let $\Leb$ denote the Lebesgue measure on $[0,1]$ and $\pi_1$ and $\pi_2$ the maps from $[0,1]^2$ to $[0,1]$ that project to the first and second coordinates, respectively.
Recall that a measure $\nu$ on $A$ and a measurable map $g$ from $A$ to $B$ give rise to a push-forward measure $g_\#\nu$ on $B$ defined by $g_\#\nu(C)=\nu(g^{-1}(C))$.

By definition, a \emph{permuton} is a probability measure $\mu$ on the unit square $[0,1]^2$
whose projections to the horizontal and vertical axes are both uniform:~$(\pi_1)_\# \mu=(\pi_2)_\# \mu=\Leb$.
These objects have recently been considered to provide a limit theory for large permutations (see, e.g.,~\cite{MR2995721,bassino2020universal,grubel2022ranks}).
The heart of this theory is that the set of permutons $ \mc P $ can be viewed as a natural completion of the set of permutations $ \stateSpace $. 
Here $ \stateSpace $ is identified as a subset of $ \mc P $ by associating to $ \pi \in \stateSpace_n $ the permuton 
$\mu_{\pi}=\frac1n \sum_{i=1}^n \lambda_{ i,\pi(i) }$,
where $\lambda_{ j, k }$ is the uniform probability measure on the square $[\frac{j-1}{n}, \frac jn ] \times [\frac{k-1}{n}, \frac kn ] $.
Equivalently, $ \mu_\pi $ is the permuton with piecewise constant density 
$
	g( x, y ) = n \, \indicator( \pi(\ceiling{ nx })= \ceiling{ ny } ).
$
One then realizes $ \mc P $ as the completion of the permutons $ \{ \mu_\pi \}_{ \pi \in \stateSpace } $ by equipping it with the Wasserstein metric.
This metric induces the weak topology of measures on $ \mc P $, makes $ \mc P $ a compact space, and is also combinatorially natural.
Indeed, the convergence of permutons under this metric is equivalent to the convergence of their images under a certain family of combinatorial \orange{observables}.
These functions, which we denote by $ \{ \density^o_\pi \}_{ \pi \in \stateSpace } $, are permuton analogues of the pattern densities, and are defined as follows.
The function $\density^o_\pi $ maps a permuton $ \mu $ to the probability that $ | \pi | $ i.i.d.~points with distribution $\mu$ form the diagram of $\pi$ (see \cite[Section 2]{bassino2020universal} for details).
\medskip

In the following short proofs, we demonstrate how our analytic hypotheses \ref{assumption state space approximation}--\ref{assumption continuous density functions are limits} follow immediately from the theory of permutons.
For convenience, a sequence of permutations is said to converge to a permuton $\mu$ if the associated permutons converge to $\mu$.
\begin{proof}[Proof of \ref{assumption state space approximation}]
    Let $\mu$ be a permuton.
    Applying \cite[Lemma 2.2]{bassino2020universal}, there exists a sequence of random permutations $ \{ \sigma_n \}_{ n \ge 1 } $
    that converges a.s.~to $\mu$ and satisfies $ | \sigma_n | = n $.
    In particular, there exists at least one sequence of permutations $ \{ \tau_n \}_{n \ge 1}$ that converges to $ \mu $
    and satisfies $|\tau_n|=n$.
\end{proof}
\begin{proof}[Proof of \ref{assumption continuous density functions are dense}]
    This follows directly from \cite[Proposition 17]{ModGaussian3}, which implies that the \orange{functions $ \{ \density^o_\pi \}_{ \pi \in \stateSpace } $}
    span a dense subalgebra of $C(\mc P)$, the space of \orange{real-valued} continuous functions on $\mc P$.
\end{proof}
\begin{proof}[Proof of \ref{assumption continuous density functions are limits}]
    This follows directly from the estimate in \cite[Lemma 3.5]{MR2995721}.
\end{proof}

\subsection{The up-steps and the commutation relation}
In this section, we recall the up-steps in our permutation model and show that the associated up-down chains fit into our general framework.

It will be convenient to introduce an operation that modifies the diagram of a permutation called \emph{inflation}.
Inflating a point involves replacing it by two {new} points that are consecutive both in position and in value.
Examples are depicted in \cref{fig:ExamplesUpOperatorPermutation}.
\begin{figure}[t]
	\centering
	\includegraphics[height=35mm]{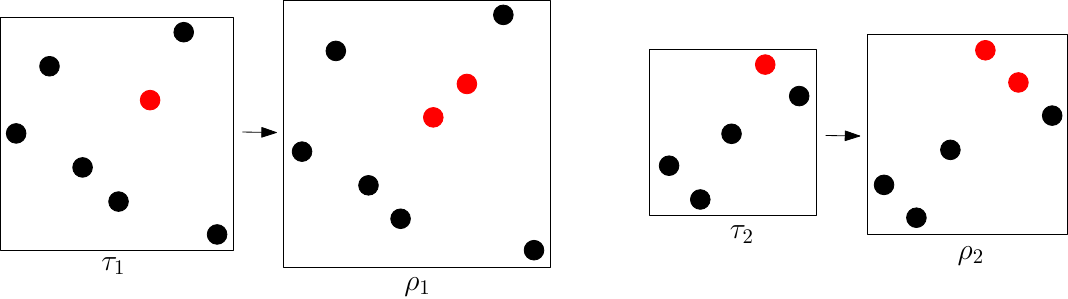}
    \captionsetup{width = 0.90 \textwidth}
    \caption{
	For $ i = 1, 2 $, inflating the red point in $ \tau_i $ into two new red points results in $ \rho_i $.
	On the left/right, the new points are in increasing/decreasing position.
	Since inflating the rightmost point in $ \tau_2 $ also results in $ \rho_2 $, a factor of 2 appears in the second equality of \eqref{eq:example_upKernel_Perm}.
    }
\label{fig:ExamplesUpOperatorPermutation}
\end{figure}

As in the introduction (Section~\ref{ssec:intro-permuton-graphon}), the up-steps that we will consider depend on a fixed parameter $p$ in $(0,1)$ and involve random inflations.
In particular, our up-step chooses a uniformly random point in the diagram of a permutation and inflates it, placing the new points in increasing position with probability $p$ and in decreasing position with probability $1-p$.
For the examples in \cref{fig:ExamplesUpOperatorPermutation}, the associated transition probabilities are
\begin{equation}\label{eq:example_upKernel_Perm}
 \upKernel_7(\tau_1,\rho_1) = \tfrac{p}{7},
\qquad
\upKernel_5(\tau_2,\rho_2)= \tfrac{2(1-p)}{5}.
\end{equation}
\medskip

It should be clear that the sets $ \{ \stateSpace_n \}_{ n \ge 1 } $ satisfy Assumption~\ref{assumption finite state spaces} with a shift of index (as in \cref{sec:previous_examples}).
\orange{In the following result, we show that our up-steps and down-steps satisfy condition \ref{assumption:commutation}}.
\begin{prop}
\label{prop:commutation_permutations}
The \orange{up- and down-steps on permutations} satisfy the following commutation relations:
	\begin{equation}
		\upKernel_n
		\downKernel_{ n + 1 }
			=
				\frac{n-1}{n+1}\,
        		\downKernel_n
        		\upKernel_{ n - 1 }
			+
				\frac{2}{n+1}\,
        		\delta_n,
			\qquad
				n \ge 2
			,
		\label{eq:commutation_permutations}
   	\end{equation}
where $ \delta_n$ denotes the identity kernel on $ \stateSpace_n $.
Consequently, \orange{the associated up-down chains satisfy} Assumptions~\AssumptionsOneFive (with a shift of index) together with the rates
$
	\generatorRates_n
		= 
			n ( n + 1 )
$	
for $n \ge 1$.
\end{prop}

\begin{proof}
    Recall from \cref{lem:avoid_diag} that it suffices to establish the relation for each off-diagonal pair.
    Let then $ \sigma $ and $ \tau $ be distinct permutations of $ [ n ] $.
    
	\orange{We first compute $( \upKernel_n \downKernel_{n + 1} )( \sigma, \tau )$, the probability that an up-down step from $ \sigma $ results in $ \tau $.
	Recall that an up-down step from $ \sigma $}
	chooses a uniformly random point in $\sigma$,
	\orange{inflates it to obtain a new permutation $\rho$}, and
	removes a uniformly random point from $\rho$.
	\orange{For this to result in $ \tau $, we cannot remove a new point in $ \rho $.}
	Therefore, we can view the point that we remove in $ \rho $ as a point in $ \sigma $ that is different from the point that we inflate.

	Let us consider the case when $ ( i, \sigma(i)) $ and $ ( j, \sigma(j)) $ are the distinct points in $ \sigma $ that the up-down step will inflate and remove, respectively.
	Here there are only two possible outcomes for the final permutation, which correspond to placing the new points in $ \rho $ in increasing and decreasing position.
	Denote these outcomes by $\sigma_{\nearrow i}^{\setminus j}$ and $\sigma_{\searrow i}^{\setminus j}$, respectively. 
	Summing over all of the equiprobable cases, we obtain
    $$
    	( \upKernel_n \downKernel_{n + 1} )( \sigma, \tau )
    		=
    			\sum_{ 1 \le i \neq j \le n }
    				\frac{1}{ n ( n + 1 ) } 
    				\left( p \indicator\big( \tau = \sigma_{\nearrow i}^{\setminus j} \big)
    				+
    				(1-p) 
					\indicator\big( \tau = \sigma_{\searrow i}^{\setminus j} \big)
					\right)
    		.
    $$
	We will compute $( \downKernel_n \upKernel_{n - 1} )( \sigma, \tau )$ similarly.
	A down-up step from $ \sigma $ involves
		removing a uniformly random point in $\sigma$
		and inflating a uniformly random point in the resulting permutation.
	The point that we inflate can immediately be viewed as a point in $ \sigma $ that is different from the point that we remove.
	Let us consider the case when $ ( i, \sigma(i)) $ and $ ( j, \sigma(j)) $ are the distinct points in $ \sigma $ that the down-up step will inflate and remove, respectively.
Observe that the final permutation is not affected by the order in which these operations occur.
In particular, its distribution is the same as it was in the above up-down case.
	We can therefore proceed as before, modifying only the probability of choosing each pair $ ( i, j ) $, to obtain
    $$
    	\orange{( \downKernel_n \upKernel_{n - 1} )( \sigma, \tau )}
    		=
    			\sum_{ 1 \le i \neq j \le n }
    				\frac{1}{ n ( n - 1 ) } 
    				\left( p \indicator\big( \tau = \sigma_{\nearrow i}^{\setminus j} \big)
    				+
    				(1-p) 
					\indicator\big( \tau = \sigma_{\searrow i}^{\setminus j} \big)
					\right)
    		.
    $$

	Comparing the two identities above, we find that the commutation relation holds on the off-diagonal.
    Applying \cref{lem:avoid_diag} then establishes \eqref{eq:commutation_permutations}, and the second claim follows immediately from \cref{prop up-down chains give intertwining}. 
	Notice that the rates
	$
	\generatorRates_n
		= 
			n ( n + 1 )
	$
	indeed satisfy $\frac{\generatorRates_{n-1}}{\generatorRates_{n}}=\beta_n=\frac{n-1}{n+1}$.
\end{proof}

\subsection{Analysis of the discrete chain}
We have shown that our permutation-valued chains fit into our general framework.
For the remainder of the section, we discuss how that framework specializes into the current setting (using the same notation as before).
This subsection focuses on the discrete chains, while the next one \orange{concerns the} scaling limit.

\subsubsection{Stationary distributions}
The first result we discuss is \cref{prop properties of stationary measures}, which identified the stationary distributions of the up-down chains.
In this setting, these are the distributions of the so-called {\em random recursive separable permutations (of parameter $ p $)} introduced and studied in our previous paper \cite{vfkrPermuton}.
Indeed, the random recursive separable permutation of size $ n $ is defined to be the permutation obtained by performing $ n -1 $ up-steps starting from the unique permutation of size 1, which coincides with the description in (\ref{defn stationary measures}) \orange{(with a shift of index)}.
We denote these stationary distributions by $ \{ \upDownDist_n \}_{ n \ge 1 } $, as in the general setting.
An explicit combinatorial formula for these distributions (involving separation trees) can be found in \cite[Proposition 1.7]{vfkrPermuton}.
We record the above observation in the following result.
\begin{prop}
	\label{prop stationary dist of perm chains}
	The unique stationary measure of $ \upDownChain_n $, denoted by $ \upDownDist_n $, %
	 is the law of the random recursive separable permutation of size $ n $ and parameter $ p $.
\end{prop}

\subsubsection{Asymptotics of pattern densities}
\label{section asymptotics of up-down pattern densities}

We now consider \cref{prop density estimate}, which describes the asymptotic behavior of various statistics of the chains.
Here, we are primarily interested in pattern densities, which play an important role.
The estimate provided by our theory is as follows:
for every pattern $\pi$, there exists a constant $B_\pi$
such that
\begin{equation*}
		\bigg|
    		\mbb E
    		\left[
    				\density_\pi
    				\big( 
    					\upDownChain_n ( m )
    				\big)
    		\right]
		-
        \upDownDist_{ |\pi| }(\pi)
		\bigg|
			\le
        		B_\pi
                \Big( 1 - \tfrac{j(j-1)}{n(n+1)} \Big)^m
             ,
             	\qquad
				n \ge | \pi |, 
				\, 
				m \ge 0
			,
              \end{equation*}
where $j \ge 2$ is the smallest size of a permutation $\rho \ne 1$ such that $\eigenfunction_\rho$ appears
in the $\eigenfunction$-expansion of $\density_\pi$.

We remark that the parameter $j$ above has a combinatorial description.
For this, we introduce some terminology.
Two elements which are consecutive both in position
and in value in a permutation form an {\em adjacency}.
Replacing these two elements by a single one leads to a smaller permutation -- this operation will be referred to as {\em shrinking the adjacency}.
This is the reverse of the \orange{inflation} operation used \orange{in the up-step}.
Finally, we introduce the {\em nonseparable core} of a permutation $\pi$, denoted by $\mathrm{ns}(\pi)$, as the permutation obtained by repeatedly shrinking all adjacencies in $\pi$ (the resulting permutation is independent of the order in which we shrink).
An example is given in Figure~\ref{fig:core}.

\orange{A} permutation is called {\em separable} if it can be obtained from the permutation $1$ by repeated inflations, or equivalently, if its nonseparable core is $1$.
\orange{Separable} permutations are standard objects in the literature\footnote{
For background on separable permutations, see, e.g., \cite{bassino2018BrownianSeparable}}, \orange{but to our knowledge, the notion of a {\em nonseparable core} is new.}
\begin{figure}
    \centering
    \includegraphics[scale=.9]{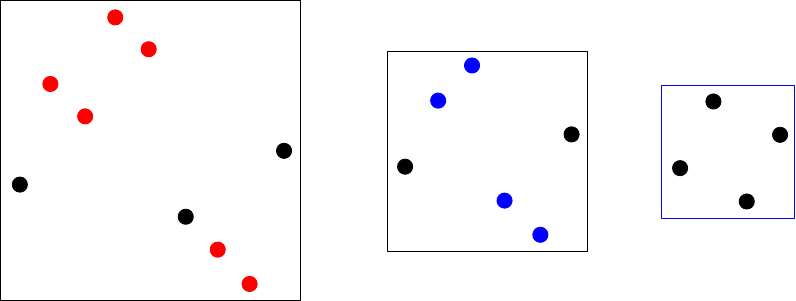}
    \captionsetup{width = 0.9 \textwidth}
    \caption{
        \emph{Left}. 
        The diagram of a permutation $\pi$. 
        Some set of disjoint adjacencies have been marked in red. 
        \emph{Middle}.
        The permutation obtained by shrinking those adjacencies. 
        Some new set of disjoint adjacencies have been marked in blue. 
        \emph{Right}.
        The permutation obtained by shrinking those adjacencies.
        This permutation does not contain any adjacency -- it is the nonseparable core of $\pi$.
    }
    \label{fig:core}
\end{figure}
\begin{lemma}
  If $\pi$ is separable, then $ \upDownDist_{ | \pi | }( \pi ) > 0 $ and $j=2$. 
  If $\pi$ is not separable, then $ \upDownDist_{ | \pi | }( \pi ) = 0 $ and $j$ is the size of the nonseparable core of $\pi$.
\end{lemma}
\begin{proof}
  The claims regarding the stationary distributions follow immediately from the description in (\ref{defn stationary measures}) and the definition of a separable permutation.
  To address the other claims, let us recall that the $\eigenfunction$-expansion of $\density_\pi$ can be found in \cref{prop expansion of density in eigenbasis}:
    $$
    	\density_\pi
    		=
        		\sum_{ | \sigma | \le | \pi | }
        			\upKernel_{ | \sigma |, | \pi | }( \sigma, \pi )
    				\,
    				\coeffDensityInEigBasis_{ | \sigma |, | \pi | }
    				\,
        			\eigenfunction_\sigma
			,
    $$
    where 
    $
    	\coeffDensityInEigBasis_{ i, k }
    		=
    			\prod_{ m = i }^{ k - 1 }
    				\frac{
    						\generatorRates_m
    					}{
    						\generatorRates_m - \generatorRates_{ i - 1 }
    					}$.
   Since $\generatorRates_m$ is positive and increasing in $m$, the coefficients $\coeffDensityInEigBasis_{ | \sigma |, | \pi | }$
   are all positive. Hence, the quantity $j$ we are looking for is the smallest size
   of a permutation $\rho \ne 1$ such that
   $\upKernel_{ | \rho |, | \pi | }( \rho, \pi )$ is nonzero.
   When $\pi$ is separable, it is possible to obtain it from either $\rho=12$ or $\rho=21$ by repeated inflations. 
   Therefore, in this case, $j=2$.

   Let us suppose that $\pi$ is nonseparable.
   Since shrinking and inflating are opposite operations, $ \pi $ can be obtained from its nonseparable core by repeated inflations.
Since $ \mathrm{ns}( \pi ) \neq 1 $, this implies that $j \le |\mathrm{ns}(\pi)|$. 
   Suppose now that $ \rho \neq 1 $ and $\upKernel_{ | \rho |, | \pi | }( \rho, \pi )$ is nonzero.
   Using again the fact that shrinking and inflating are opposite operations, we find that $\rho$ can be obtained from $\pi$ by shrinking adjacencies. 
   Shrinking any remaining adjacencies in $\rho$ must then yield $\mathrm{ns}(\pi)$, giving us that $|\rho| \ge |\mathrm{ns}(\pi)|$. 
   Since this holds for all $ \rho $, it follows that $j \ge |\mathrm{ns}(\pi)|$.
\end{proof}

\subsubsection{Separation distance}
\label{ssec:separation_permutations}

Here we specialize \cref{thm:separation_distance,prop convergence of sep dist} and prove part of \cref{prop:separation-intro-permutations}.
The remainder of the proof will be given in \cref{section sep dist of permuton diff}.

To begin, note that the identity permutation $ \sigma = 1 \cdots n $ and the reverse permutation $\tau = n \cdots 1$ are at distance $n-1$ for the up-down kernel $\upDownKernel_n$.
Consequently, Assumption \ref{assumption distant elements} holds \orange{(with a shift of index)} and \cref{thm:separation_distance} applies.
Recalling that $\generatorRates_i=i(i+1)$ \orange{and that there is a shift of index}, we obtain
  \[\sepDist_n(m)
  	= 
		\sum_{i=1}^{n-1} 
		\bigg(1-\frac{i(i+1)}{n(n+1)}\bigg)^{m} 
		\prod_{1 \le j \le n-1 \atop j \ne i} 
		\frac{j(j+1)}{j(j+1)-i(i+1)}
	,
		\qquad
		n \ge 2,
		\,
		m \ge 0
	.\] 
Using the identity $j(j+1)-i(i+1)=(j-i)(i+j+1)$, we have for $i \le n-1$
  \begin{multline*}
    \prod_{1 \le j \le n-1 \atop j \ne i} \frac{j(j+1)}{j(j+1)-i(i+1)}
    = \frac{\prod_{1 \le j \le n-1} j(j+1)}{i(i+1)} \cdot \prod_{1 \le j \le n-1 \atop j \ne i} \frac{1}{j-i} \cdot \prod_{1 \le j \le n-1 \atop j \ne i} \frac{1}{j+i+1} \\
  = \frac{(n-1)!\, n!}{i(i+1)} \cdot \frac{(-1)^{i-1}}{(i-1)!(n-1-i)!} \cdot \frac{(i+1)! (2i+1)}{(n+i)!}
  =\frac{(-1)^{i-1}\, (2i+1)\, (n-1)!\, n!}{(n-1-i)!\, (n+i)!}.\qedhere
\end{multline*}
This establishes \eqref{eq:separation-intro-permutations}.
Computing the limit
\begin{align*}
    \prod_{j = 1 \atop j \ne i}^\infty 
    	\frac{\generatorRates_{j}}{\generatorRates_{j}-\generatorRates_{i}}
			=
                \prod_{j = 1 \atop j \ne i}^\infty 
                	\frac{j ( j + 1 )}{ j ( j + 1 ) - i ( i + 1 ) }
			& =
                \lim_{ n \to \infty }
				(-1)^{i-1}\, ( 2 i + 1 )
				\frac{ (n-1)!\, n!}{(n-1-i)!\, (n+i)!}
			\\
			& =
				(-1)^{i-1}\, ( 2 i + 1 )
\end{align*}

\noindent
and applying \cref{prop convergence of sep dist} then yields the convergence to the series in \eqref{eq:separation-intro-asymptotics}.

\subsection{The limiting diffusion}

Let us turn our attention to our convergence result \cref{corol:ConvProcesses}.
This result specializes to the following theorem, which is the permutation half of \cref{thm:intro-scaling-limit-examples}.

\begin{theorem}
  \label{thm:CvChainPermuton}
  Let $\{ \si_{n,0} \}_{ n \ge 1 } $ be a sequence of (random) permutations converging to a (random) permuton $\mu_0$.
  For all $n$, let $\upDownChain_n$ be the up-down chain with transition kernel $ \upKernel_n \downKernel_{ n + 1 } $ and initial distribution $\Law( \si_{n,0} )$. 
  Then there exists a Feller process $\limitProcess$ in $ \permutonSpace $ with initial distribution $\Law(\mu_0)$ and sample paths in $D([0,\infty), \permutonSpace)$ satisfying the path convergence
	$$
		\big(
			\inclusion( \upDownChain_n ( \floor{ n^2 t } ) )
		\big)_{ t \ge 0 }
			\Longrightarrow
        		\big(
        				\limitProcess( t )
        		\big)_{ t \ge 0 }
			.
	$$
	
	\noindent
	Here $ \inclusion $ denotes the inclusion from the set of permutations to the space of permutons.
\end{theorem}

\begin{remark}
Using the Wasserstein metric on the space of permutons, it can be shown that removing or inflating a point in a permutation $\pi$ moves the associated permuton $\mu_\pi$ by at most $\frac{2}{n}$. 
Hence, the assumption of \cref{prop:path_continuity} is satisfied and $\limitProcess$ is almost surely continuous.
\end{remark}

For the remainder of \cref{section permutation example}, we \orange{study the limiting diffusion $F$}.
\subsubsection{Stationary distribution} 

We proceed by considering \cref{prop:Stationary Measure of Limit Process}, which showed that the limiting process of the up-down chains is ergodic and describes its unique stationary distribution in various forms.
In particular, the description in \cref{prop:Stationary Measure of Limit Process}\ref{claim limit of stationary distributions} tells us that this distribution is the weak limit of the stationary distributions of the up-down chains (or rather of their push-forwards on the space of permutons).
Using the description in \cref{prop stationary dist of perm chains}, it follows that this limit was identified in \cite{vfkrPermuton} as the {\em recursive separable permuton}, denoted by $\murec_p$.
Therefore, we have the following result.
\begin{prop}
	The unique stationary measure of $ F $ is the law of $\murec_p$.
\end{prop}

\subsubsection{Asymptotics of pattern densities}
\label{ssec:asymp_pattern_densities_diffusion}
We now specialize \cref{prop densities of limit process}, which describes the asymptotic behavior of various statistics of the limiting process.
Again, we are primarily interested in pattern densities.
Following the arguments in \cref{section asymptotics of up-down pattern densities}, we find that our estimate takes two forms, depending on whether or not the pattern is separable.
\begin{prop}
    \label{prop:asymptotic_densities_separable}

	Let $\pi$ be a pattern.
	Then as $ t \to \infty $, we have the estimate
	$$
		\mbb E[\density^o_\pi( \limitProcess(t))]
			=
				\begin{cases}
					\mbb E[\density^o_\pi(  \murec_p ) ] + O(e^{-2t}),
						&	
							\qquad
							\pi \text{ is separable},
						\\
        			O\left( 
        					e^{ - t \, j(j-1) } 
        			\right),
						&	
							\qquad
							\pi \text{ is nonseparable}
						,
				\end{cases}
	$$

	\noindent
	where $j$ is the size of the nonseparable core of $\pi$.
\end{prop}

\subsubsection{Separation distance}
\label{section sep dist of permuton diff}
In this section, we specialize \cref{prop convergence of sep dist,theorem discrete and continuous sep dist} and prove the remainder of \cref{prop:separation-intro-permutations}.

To begin, recall from \cref{ssec:separation_permutations} that \cref{prop convergence of sep dist} applies in this setting.
The hypotheses of \cref{theorem discrete and continuous sep dist} are verified by taking the permutations $ r_n = 1 \cdots n $, which are at distance $ n - 1 $ from the permutations $ s_n = 1 \cdots n $, and clearly satisfy $p^{\downarrow}_{n,n-1}(r_n,r_{n-1})=1$ for $n \ge \orange{2}$. 
\orange{We can therefore identify the limit
\[
\sepDist_\limitProcess(t) 
= \lim_{n \to +\infty}\sepDist_n^*( t )
= \lim_{n \to +\infty}  \sepDist_n(\floor{\generatorRates_n t})
,
	\qquad t > 0
.
\]}

Most of \orange{the properties of $ \sepDist_\limitProcess $ follow from} this expression and some \orange{properties of the Dedekind eta function.}
\orange{Indeed}, the product representation \orange{in item~\ref{item product form of sep dist}} is an immediate consequence of an identity due to Jacobi (see e.g.~(1.1) in \cite{leininger99}).
\orange{Similarly}, the \orange{symmetry in item~\ref{item:symmetry_DeltaF} can be obtained from the following well-known identity} (see, e.g.~\cite[Chapter 10 Proposition 1.9]{etaReference})\orange{:}
$$
	\eta( z )
	\sqrt{ z / i }
		=
			\eta( -1/z)
		,
			\qquad
			\text{Im } z > 0
		.
$$
\orange{The} asymptotic description \orange{in item~\ref{item:asymp_DeltaF_Infty}} is simply a particular case of the second claim in \cref{prop convergence of sep dist}.
\orange{Applying the symmetry identity then establishes the} asymptotic description \orange{in item~\ref{item:asymp_DeltaF_Zero}}.

\orange{Moving on to item~\ref{item:DeltaF_CInfty}, the regularity on $(0,\infty)$} is immediate
since the series converges absolutely uniformly on $(t_0,\infty)$ \orange{for} any $t_0>0$.
It \orange{only remains then} to verify \orange{the behavior} at $0$, that $ 
    	\frac{d^k}{dt^k} 
		\sepDist_\limitProcess( t ) \to 0 
	$ as $ t \to 0 $ for $ k \ge 1 $.
For this, we set 
$
	v( t )
		=
			\prod_{ j = 1 }^\infty
				( 1 - e^{ - 2 j t } )
$
and first analyze the behavior of
$$
	V( t )
		=
			\ln v( t )
		=
			\sum_{ j = 1 }^\infty
				\ln( 1 - e^{ - 2 j t } )
		,
			\qquad
			t > 0
		.
$$

\noindent
We will make frequent use of the family of polylogarithm functions
\begin{equation}
	\label{series form of polylog}
	\polylog_k( z )
		=
			\sum_{ j = 1 }^\infty
				\frac{ z^j }{ j^k }
		,
			\qquad
			| z | < 1,
			\,
			k = 1, 0, -1, \ldots
		,
\end{equation}

\noindent
which satisfy the recursion
\begin{equation}
	\label{recursion for polylog}
	\polylog_1( z )
		=
			-\ln ( 1 - z )
		,
			\qquad
	\polylog_{ k }( z )
		=
			z \polylog_{ k + 1 }'( z )
		,
			\qquad
			k \le 0
		,
\end{equation}

\noindent
and admit the representations
\begin{equation}
	\label{rational form of polylog}
	\polylog_k( z )
		=
			\frac{
				z P_k( z )
			}{
				( 1 - z )^{ 1 - k }
			}
		,
			\qquad
			k \le 0
		,
\end{equation}

\noindent
for some polynomials $ P_k $.

\begin{prop}
\label{prop results for V}
The following statements hold:
\begin{enumerate}[ label = (\roman*) ]
	\item
	\label{claim series for V derivatives}
    $
    	V^{(m)}( t )
    		=
    			(-1)^{ m + 1 }
    			2^m
    			\sum_{ j = 1 }^\infty
    				j^m
    				\polylog_{(1-m)} ( e^{ - 2 j t } )
    $
    for $ t > 0 $ and $ m \ge 0 $,

	\item
	\label{claim uniform convergence of series}
	the above series converge uniformly on closed subsets of $ ( 0, \infty ) $,
	and
	
	\item
	\label{claim asymptotics of derivatives}
	$ V^{(m)}( t ) = O( t^{ - 2 m - 1 } ) $ as $ t \to 0^+ $ for $ m \ge 1 $.
\end{enumerate}
\end{prop}

\begin{proof}

We first address \ref{claim uniform convergence of series} for $ m \ge 1 $.
Let $ m \ge 1 $ and $ C $ be a closed subset of $ ( 0, \infty ) $.
Since $ \inf C > 0 $, the points $ \{ e^{ - 2 j t } \}_{ j \ge 1, t \in C } $ are bounded away from $ z = 1 $.
The form in (\ref{rational form of polylog}) then implies that
$
	| \polylog_{(1-m)}( e^{ - 2 j t } ) |
		\le
			B e^{ - 2 j t }
$
for all $ j \ge 1 $, $ t \in C $, and some constant $ B $ that does not depend on $ j $ or $ t $.
Therefore, we have the following tail bound:
$$
	\left|
	\sum_{ j = n }^\infty
		j^m
		\polylog_{(1-m)} ( e^{ - 2 j t } )
	\right|
		\le
        	B
			\sum_{ j = n }^\infty
        		j^m
        		e^{ - 2 j t }
		,
			\qquad
			n \ge 1,
			\,
			t \in C
		.
$$

\noindent
Observe now that the second sum is exactly the tail of the series for $ \polylog_{ -m}( e^{ - 2 t } ) $ (see (\ref{series form of polylog})), and since the series in (\ref{series form of polylog}) converges uniformly on compact subsets of the unit disk, the above sums converge uniformly to zero on $ C $.

Applying the recursion in (\ref{recursion for polylog}), we see that the series appearing in \ref{claim series for V derivatives} can be obtained from each other by differentiating term-by-term.
Together with the fact that the $ m = 0 $ series is already known to converge to $ V^{ (0) } = V $, this implies that these series do represent the derivatives of $ V $ and that the series for $ V $ is uniformly convergent on closed subsets of $ ( 0, \infty ) $.
This establishes \ref{claim series for V derivatives} and the $ m = 0 $ case in \ref{claim uniform convergence of series}.

Turning our attention now to \ref{claim asymptotics of derivatives}, we fix $ t > 0 $.
Since the points $ \{ e^{ - 2 j t } \}_{ j \ge 1} $ lie in $ [ 0, 1 ] $, the form in (\ref{rational form of polylog}) implies that
$
	| \polylog_{(1-m)}( e^{ - 2 j t } ) |
		\le
			B_m 
			e^{ - 2 j t } 
			( 1 - e^{ - 2 t } )^{ -m }
$
for all $ j, m \ge 1 $ and some constant $ B_m $ depending only on $ m $.
Combining this with the representations in \ref{claim series for V derivatives} and (\ref{series form of polylog}), we obtain the following bounds:
\begin{align*}
	\left|
    	V^{(m)}( t )
	\right|
		& \le
			\frac{ 
    			2^m
            	B_m
			}{
				( 1 - e^{ - 2 t } )^{ m }
			}
			\sum_{ j = 1 }^\infty
        		j^m
				e^{ - 2 j t } 
		\\
		& =
			\frac{ 
    			2^m
            	B_m
			}{
				( 1 - e^{ - 2 t } )^{ m }
			}
			\polylog_{ -m }( e^{ - 2 t } )
		\\
		& \le
			\frac{ 
    			2^m
            	B_m
			}{
				( 1 - e^{ - 2 t } )^{ m }
			}
			\frac{ 
    			B_{ m + 1 } 
    			e^{ - 2 t } 
			}{
    			( 1 - e^{ - 2 t } )^{ m + 1 }
			}
		\\
		& \le
			\frac{ 
    			2^m
            	B_m
    			B_{ m + 1 } 
			}{
				( 1 - e^{ - 2 t } )^{ 2 m + 1 }
			}
\end{align*}

\noindent
for $ t > 0 $ and $ m \ge 1 $.
Observing that $ e^{ - 2 t } \le 1 - t $ as $ t \to 0^+ $ establishes \ref{claim asymptotics of derivatives}.
\end{proof}

We can now describe the behavior of $ v $ and its derivatives.
This result, together with the relation 
$ 
	\sepDist_\limitProcess( t )
		=
			1 - v( t )^3
		,
$
establishes the remaining limits in \cref{prop:separation-intro-permutations}.

\begin{prop}
For $ m \ge 0 $, we have that
$
	\lim_{ t \to 0^+ } 
		v^{(m)}( t )
			=
				0
			.
$
\end{prop}

\begin{proof}

	Starting from the identity $ v' = v V' $, we can establish inductively that the $ v^{(m)} $ have the form 
	$$
		v^{(m)} 
			= 
				v R_m( V', V'', \ldots, V^{(m)} )
			,
				\qquad
				m \ge 1
	$$
	
	\noindent
	for some polynomials $ R_m $.
	Estimating these polynomial terms using Proposition \ref{prop results for V}\ref{claim asymptotics of derivatives}, we have that for each $ m \ge 0 $, there is some $ n_m \ge 1 $ such that
	$$
		v^{(m)}( t )
			= 
				v( t ) O( t^{-n_m } )
			\quad
				\text{ as } t \to 0^+
			.
	$$

	\noindent
	Observing that for any $ n \ge 1 $, we have the convergence
	\begin{align*}
		\left| 
			\frac{ v( t ) }{ t^n }
		\right|
			=
    			\prod_{ j = n + 1 }^\infty
    				( 1 - e^{ - 2 j t } )
    			\prod_{ j = 1 }^n
        			\frac{ 
        				1 - e^{ - 2 j t }
    				}{ t }
			\le
				( 1 - e^{ - 2 t ( n + 1 ) } )
    			\prod_{ j = 1 }^n
        			\frac{ 
        				1 - e^{ - 2 j t }
    				}{ t }
				\xrightarrow[ t \to 0^+ ]{}
    			0
	\end{align*}
	
	\noindent
	concludes the proof.
\end{proof}

\subsubsection{A semi-discrete approximation of $F$}
\label{ssec:semi-discrete-permutations}
In this section, we construct another family of Markov chains that converges to $ \limitProcess $.
Unlike the up-down chains though, these chains will be {\em semi-discrete} -- their state space will be the continuous space of permutons.
\orange{Given} $\eps>0$ and $s$ in $[0,1]$, we define the following function \orange{on $[0,1]$:}
\[ \varphi^{s,\eps} (x)= \begin{cases} (1-\eps)x, & x \le s,\\
(1-\eps) x +\eps, & x > s.\end{cases}\]
In particular, if $(x_0,y_0)$ is a point in $[0,1]^2$,
the pair $(\varphi^{x_0,\eps}, \varphi^{y_0,\eps})$ defines a function from $[0,1]^2$ to itself
(acting componentwise).
We let
$\delta^{(\eps),\nearrow}_{(x_0,y_0)}$ (resp.~$\delta^{(\eps),\searrow}_{(x_0,y_0)}$)
be the uniform measure of mass $1$ on the increasing (resp. decreasing) diagonal
of the square $$[(1-\eps)x_0,(1-\eps) x_0 +\eps] \times [(1-\eps)y_0,(1-\eps) y_0 +\eps].$$
The increasing (resp.~decreasing) inflation of a measure $\mu$ at $(x_0,y_0)$ is then
\[\textstyle \inf^{\eps,\bullet}_{(x_0,y_0)} (\mu) \ldef  
(1-\eps) (\varphi^{x_0,\eps}, \varphi^{y_0,\eps})_{\#} (\mu)
+ \eps\, \delta^{(\eps),\bullet}_{(x_0,y_0)},\]
where $\bullet$ is $\nearrow$ (resp.~$\searrow$) for an increasing (resp.~decreasing) inflation.
This definition is illustrated \orange{in} Figure~\ref{fig:EpsilonInflation}.
Finally, we define
\[\textrm{Inf}^{\, \eps}(\mu) = \textrm{inf}^{\, \eps,B}_{(X_0,Y_0)}(\mu),\]
where $(X_0,Y_0)$ has distribution $\mu$, and $B$ is independent from $(X_0,Y_0)$,
and takes value $\nearrow$ with probability $p$ and $\searrow$ with probability $1-p$.
We note that $\textrm{Inf}^{\, \eps}(\mu)$ is a random measure, whose distribution
is determined by $\mu$ and $p$.
\medskip
\begin{figure}
\[\includegraphics{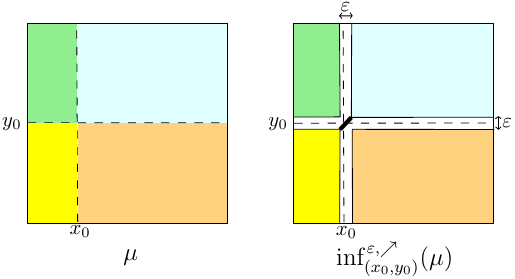}\]
\caption{Schematic representation of the map $\textrm{inf}_{(x_0,y_0)}^{\, \eps,\nearrow}$ on permutons. 
The term $ (\varphi^{x_0,\eps}, \varphi^{y_0,\eps})_{\#} (\mu)$ is obtained by splitting $\mu$ into four parts, and contracting slightly each part in a corner of the square $[0,1]^2$  (see the colored rectangles on the picture).  
The second term, here $\delta^{(\eps),\nearrow}_{(x_0,y_0)}$ is represented
by the bold line segment around $(x_0,y_0)$.}
\label{fig:EpsilonInflation}
\end{figure}

\orange{We can now} define our approximation.
Fix some (possibly random) permuton $\mu_0$,
 \orange{and let} $F_\eps$ be the pure-jump Feller process with initial condition $\mu_0$,
 which jumps at rate $2\eps^{-2}$
from $\mu$ to $\textrm{Inf}^{\, \eps}(\mu)$. In other terms it is the Feller process on $\mathcal P$
with domain $C(\mathcal P)$ and generator
\begin{equation}\label{eq:def_Aeps}
  \pregenerator_\eps f (\mu)= 2 \eps^{-2} \Big( \mathbb E \big[ f(\textrm{Inf}^{\, \eps}(\mu) ) \big] - f(\mu) \Big).
\end{equation}

\orange{We then have the following result, justifying that $F_\eps$ is indeed an approximation of $F$.}
\begin{prop}
  \label{prop:semi_discrete_approx}
$F_\eps$ tends to $F$ in the Skorokhod space $D([0,\infty), \mathcal P)$.
\end{prop}

We start with some notation and a lemma. For a permutation $\pi$ and $m \ge 1$, we let $\mathcal I_m(\pi)$ be the
set of pairs $(\tau,i)$ such that replacing the point $(i,\tau(i))$ by an increasing run of length $m$ (i.e.~$m$ points
which are consecutive both in values and in positions, put in increasing order) yield $\pi$.
We define $\mathcal D_m(\pi)$ similarly using decreasing runs.
\begin{lemma}
  For any permuton $\mu$, pattern $\pi$ of size $k \ge 1$, and $\eps$ in $[0,1]$, we have
  \begin{multline}\label{eq:dpi_InfEps}
  \mathbb E\Big[\density_\pi^o \big( \textrm{Inf}^{\, \eps}(\mu) \big) \Big]
  = (1-\eps)^{k} \density_\pi^o(\mu)
  + p \sum_{m \ge 1} \binom{k}{m} \frac{(1-\eps)^{k-m} \eps^m}{k-m+1} \left( \sum_{(\tau,i) \in \mathcal I_m(\pi)} \density_\tau^o(\mu)  \right)\\
  + (1-p) \sum_{m \ge 1} \binom{k}{m} \frac{(1-\eps)^{k-m} \eps^m}{k-m+1} \left( \sum_{(\tau,i) \in \orange{\mathcal D}_m(\pi)} \density_\tau^o(\mu)  \right).
\end{multline}
  \label{lem:dpi_InfEps}
\end{lemma}
\begin{proof}
  Let $(X_0,Y_0)$ and $B$ be the random variable involved in the construction of $\textrm{Inf}^{\, \eps}(\mu)$.
  Conditionally on these random variables, to evaluate $\density_\pi^o \big( \textrm{Inf}^{\, \eps}(\mu) \big) $, we further take 
  $(X_1,Y_1)$, \dots, $(X_k,Y_k)$ to be i.i.d.~random variables with distribution $\textrm{Inf}^{\, \eps}(\mu)$.
  By definition each $(X_i,Y_i)$ is either distributed with law $(\varphi^{X_0,\eps}, \varphi^{Y_0,\eps})_{\#} (\mu)$ 
  with probability $1-\eps$ or with law $\delta^{(\eps),B}_{(X_0,Y_0)}$ with probability $\eps$.
  We let $M$ be the number of variables $(X_1,Y_1)$, \dots, $(X_k,Y_k)$ having law $\delta^{(\eps),B}_{(X_0,Y_0)}$.
  Clearly, $M$ is a random variable with law $\mathrm{Binomial}(k,\eps)$. We will work conditionally on $M$ and $B$.

  The case $M=0$ occurs with probability $(1-\eps)^k$. In this case $(X_1,Y_1)$, \dots, $(X_k,Y_k)$ are i.i.d. variables
  with law $(\varphi^{X_0,\eps}, \varphi^{Y_0,\eps})_{\#} (\mu)$. Since $\varphi^{X_0,\eps}$ and $\varphi^{Y_0,\eps}$ are increasing functions, the map $(\varphi^{X_0,\eps}, \varphi^{Y_0,\eps})$ does not change the relative order of the coordinates,
  and the probability that $(X_1,Y_1)$, \dots, $(X_k,Y_k)$ form the pattern $\pi$ is the same as for
  i.i.d.~random variable of law $\mu$.
  Hence the probability that $M=0$ and $(X_1,Y_1)$, \dots, $(X_k,Y_k)$ form the pattern $\pi$
  is given by $(1-\eps)^k \density_\pi^o(\mu)$.

  We now consider the case $M=m$ ($1 \le m \le k$) and $B= \nearrow$. 
  This happens with probability $p  \binom{k}{m} (1-\eps)^{k-m} \eps^m$.
  Without loss of generality, we assume that $(X_1,Y_1)$, \dots, $(X_{k-m},Y_{k-m})$ have
  law $(\varphi^{X_0,\eps}, \varphi^{Y_0,\eps})_{\#} (\mu)$ and $(X_{k-m+1},Y_{k-m+1})$, \dots, $(X_{k},Y_{k})$
  have law $\delta^{(\eps),B}_{(X_0,Y_0)}$. By construction,
  these last points are consecutive in values and in positions in the sample $(X_1,Y_1)$,\dots,$(X_k,Y_k)$,
  and are in increasing order (since $B= \nearrow$).
  Moreover they have the same place as $(X_0,Y_0)$ in the sample $(X_0,Y_0)$, $(X_1,Y_1)$, \dots, $(X_{m-k},Y_{m-k})$.
  Hence $(X_1,Y_1)$, \dots, $(X_{k},Y_{k})$ form the pattern $\pi$ if and only if,
   for some $(\tau,i)$ in $\mathcal I_m(\pi)$, the points
  $(X_0,Y_0)$, $(X_1,Y_1)$, \dots, $(X_{m-k},Y_{m-k})$ form the pattern $\tau$ and $(X_0,Y_0)$ is the
  $i$-th point from the left in this sample.
  For a given $(\tau,i)$, the latter happens with probability $\frac{\density_\tau^o(\mu)}{k-m+1}$.
  Summing over all $(\tau,i)$ in $\mathcal I_m(\pi)$, we get that the probability that $M=m$, $B=\nearrow$ and
  $(X_1,Y_1)$, \dots, $(X_{k},Y_{k})$ form the pattern $\pi$ is 
  \[ p\binom{k}{m} \frac{(1-\eps)^{k-m} \eps^m}{k-m+1} \left( \sum_{(\tau,i) \in \mathcal I_m(\pi)} 
  \density_\tau^o(\mu)\right).\]

  The case $B=\searrow$ is similar, replacing $p$ by $1-p$ and $\mathcal I_m(\pi)$ by $\mathcal D_m(\pi)$.
  Summing the various contributions, we get the formula in the lemma.
\end{proof}

\begin{proof}
  [Proof of \cref{prop:semi_discrete_approx}]
  \cref{lem:dpi_InfEps} implies that the space $\mathcal H=\linearSpan\{\density_\pi\}_{\pi \in \stateSpace}$ is invariant
  by the generator $\mathcal A_\eps$. Since it is also dense in $C(\mathcal P)$
  (Assumption \ref{assumption continuous density functions are dense}),
  \cref{prop:dense_invariant_core} tells us that $\mathcal H$ is a core for $\mathcal A_\eps$.
  On the other hand, we know from \cref{prop limiting semigroup and generator}, item \ref{claim generator and core}
  that $\mathcal H$ is a core for the generator $\mathcal A$ of $F$.

  Let $\pi$ be a pattern of size $k$. Our next goal is to prove that 
  \begin{equation}\label{eq:cv_Aeps}
    \lim_{\eps \to 0} \mathcal A_\eps \density^o_\pi = \mathcal A \density^o_\pi.
  \end{equation}
  The quantity $\mathbb E\Big[\density_\pi^o \big( \textrm{Inf}^{\, \eps}(\mu) \big) \Big]$
  appearing in $\mathcal A_\eps \density^o_\pi$ (see \eqref{eq:def_Aeps}) is computed in Lemma~\ref{lem:dpi_InfEps}.
  We expand it into powers of $\eps$.
  \begin{itemize}
    \item The constant term in \eqref{lem:dpi_InfEps} is $\density^o_\pi(\mu)$.
      This cancels with the term $-\density^o_\pi(\mu)$ in the definition of the generator (\cref{eq:def_Aeps}).
    \item Let us look at linear terms in $\eps$. The first summand in \eqref{eq:dpi_InfEps} gives $-k \eps \density^o_\pi(\mu)$.
      In the sums, only the summands corresponding to $m=1$ yield a linear term in $\eps$. 
      For $m=1$, we simply have $\mathcal I_1(\pi)=\{(\pi,i), 1\le i \le k\}$, so that
      \[ \left( \sum_{(\tau,i) \in \mathcal I_1(\pi)} 
  \density_\tau^o(\mu)\right) = k \density^o_\pi(\mu),\]
  \orange{and the same holds with $\mathcal D_1(\pi)$ instead of $\mathcal I_1(\pi)$}.
      Summing everything up, the linear terms in $\eps$ in the right-hand side of \cref{eq:dpi_InfEps} are 
      \[-k \, \eps \,\density^o_\pi(\mu) + p\, k\, \frac{\eps}{k}\, k\, \density^o_\pi(\mu) 
      + (1-p)\, k\, \frac{\eps}{k}\, k\, \density^o_\pi(\mu) = 0.\]
    \item We now consider the quadratic terms in $\eps$ in the right-hand side of \eqref{eq:dpi_InfEps}.
      The first summand and the terms $m=1$ in the sums give the following quadratic terms
      \[\binom{k}2 \eps^2 \density^o_\pi(\mu) - p k \frac{(k-1) \eps^2}{k} k\, \density^o_\pi(\mu)
      - (1-p) k \frac{(k-1) \eps^2}{k} k\, \density^o_\pi(\mu) = - \binom{k}2 \eps^2 \density^o_\pi(\mu).\]
      Finally, the terms $m=2$ in the sums also yield some quadratic term, namely:
      \[ \binom{k}2 \frac{\eps^2}{k-1} \bigg( p \sum_{(\tau,i) \in \mathcal I_2(\pi)} \density_\tau^o(\mu)
      +(1-p) \sum_{(\tau,i) \in \mathcal D_2(\pi)} \density_\tau^o(\mu)\bigg) .\]
      That $(\tau,i)$ belongs to $\mathcal I_2(\pi)$ \orange{(resp.~$\mathcal D_2(\pi)$)}exactly means that $\pi$ can be obtained from $\tau$
      by an increasing \orange{(resp.~decreasing)} inflation of $(i,\tau(i))$, hence the quantity above rewrites
      as 
      \[\binom{k}2 \eps^2 \bigg(\sum_{\tau \in \stateSpace_{k-1}} \upKernel(\tau,\pi) \, \density_\tau^o(\mu) \bigg).\]
  \end{itemize}
    We finally get that
    \[\lim_{\eps \to 0} A_\eps \density^o_\pi(\mu) = \lim_{\eps \to 0} 2 \eps^{-2} \Big( \mathbb E \big[ \density^o_\pi(\textrm{Inf}^{\, \eps}(\mu) ) \big] - \density^o_\pi(\mu) \Big) = k(k-1) \bigg(-\density^o_\pi(\mu) + \sum_{\tau \in \stateSpace_{k-1}} \upKernel(\tau,\pi) \, \density_\tau^o(\mu) \bigg).\]
    Comparing with \cref{prop algebraic identities in limit}, item \ref{claim generator on density functions},
    this proves \eqref{eq:cv_Aeps}. By linearity, $\lim_{\eps \to 0} \mathcal A_\eps f = \mathcal A f$
    for all $f$ in the common core $\mathcal H$ of $\mathcal A_\eps$ and $\mathcal A$.
    Applying \cite[Chapter 1, Theorem 6.1]{ethierKurtzBook} (with $L_n=L=C(\mathcal P)$, $\pi_n=\mathrm{id}$ and $f_n=f$),
    it holds that, for any $t>0$ and any $f$ in $C(\mathcal P)$, 
    the function $\mathcal T_\eps(t) f$ tends to $\mathcal T(t) f$, 
    where $\mathcal T_\eps$ and $\mathcal T$ are the transition
    operator semigroups associated with $F_\eps$ and $F$.
    Applying further \cite[Chapter 4, Theorem 2.5]{ethierKurtzBook}, we get that $F_\eps$ converges to $F$ in the Skorokhod
    space $D([0,\infty), \mathcal P)$, as claimed.
\end{proof}

\begin{remark}
In the construction of $\mathrm{Inf}^\eps(\mu)$, instead of inserting a scaled copy of the increasing diagonal with probability $p$ and a copy of the decreasing one with probability $1-p$,
we could have inserted a scaled copy of any (potentially random) permuton $\mu_0$
with $d(12,\mu_0)=p$ (or $\mathbb E(d(12,\mu_0))=p$ if $\mu_0$ is random). The resulting processes $F_{\eps,\mu_0}$ would still be approximations of $F$ in the sense of Proposition~\ref{prop:semi_discrete_approx}.
\end{remark}

\section{A graph example}
\label{sec:graph}

\subsection{Subgraph densities and graphons}
\label{ssec:graphons}
In this section, for $n \ge 1$, 
we let $\mathcal G_n$ be the set of unlabelled simple undirected graphs on $n$ vertices.
If $G$ is a graph with vertex set $V$ and $I$ a subset of $V$,
we let $G[I]$ be the induced subgraph of $G$ on $I$,
that is the graph with vertex set $I$ containing the edges of $G$
with both extremities in $I$.
A special case of interest is when $I=V\setminus \{v\}$ for some vertex $v$:
then $G[I]$ is simply obtained by erasing the vertex $v$ and all incident edges.

Erasing a uniform random vertex $v$ in a graph $G$ defines a kernel $\downKernel$
from $\mathcal G_n$ to $\mathcal G_{n-1}$.
With this definition, if $H$ and $G$ are graphs of size $k$ and $n$ respectively, we have
$$
	d_H(G)=\downKernel_{ n, k }(G,H)
		=
			\frac{\occ(H,G)}{\binom{n}{k}}
		,
			\qquad
			| V_G | = n \ge k = | V_H |
		,
$$
where $\occ(H,G)$ is the number of subsets $I$ of $V_G$ such that $G[I]$ is isomorphic
to $H$. The quantity \orange{$d_H(G)$} is usually referred to as the {\em induced
subgraph density} of $H$ in $G$.

These (induced) subgraph densities play a central role in the theory of graph limits,
also known as graphons.
A {\em graphon} is a symmetric function from $[0,1]^2$ to $[0,1]$.
A graph $G$ can be seen as a graphon $W_G$ by considering its rescaled adjacency matrices, i.e.~
\[W_G(x,y)=\begin{cases}1 &\text{ if $\{\lceil nx \rceil,\lceil ny \rceil\}$ is an edge of $G$;}\\
0 & \text{ otherwise.} \end{cases}\]
The function $W_G$ depends on the labeling of the vertices of $G$ by numbers from $1$
to $n$, but we usually consider graphons up to some equivalence relation.
Convergence of graphons is defined via a pseudo-metric, call box distance,
whose definition is irrelevant here.
It turns out that induced subgraph densities can be extended to the space of graphons
and that the convergence for the box metric is equivalent to convergence of all subgraph densities.
A detailed introduction to graphon theory can be found in the book of Lovász \cite{BookGraphons}.
In particular Assumptions \ref{assumption state space approximation}--\ref{assumption continuous density functions are limits} 
are well-known to hold for the space of graphons.

\subsection{Up-kernel, commutation and consequences}
As for permutations, we also define an up-kernel $\upKernel_n$ from $\mathcal G_n$ to $\mathcal G_{n+1}$,
where $\upKernel_n(G,G')$ is the probability of obtaining $G'$ starting from $G$ by 
\orange{choosing} a uniform random vertex of $G$,
\orange{creating a copy of it with the same neighbourhood}
and connecting the selected vertex and its copy with probability $1-p$.
Then we have the following commutation relation, whose proof is identical to that of \cref{prop:commutation_permutations}.
\begin{prop}
\label{prop:commutation_graphs}
The above kernels satisfy the following commutation relations: for $n \ge 2$
	\begin{equation}
		\upKernel_n
		\downKernel_{ n + 1 }
			=
				\frac{n-1}{n+1}\,
        		\downKernel_n
        		\upKernel_{ n - 1 }
			+
				\frac{2}{n+1}\,
        		\delta_n,
		\label{eq:commutation_graphs}
   	\end{equation}
where $\delta_n$ denotes the identity kernel on $\mathcal G_n$.
Consequently, the up-down transition operator $T_n$
satisfies Assumptions \ref{assumption increasing generator rates}--\ref{assumption intertwining}, with
$
	\generatorRates_n
		= 
			n ( n + 1 )
	$	(for $n \ge 1$).
\end{prop}
The results announced in the introduction (the scaling limit of the up-down chains, the continuity of the limit in \cref{thm:intro-scaling-limit-examples},
the exact formula and asymptotics estimates for the separation distance in \cref{prop:separation-intro-permutations}) follow immediately.
\begin{remark}
	\label{remark commutation with inversion graph}
  Consider the map $I_n: \Sn_n \to \mathcal G_n$ associating with a permutation $\sigma$
  its inversion graph $G(\sigma)$ (recall that $G(\sigma)$ has vertex-set $\{1,\dots,n\}$)
  and contains the edge $\{i,j\}$ if $\{i,j\}$ is an inversion of $\sigma$, i.e.~if $(i-j)(\sigma(i)-\sigma(j))<0$.
  As mentioned in the introduction, we can easily check that the up and down transition operators on graphs and permutations
  are compatible with the inversion graph map in the following sense: for $n \ge 1$,
  \[  \upKernel_{\mathcal G,n} \circ I_n = I_{n+1} \circ \upKernel_{\Sn,n}, \quad \downKernel_{\mathcal G,n+1} \circ I_{n+1} = I_n \circ \downKernel_{\Sn,n+1},\]
  where we added subscripts $\mathcal G$ and $\Sn$ to distinguish the kernels on graphs and permutations.
  Hence it is not surprising that both $(\upKernel_{\Sn},\downKernel_{\Sn})$ and $(\upKernel_{\mathcal G},\downKernel_{\mathcal G})$ satisfy
  the same commutation relation (with the same coefficients).
  However, none of \cref{prop:commutation_permutations} or \cref{prop:commutation_graphs} implies the other,
  since $I_n$ is neither surjective, nor injective. 
  \orange{Also, the graph part
  of \cref{thm:intro-scaling-limit-examples} does not imply the permutation part,
and vice-versa.}
\end{remark}

\section*{Acknowledgements}
The authors are grateful to Koléhè Coulibaly-Pasquier for insightful discussions on separation distances of Markov chains,
and Roman Gambelin for interesting discussions, for pointing out some references,
and for sharing an early version of his work~\cite{gambelin2025scaling}.

\bibliographystyle{bibli_perso}
\bibliography{bibli}

\end{document}